\documentclass[9pt]{article}

\usepackage{amssymb,amsmath,amsthm,amsfonts}
\usepackage{graphicx}
\graphicspath{{./fig/}}
\usepackage{subfig}
\usepackage{float}
\usepackage{verbatim}
\usepackage[colorlinks,linktocpage,linkcolor=blue]{hyperref}
\usepackage{algorithm,algorithmic}
\usepackage{enumerate}
\usepackage{setspace}
\usepackage{longtable}

\newtheorem{thm}{Theorem}[section]
\newtheorem{rmk}{Remark}[section]
\newtheorem{definition}{Definition}[section]
\newtheorem{lem}{Lemma}[section]

\newtheorem{prop}{Proposition}[section]

\DeclareMathOperator*{\tr}{tr}

\DeclareMathOperator*{\supp}{supp}
\newcommand{\R}{\mathbb{R}}

\newcommand{\bi}{{\char`\\ i}}
\newcommand{\dd}{\,\mathrm{d}}
\newcommand{\ds}{\dd s}
\newcommand{\dx}{\dd x}
\renewcommand{\tt}{\tilde{t}}

\DeclareMathOperator*{\argmin}{arg\,min}
\DeclareMathOperator{\Var}{Var}

\newcommand{\onehalf}{\frac{1}{2}}
\newcommand{\RS}{\R_\Sigma} 

\title{Expectation Propagation for Nonlinear Inverse Problems \\
-- with an Application to Electrical Impedance Tomography}
\date{\today}
\author{Matthias Gehre\thanks{Center for Industrial Mathematics, University of Bremen, Bremen D-28344, Germany (mgehre@math.uni-bremen.de)} \and Bangti Jin\thanks{Department of Mathematics, University of California, Riverside, University Ave. 900, Riverside, California 92521, USA (bangti.jin@gmail.com)}}
\begin{document}
\maketitle
\begin{abstract}
In this paper, we study a fast approximate inference method based on expectation propagation
for exploring the posterior probability distribution arising from the Bayesian formulation of
nonlinear inverse problems. It is capable of efficiently delivering reliable estimates of the posterior mean
and covariance, thereby providing an inverse solution together with quantified uncertainties.
Some theoretical properties of the iterative algorithm are discussed, and the efficient implementation
for an important class of problems of projection type is described. The method is illustrated
with one typical nonlinear inverse problem, electrical impedance tomography with complete electrode model, under sparsity
constraints. Numerical results for real experimental data are presented, and compared with
that by Markov chain Monte Carlo. The results indicate that the method is accurate and computationally very efficient.
\end{abstract}

\section{Introduction}

In this work, we study a statistical method, Expectation Propagation (EP), for inferring the mean and
covariance of posterior distributions arising in the Bayesian formulation of either linear or nonlinear
inverse problems. The EP method was first developed in the machine learning community \cite{Minka:2001,Minka:2001a},
to efficiently arrive at an approximation. In contrast to popular Markov chain Monte Carlo (MCMC) methods
\cite{GilksRichardsonSpiegelhalter:1996,Liu:2008}, the EP is far less expensive yet with
little compromise in the accuracy \cite{OpperWinther:2005}.

\subsection{Bayesian formulation}
We consider a (possibly nonlinear) equation
\begin{equation}\label{eqn:model}
   F(x)=b,
\end{equation}
where the map $F:\R^n\to\R^m$, and vectors $x\in\R^n$ and $b\in\R^m$ refer to
the data formation mechanism, the unknowns and the given data, respectively.
In the context of inverse problems related to differential equations, the map
$F$ often involves a discretized solution operator of the underlying governing differential equation. In
practice, we only have access to a noisy version  $b$ of the exact data $b^\dag$, i.e.,
$b=b^\dag+\zeta$, where the vector $\zeta\in\mathbb{R}^m$ represents the noise in the data.

The Bayesian approach has received increasing attention in recent years \cite{KaipioSomersalo:2005}
due to its distinct features compared with conventional deterministic inversion techniques. In this work,
we are interested in the fast exploration of the resulting Bayesian posterior distribution. we first briefly recall
the fundamentals of the Bayesian approach. There are two basic building blocks of Bayesian modeling,
i.e., the likelihood function and the prior distribution. The likelihood
function $p(b|x)$ depends on the statistics of the noise $\zeta$. The most popular choice in practice is
the Gaussian model, i.e., $\zeta$ follows a normal distribution with mean zero and
variance $\sigma^2I_m$, or $p(\zeta) = (\sqrt{2\pi\sigma^2})^{-m} e^{-\frac{1}{2\sigma^2}
\|\zeta\|_2^2}$, where $\|\cdot\|_r$, $r\geq1$, refers to the $r$-norm of
a vector. This gives a likelihood function $p(b|x)\propto e^{-\frac{1}{2\sigma^2}\|F(x)-b\|^2_2}$.
The Gaussian model is often justified by appealing to the central limit theorem. Another popular choice is the Laplace
distribution on the noise $\zeta$, i.e., $p(\zeta)=(2\sigma)^{-m}e^{-\frac{1}{\sigma}\|\zeta\|_1}$,
which gives a likelihood function $p(b|x)\propto e^{-\frac{1}{\sigma}\|F(x)-b\|_1}$.
The Laplace model is especially suited to data with a significant amount of outliers \cite{ClasonJin:2012}.

In Bayesian modeling, we encode our prior knowledge about the unknown $x$ in a prior
distribution $p(x)$, which forms the second building block of Bayesian modeling. The main role of the prior
$p(x)$ is to regularize the ill-posedness inherent to the model \eqref{eqn:model} so as to
achieve a physically meaningful solution \cite{Franklin:1970}. Hence the effect of the prior $p(x)$ is pronounced
for ill-posed and under-determined problems; its influence will vanish for a well-posed,
invertible $F$ as the noise-level tends to zero. There are many possible choices of the
prior distribution $p(x)$, depending on the desired qualitative features of the sought-for solution:
globally smooth, sparse (with respect to a certain basis), blocky and monotone etc. Typically such prior
knowledge is encoded by a Markov random field. One popular choice is the
smoothness prior, i.e., $p(x) = N(x;\mu_0,C_0)$, with mean $\mu_0$ and covariance $C_0$.

By Bayes' formula, we obtain the posterior distribution
\[
   p(x | b) \propto p(x|b) p(x)
\]
up to a constant factor depending only on the data $b$. The posterior distribution $p(x|b)$
holds the full information about the probability of any vector $x$ explaining the observational data $b$.
A very popular and straightforward approach is to search for the maximizing element $x$ of the
posterior distribution $p(x|b)$, which leads to the celebrated Tikhonov regularization.
However, it only yields information about one point estimate from an ensemble
of many equally plausible explanations, the posterior distribution $p(x|b)$, and further, it completely
ignores the uncertainties associated with a particular solution.

\subsection{Bayesian computations}

Alternatively, one can compute the (conditional) mean of the unknown $x$ given the data $b$
to obtain ``averaged'' information about the unknown $x$, which is representative of the
posterior distribution $p(x|b)$. Further, one can infer the associated uncertainty information
respectively credible intervals by calculating the (conditional) variance of the unknown $x$.
Numerically, the computation of the mean $\mu^*$ and covariance $C^*$ entails the following
high-dimensional integrals
\begin{align*}
\mu^* = \int p(x | b) x dx\quad \mbox{and}\quad
C^* = \int p(x | b) (x-\mu^*)(x-\mu^*)^t dx,
\end{align*}
which are intractable by means of classical numerical quadrature rules, except for the case of very
low-dimensional problems, e.g., $n=1,2$. Instead, one can either resort to sampling through
Markov Chain Monte Carlo (MCMC) methods \cite{GilksRichardsonSpiegelhalter:1996,Liu:2008},
or (deterministic) approximate inference methods, e.g., variational approximations
\cite{Beal:2003,Jin:2012}.

The MCMC is currently the most versatile and popular method for exploring the posterior state.
It constructs a Markov chain with the posterior distribution $p(x|b)$ as its stationary distribution,
and draws samples from the posteriori by running the Markov chain, from which the sample mean and
sample variance can be readily computed. We illustrate the computational challenge on the classical
Metropolis-Hastings algorithm \cite{GilksRichardsonSpiegelhalter:1996}, cf. Algorithm \ref{alg:mcmc}.
The algorithm generates a set of $N$ dependent samples, and we use the sample mean and sample covariance
to approximate $\mu^*$ and $C^*$. In order to obtain accurate estimates, a fairly large
number of samples (after discarding the initial samples) are required, e.g., $N=1\times10^5\sim 1\times 10^6$.
At each iteration of the algorithm, one needs to evaluate the acceptance probability $\rho(x^{(k)},
x^{(*)})$, in order to determine whether the proposal $x^{(*)}$ should be accepted. In turn, this
essentially boils down to evaluating the likelihood $p(b|x^{(*)})$ (and the prior $p(x^{(*)})$, which
is generally much cheaper), and thus invokes the forward simulation $F(x^{(*)})$. In the context of
inverse problems for differential equations, it amounts to solving one partial differential equation (system) for
the proposed parameter vector $x^{(*)}$. However, for many practical applications, e.g., underground
flow, wave propagation, geophysics, chemical kinetics and nonlinear tomography, each forward simulation
is already expensive, and thus a straightforward application of the Metropolis-Hastings algorithm is prohibitively
expensive, if not impractical at all.

\begin{algorithm}[h!]
  \caption{Metropolis-Hastings algorithm for sampling from $p(x|b)$}\label{alg:mcmc}
  \begin{algorithmic}[1]
    \STATE Set $N$ and $x^{(0)}$, and select proposal distribution $q(x|x')$.
    \FOR {$k=0:N$}
       \STATE draw sample $x^{(*)}$ from $q(x|x^{(k)})$.
       \STATE compute the acceptance probability $\rho(x^{(k)},x^{(*)})=\min\left(1,\frac{p(x^{(*)}|b)q(x^{(k)}|x^{(*)})}{p(x^{(k)}|b)q(x^{(*)}|x^{(k)})}\right)$.
       \STATE generate $u$ from uniform distribution $U(0,1)$.
       \IF {$u\leq\rho(x^{(k)},x^{(*)})$}
          \STATE $x^{(k+1)}=x^{(*)}$.
       \ELSE
          \STATE $x^{(k+1)}=x^{(k)}$.
       \ENDIF
    \ENDFOR
  \end{algorithmic}
\end{algorithm}

Therefore, there have been many works on accelerating the Metropolis-Hastings algorithm.
A first idea is to use surrogates for the expensive forward simulator, which directly
reduces the cost of the sampling step. The surrogate can be constructed in many ways,
including reduced-order modeling \cite{Jin:2008} and polynomial chaos \cite{MarzoukNajm:2007} etc.
The key is to construct faithful surrogates, which is typically viable only in low-dimensional
parameter spaces. The second idea is to design problem-dependent proposal distributions.
In practice, simple proposal distributions are not adapted to the specific problem
structure, which is essentially true in case of inverse problems, where the Hessian
of the forward map is generally highly smoothing (with a large-dimensional numerical kernel). Accordingly,
the proposals are not well matched with the posteriori, and thus only few samples
can be accepted and a large part of the expensive computations (used in forward simulation) is wasted. To this end,
two effective strategies have been developed: One strategy relies on the idea of preconditioning,
which screens the proposal with inexpensive preconditioner, typically low-resolution
forward simulators or cheap surrogates \cite{HigdonLeeBi:2002,EfendievHouLuo:2006,ChristenFox:2005}.
The other strategy gleans optimization theory and exploits the Hessian of the forward
map for effective sample generation, e.g., Langevin methods \cite{StramerTweedie:1999}
and Hessian information \cite{GirolamiCalderhead:2011,FlathWilcoxAkcelikHill:2011,MartinWilcoxBursteddeGhattas:2012}.

Despite these impressive progresses, the application of these approaches to complex Bayesian
models, such as those arising in inverse problems for differential equations, remains challenging
and highly nontrivial, and there has been considerable interest in developing fast alternatives.
Approximate inference methods, e.g., variational methods \cite{JordanGhahramani:1999,Beal:2003}
and expectation propagation \cite{Minka:2001,Minka:2001a}, belong to this group. Approximate
inference methods often approximate the posterior distribution with computationally more trackable
distributions, in the hope of capturing the specific structure of the problem, and typically involve a sequence
of optimization/integration problems. The accuracy of these approximate methods depends crucially
on the adopted simplifying assumptions, but generally cannot be made
arbitrarily accurate. This is in stark contrast to the MCMC, which in theory can be made
arbitrarily accurate by running the Markov chain long enough. Nonetheless, often they
can offer a reasonable approximation within a small fractional of computing time needed
for the MCMC, and therefore have received increasing attention in practical applications.
Recently, a different optimization approach based on the idea of transporting the prior into
the posterior via polynomial maps was developed in \cite{MoselhyMarzouk}, where the
degree of polynomial restricts the accuracy of the approximation.

In this work, we study the EP method \cite{Minka:2001,Minka:2001a} for numerically
exploring the posterior distribution.
The EP method is of deterministic optimization nature. Algorithmically, it iteratively adjusts
the iterates within the Gaussian family by minimizing the Kullback-Leibler divergence and looping
over all the factors in likelihood/prior in a manner analogous to the classical Kaczmarz method
\cite{Kaczmarz:1937} for solving a linear system of equations.
The EP generally delivers an approximation with reasonable accuracy \cite{OpperWinther:2005}.

\subsection{Our contributions}

The goal of the present work is to study the EP method, and to demonstrate its significant
potentials for (nonlinear) inverse problems. We shall discuss the basic idea, theoretical
properties, and the efficient implementation for an important class of Bayesian formulations.

Our main contributions are as follows. Theoretically, we discuss well-definedness of the EP method
in the multi-dimensional case (Theorem \ref{thm:log_concave_well_defined}), which extends a known
result \cite{Seeger:2008}. We provide additional theoretical insight into the algorithm for some
problems with special structures.
Algorithmically, we propose a coupling of the EP method with an iterative linearization
strategy to handle nonlinear inverse problems. Further, numerical results for one
representative nonlinear inverse problem, electrical impedance tomography (EIT) with
real experimental data on a water tank, under sparsity constraints, will be presented.
To the best of our knowledge, this represents the first application of the EP method to a
PDE-constrained inverse problem.

The sparsity approach for EIT has recently been revisited in
\cite{GehreKluthLipponenJinSeppanenKaipioMaass:2012} in the framework of classical regularization,
and this work extends the deterministic model to a Bayesian setting. The extension provides
valuable insights into the sparsity imaging technique, e.g., the reliability of a particular
solution. The numerical results indicate that with proper implementation, the EP method can
deliver approximations with good accuracy to the posterior mean and covariance for EIT imaging, within
a small fraction of computational efforts required for the MCMC.

The rest of the paper is organized as follows. In Section 2, we describe the basic idea of the EP algorithm,
and discuss some theoretical properties. In Sections 3 and 4, we discuss the efficient implementation of a special
class of problems, and present numerical results for EIT imaging to illustrate its accuracy and efficiency.
In the appendices, we collect some auxiliary results and technical proofs.

\section{EP Theory}

In this part, we describe the expectation propagation (EP) algorithm for approximate Bayesian
inference of a generic probability distribution that can be factorized into a specific form.
We begin with the basic idea and algorithm, and then discuss some properties.

\subsection{Basic idea}
The goal of the EP algorithm is to construct a Gaussian approximation to a generic distribution $p(x)$.
It \cite{Minka:2001,Minka:2001a} approximates the distribution $p(x)$
by a Gaussian distribution $\tilde{p}(x)$. To this end, we first factorize the target distribution $p(x)$ into
\begin{equation*}
 p(x)=\prod_{i\in I} t_i(x).
\end{equation*}
where $I$ is a finite index set for the factors, which is often omitted below for notational simplicity.
The optimal factorization is dependent on the specific application.
Then a Gaussian approximation $\tilde{t}_i(x)$ to each factor $t_i(x)$
is sought after by minimizing the difference between the following two distributions
\[
  \prod_j \tt_j(x)\quad \mbox{ and }\quad t_i(x)\prod_{j\neq i} \tt_j(x)
\]
for every $i$. Following \cite{Minka:2001}, we measure the difference between two probability
distributions $p$ and $q$ by the Kullback-Leiber divergence $D_{KL}(p,q)$ (also known as
relative entropy and information gain etc. in the literature) \cite{KullbackLeibler:1951} defined by
\begin{equation*}
   D_{KL}(p,q) = \int p(x) \log \frac{p(x)}{q(x)} \dx .
\end{equation*}
In view of Jensen's inequality, the Kullback-Leibler divergence $D_{KL}(p,q)$
is always nonnegative, and it vanishes if and only if $p=q$ almost everywhere.
However it is not symmetric and does not satisfy the triangle inequality, and
thus it is not a metric.

Nonetheless, the divergence $D_{KL}(p,q)$ has been very popular in measuring
of the difference between two probability distributions $p$ and $q$. Specifically,
the divergence of $q$ from $p$, denoted $D_{KL}(p,q)$, is a measure of the
information lost when $q$ is used to approximate $p$, which underlies the
EP algorithm and many other approximations, e.g., variational method
\cite{Beal:2003,JordanGhahramani:1999} and
coarse graining \cite{ChaimovichShell:2011,BilionisZabaras:2013}.

Specifically, we update the factor approximation $\tilde{t}_i(x)$ by
\begin{equation}\label{eqn:kld_factor}
   \tilde{t}_i(x) = \argmin_{\tilde{q} \text{ Gaussian}} D_{KL}\left(t_i(x)\prod_{j\neq i} \tilde{t}_j(x),
   \tilde{q}(x)\prod_{j\neq i} \tt_j(x)\right).
\end{equation}
Intuitively, the goal of the update $\tilde{t}_i(x)$ is to ``learn'' the information contained
in the factor $t_i(x)$, and by looping over all factors, hopefully the essential characteristics
of the posterior $p(x)$ are captured. We observe that even though each update involves only a ``local''
problem, the influence is global due to the presence of the term $\prod_{j\neq i}\tt_j(x)$.

\subsection{Minimizing $D_{KL}$ and basic algorithm}
The central task in the EP algorithm is to compute the minimizers to the Kullback-Leibler
divergences, i.e., Gaussian factors $\tilde{t}_j(x)$.
To this end, we first consider the general problem of minimizing the Kullback-Leibler divergence
$D_{KL}(q,\tilde{q})$ between a non-gaussian distribution $q(x)$ and a Gaussian distribution
$\tilde{q}(x)=N(x;\mu,C)$ with respect to $\tilde{q}$, which is shown in the next result.
We note that the minimizer $\tilde{q}$ is completely determined by the mean $\mu$
and covariance $C$.
\begin{thm}\label{thm:kld}
Let $q$ be a probability distribution with a finite second moment, $\mu^*=E_q[x]$ and
$C^*=\mathrm{Var}_q[x]$. Then there exists a minimizer to
\begin{equation*}
  \min_{\tilde{q} \text{ Gaussian}} D_{KL}(q,\tilde{q})
\end{equation*}
over any compact set of Gaussians with a positive definite covariance. Further, if
$N(x;\mu^*,C^*)$ belongs to the interior of the compact set, then it is a local minimizer.
\end{thm}
\begin{proof}
The existence follows directly from the compactness assumption, and the continuity of
Kullback-Leibler divergence in the second argument.
Let $\tilde{q}(x)=N(x;\mu,C)$. The divergence $D_{KL}(q,\tilde{q})$ can be expanded into
\begin{align*}
D_{KL}(q,\tilde{q}) &= \int q(x) \left(\log q(x)
	+ \tfrac{d}{2}\log2\pi - \tfrac{1}{2}\log|{C}^{-1}|
	+ \tfrac{1}{2} (\mu-x)^t C^{-1} (\mu-x) \right) \dx.
\end{align*}
We first compute the gradient $\nabla_{\mu} D_{KL}(q,\tilde{q})$
of the divergence $D_{KL}(q,\tilde{q})$
with respect to the mean $\mu$ of the Gaussian factor $\tilde{q}(x)$, i.e.,
\begin{equation*}
  \nabla_{\mu} D_{KL}(q,\tilde{q}) = \int q(x) {C}^{-1} (\mu-x) \dx,
\end{equation*}
and set it to zero to obtain the first condition $\mu=E_q[x]$.

Likewise, we compute the gradient $\nabla_{P} D_{KL}(q,\tilde{q})$ of $D_{KL}(q,\tilde{q})$ with
respect to the precision $P={C}^{-1}$
\begin{equation*}
  \nabla_{P} D_{KL}(q,\tilde{q})
   = \int q(x) \left( -\tfrac{1}{2}P^{-1} + \tfrac{1}{2} (\mu-x)(\mu-x)^t \right) \dx,
\end{equation*}
where we have used the following matrix identities \cite{Dwyer:1967}
\begin{equation*}
\frac{\partial}{\partial A} \log|A| = A^{-1},\quad x^tAx = \tr(Axx^t)\quad \mbox{ and }\quad\frac{\partial}{\partial A} \tr(AB) = B^t.
\end{equation*}
The gradient $\nabla_PD_{KL}(q,\tilde{q})$ vanishes if and only if the condition
$C = E_{q}[(x-\mu)(x-\mu)^t]$ holds. Together with the condition on the mean $\mu$,
this is equivalent to $C=\mathrm{Var}_q[x]$. Now the second assertion follows
from the convexity of the divergence in the second argument and the assumption that $N(x;\mu^*,C^*)$
is an interior point of the compact set.
\end{proof}

The method is called Expectation Propagation (EP) because it propagates
the expectations of $x$ and $xx^t$ under the true distribution $p$ to the Gaussian approximation
$\tilde{p}$. By Theorem \ref{thm:kld}, minimizing $D_{KL}(q,\tilde{q})$ involves computing (possibly
very high-dimensional) integrals. It is feasible only if the integrals can be reduced to
a low-dimensional case. One such case will be discussed in Section \ref{sec:ep_projection}.
Now we can state the basic EP algorithm;
see Algorithm \ref{alg:ep_simple}. The stopping criterion at Step 7 can be based on monitoring
the relative changes of the approximate Gaussian factors $\tilde{t}_i(x)$.

\begin{algorithm}[h!]
  \caption{Approximate the posterior $p(x)=\prod_{i} t_i(x)$ by $\prod_{i} \tilde{t}_i(x)$ with Gaussian factors $\{\tilde{t}_i\}$.}
  \label{alg:ep_simple}
    \begin{algorithmic}[1]
       \STATE Start with initial distributions $\{\tilde{t}_i\}$
       \WHILE{not converged}
       \FOR{$i\in I$}
       \STATE Calculate $\tilde{q}(x)$ to $t_i(x) \prod_{j\neq i} \tilde{t}_j(x)$
         by $\min\limits_{\mbox{gaussian } \tilde{q}}D_{KL}(t_i(x) \prod_{j\neq i} \tilde{t}_j(x),\tilde{q}(x))$
       \STATE Set $\tilde{t}_i(x) = \frac{\tilde{q}(x)}{\prod_{j\neq i} \tilde{t}_j(x)}$
       \ENDFOR
       \STATE Check the stopping criterion.
       \ENDWHILE
    \end{algorithmic}
\end{algorithm}

\begin{rmk}
The EP method estimates both the mean and covariance of the posterior distribution.
In some cases, the mean might be close to the maximum a posteriori estimate.
Since the latter can be efficiently computed via modern minimization algorithms,
then the EP algorithm should be adapted to estimate the covariance only.
We shall not delve into the adaptation here, and leave it to a future work.
\end{rmk}

\subsection{Basic properties of EP algorithm}

In general, the convergence of the EP algorithm remains fairly unclear, and
nonconvergence has been observed for a naive implementation of the algorithm,
especially for the case that the factors are not log-concave, like the $t$ likelihood/prior. In
this part, we collect some results on the convergence issue.

Our first result is concerned with the exact recovery for Gaussian factors within one step.
\begin{prop}\label{eptheory:ConvergenceForGaussian}
Let the factors $\{t_i\}$ be Gaussian. Then the EP algorithm converges
after updating each factor approximation once with the limit $\tilde{t}_i=t_i,\ i\in I$.
\end{prop}
\begin{proof}
By the definition of the EP algorithm, the $i$th factor $\tilde{t}_i$ is found by
$\argmin_{\tilde{q}\ \mathrm{Gaussian}}D_{KL}\left(t_i\prod_{j\neq i}\tilde{t}_j,\tilde{q}\right)$.
Due to the Gaussian assumption on the factor $t_i$, the choice $\tilde{q}=t_i\prod_{j\neq i}
\tilde{t}_j$ is Gaussian and it minimizes the Kullback-Leibler divergence with a minimum
value zero, i.e., $\tilde{t}_i=t_i$. Clearly one sweep exactly recovers
all the factors.
\end{proof}

The next result shows the one-step convergence for the case of decoupled factors, i.e.,
$t_j(x)=t_j(x_j)$ with the partition $\{x_j\}$ of $x$ being pairwise disjoint. It is an analogue to
the Gauss-Seidel iteration method for a linear system of equations with a diagonal matrix.
We will provide a constructive proof, which relies on Theorem \ref{thm:proj} below, and
hence it is deferred to Appendix \ref{sec:app:proof}.
\begin{prop}\label{lem:convergence_blockdiagonal}
Let $q(x) = \prod t_i(x_i)$, and $\{x_i\}$ form a pairwise disjoint partition of the vector $x \in \R^n$.
If the factors $\tilde{t}_i(x)$ are initialized as $\tilde{t}_i(x)=\tilde{t}_i(x_i)$ 
then the EP algorithm converges after one iteration and the limit iterate $\{\tt_i\}$ satisfies
the moment matching conditions.
\end{prop}

A slightly more general case than Proposition \ref{lem:convergence_blockdiagonal} is the case of
a linear system with a ``triangular-shaped'' matrix. Then one naturally would conjecture that
the convergence is reached within a finite number of iterations. However, it is still unclear
whether this is indeed the case.

The next result shows the well-definedness of the EP algorithm for log-concave factors.
The proof requires several technical lemmas in Appendix \ref{sec:app:lemmas}. The
result is essentially a restatement of \cite[Theorem 1]{Seeger:2008}. We note that the
proof in \cite{Seeger:2008} is for the one-dimensional case, whereas the proof
below is multi-dimensional. Further, we specify more clearly the conditions.
\begin{thm}\label{thm:log_concave_well_defined}
Let the factors $\{t_i\}$ in the posterior $\prod t_i(x)$ be log-concave, uniformly bounded
and have support of positive measure. If at iteration $k$, the factor covariances $\{C_i\}$
of the approximations $\{\tt_i\}$ are positive definite, then the EP updates at $k+1$ step
are positive semidefinite.
\end{thm}
\begin{proof}
By the positive definiteness of the factor
covariances $\{C_i\}$, the covariance $C_\bi$ of the cavity distribution $\tilde{t}_\bi =
\prod_{j\neq i} \tilde{t}_j$ is positive definite. Next we show that the partition function
$Z= \int N(x;\mu_\bi,C_\bi) t_i(x) \dx$ is log-concave in $\mu_\bi$. A straightforward computation
yields that $N(x;\mu_\bi,C_\bi)$ is log-concave in $(x,\mu_\bi)$. By Lemma \ref{lem:prod_logconcave}
and log-concavity of $t_i(x)$, $N(x;\mu_\bi,C_\bi)t_i(x)$ is log-concave in $(x,\mu_\bi)$,
and thus Lemma \ref{lem:marg_logconcave} implies that $Z$ is
log-concave in $\mu_\bi$, and $\nabla_{\mu_\bi}^2 \log Z$ is negative semi-definite.

Now by Lemma \ref{cor:gaussian:tilted_moments}, the covariance $C$ of $Z^{-1}N(x;\mu_\bi,C_\bi)t_i(x)$ is given by
\begin{equation*}
  \begin{aligned}
   C &= C_\bi \left(\nabla_{\mu_\bi}^2 \log Z\right)C_\bi + C_\bi. 
  \end{aligned}
\end{equation*}
Meanwhile, by the boundedness of factors $t_i$, the covariance $Z^{-1}\int N(x;\mu_\bi,C_\bi)t_i(x) (x-\mu^*)
(x-\mu^*)^t \dx$ (with $\mu^*=Z^{-1}\int N(x;\mu_\bi,C_\bi)t_i(x)x\dx$) exists, and further by Lemma
\ref{lem:posdefvariance} and the assumption on $t_i$, it is positive definite. Hence,
$C$ is positive definite.

Since the Hessian $\nabla_{\mu_\bi}^2\log Z$ is negative semi-definite, $C-C_\bi=C_\bi
\left(\nabla_{\mu_\bi}^2 \log Z\right)C_\bi$ is also negative semi-definite, i.e., for
any vector $w$, $w^t C w \leq w^t C_\bi w$. Now we let $\tilde{w}= C^{\frac12} w$. Then
$\|\tilde{w}\|^2 \leq \tilde{w}^tC^{-\frac12} C_\bi C^{-\frac12} \tilde{w}$ for any
vector $\tilde{w}$. By the minmax
characterization of eigenvalues of a Hermitian matrix, $\lambda_\mathrm{min}(C^{-\frac12}
C_\bi C^{-\frac12})\geq1$. Consequently, $\lambda_\mathrm{max}(C^\frac12 C_\bi^{-1}
C^\frac{1}{2})\leq 1$, and equivalently $\|\tilde w\|^2\geq\tilde{w}^tC^\frac12 C_\bi^{-1}
C^\frac{1}{2}\tilde{w}$ for any vector $\tilde{w}$. With the substitution $w=C^{\frac{1}{2}}
\tilde{w}$, we get $w^tC^{-1}w \geq w^t C_\bi^{-1} w$ for any vector $w$, i.e., $C^{-1}-
C_\bi^{-1}$ is positive semidefinite. The conclusion follows from this and the fact that
the inverse covariance of the factor approximation $\tt_i$ is given by $C^{-1}-C_\bi^{-1}$.
\end{proof}

\begin{rmk}
Theorem \ref{thm:log_concave_well_defined} only ensures the well-definedness of the EP algorithm
for one iteration. One might expect that in case of strictly log-concave factors, the
well-definedness holds for all iterations. However, this would require a strengthened version of
the Pr\'ekopa-Leindler inequality, i.g., preservation of strict log-concavity under marginalization.
\end{rmk}
\section{EP with projections}\label{sec:ep_projection}

Now we develop an applicable EP scheme by assuming that, apart from the Gaussian base $t_0(x)$,
each factor involves a probability density function of the form $t_i(U_ix)$, i.e.,
\begin{equation}\label{eq:epwithproj}
t_0(x)\prod_i t_i(U_ix),
\end{equation}
where the matrices $U_i$ have full rank, and can represent either the projection onto several
variables or linear combinations of the variables. The specific form of $U_i$ depends on
concrete applications. It occurs e.g., in
robust formulation with the Laplace likelihood \cite{Gao:2008} and total variation prior.
It represents an important structural property,
and should be utilized to design efficient algorithms. The goal of this part is to derive
necessary modifications of the basic EP algorithm and to spell out all the implementation details,
following the very interesting works \cite{vanGerven:2010}.

Each factor $t_i(U_ix)$ effectively depends on $x$ through $U_ix$, albeit it lives on a high-dimensional
space (that of $x$). This enables us to derive concise formulas for the mean and the covariance
under the product density of one nongaussian factor $t_i(U_ix)$ and a Gaussian; see
the following result. The proof follows from a standard but tedious change of variable, and hence
it is deferred to Appendix \ref{sec:app:proof:thmproj}.
\begin{thm}\label{thm:proj}
Let $\mu\in\mathbb{R}^n$, $C\in\mathbb{R}^{n\times n}$ be symmetric positive definite and
$U\in \R^{l\times n}\ (l\le n)$ be of full rank, $Z=\int t(Ux)N(x;\mu,C)\dx$ be the normalizing
constant. Then the mean $\mu^\ast=E_{Z^{-1}t(Ux) N(x;\mu,C)}[x]$
and covariance $C^\ast=\Var_{Z^{-1}t(Ux)N(x;\mu,C)}[x]$ are given by
\begin{equation*}
   \begin{aligned}
     \mu^\ast &= \mu+ CU^t(UCU^t)^{-1}(\overline{s}-U\mu),\\
     C^\ast &= C + CU^t(UCU^t)^{-1}(\overline{C}-UCU^t)(UCU^t)^{-1}UC,
   \end{aligned}
\end{equation*}
where $\overline{s}\in\mathbb{R}^l$ and $\overline{C}\in\mathbb{R}^{l\times l}$ are respectively given by
\begin{equation*}
   \overline{s}=E_{Z^{-1}t(s) N(s;U\mu,UCU^t)}[s] \quad\mbox{and}\quad \overline{C}=\Var_{Z^{-1}t(s)N(s;U\mu,UCU^t)}[s].
\end{equation*}
\end{thm}

Now we can present the EP algorithm for the approximate inference of the posterior distribution
\eqref{eq:epwithproj}; see Algorithm \ref{alg:epproj}. We shall skip the update of the Gaussian
factor $t_0(x)$. Here we have adopted the following canonical representation of the Gaussian
approximation $\tilde{t}_i(U_ix)$ to $t_i(U_ix)$, i.e.,
\begin{equation}\label{eqn:ttilde}
  \tilde{t}_i(U_ix) \propto \exp\left((U_ix)^th_i-\frac{1}{2}(U_ix)^tK_i(U_ix)\right),
\end{equation}
where $h_i$ and $K_i$ are low-rank parameters to be updated. With the representation \eqref{eqn:ttilde}, there holds
\begin{equation*}
 \begin{aligned}
   \tilde{t}_i(x)& = N(x;\mu_i,C_i) \propto e^{-\frac{1}{2}(x-\mu_i)^tC_i^{-1}(x-\mu_i)}\\
            &\propto e^{-\frac{1}{2}x^tC_i^{-1}x+x^tC_i^{-1}\mu_i}=: G(x,h_i,K_i),
\end{aligned}
\end{equation*}
with $h_i=C_i^{-1}\mu_i$ and $K_i=C_i^{-1}$. Then the approximate distribution $\prod_i\tilde{t}_i(U_ix)$
is a Gaussian distribution with parameters
\begin{equation*}
    K = K_0 + \sum_iU_i^tK_iU_i\quad \mbox{and}\quad
    h = h_0 + \sum_iU_i^th_i,
\end{equation*}
where $(K_0,h_0)$ is the parameter tuple of the Gaussian base $t_0$.

\begin{algorithm}[h]
   \caption{Serial EP algorithm with projection.}\label{alg:epproj}
   \begin{algorithmic}[1]
      \STATE Initialize with $K_0=C_0^{-1}$, $h_0=C_0^{-1}\mu_0$, $K_i = I$ and $h_i = 0$ for $i\in I$
      \STATE $K = K_0 + \sum_i U_i^tK_iU_i$
      \STATE $h = h_0 + \sum_i U_i^th_i$
      \WHILE{not converged}
        \FOR{$i\in I$}
        \STATE $\widehat{C}_i^{-1} = (U_iK^{-1}U_i^t)^{-1}-K_i$
        \STATE $\widehat{\mu}_i 
          = (I-U_iK^{-1}U_i^tK_i)^{-1}(U_iK^{-1}h-U_iK^{-1}U_i^th_i)$
        \STATE $K_i = \mathrm{Var}_{Z^{-1}t_i(s_i)N(s_i;\widehat{\mu}_i,\widehat{C}_i)}[s_i]^{-1} - \widehat{C}_i^{-1}$, with $Z=\int t_i(s_i)N(s_i;\widehat{\mu}_i,\widehat{C}_i)\ds_i$
        \STATE $h_i = \mathrm{Var}_{Z^{-1}t_i(s_i)N(s_i;\widehat{\mu}_i,\widehat{C}_i)}[s_i]^{-1}E_{Z^{-1}t_i(s_i)N(s_i;\widehat{\mu}_i,\widehat{C}_i)}[s_i] - \widehat{C}_i^{-1}\widehat{\mu}_i$
        \STATE $K = K_0 + \sum_i U_i^tK_iU_i$
        \STATE $h = h_0 + \sum_i U_i^th_i$
        \ENDFOR
        \STATE Check the stopping criterion.
      \ENDWHILE
      \STATE Output the covariance $C=K^{-1}$ and the mean $\mu=K^{-1}h$.
   \end{algorithmic}
\end{algorithm}

The inner loop in Algorithm \ref{alg:epproj} spells out the minimization step in Algorithm
\ref{alg:ep_simple} using the representation \eqref{eqn:ttilde} and Theorem \ref{thm:proj}.
In brevity, Steps 6--7 form the (cavity) distribution $\tilde{p}^\bi=t_0\prod_{j\neq i}
\tilde{t}_j(U_jx)$, i.e., mean $\mu_\bi$ and covariance $C_\bi$, and the low-rank representation
$\widehat{C}_i=U_iC_\bi U_i^t$ and $\widehat{\mu}_i=U_i\mu_\bi$; Steps 8--9 update the
parameters pair $(h_i,K_i)$ for the $i$th gaussian factor $\tilde{t}_i(U_ix)$.
Steps 10--11 then update the global approximation $(h,K)$ by the new values for the $i$th
factor. There is a parallel variant of the EP algorithm, where steps 10--11 are moved
behind the inner loop. The parallel version can be much faster on modern multi-core
architectures, since it scales linearly in the number of processors, but it is less robust
than the serial variant. Now we derive the main steps of Algorithm \ref{alg:epproj}.

\begin{enumerate}[(i)]
  \item Step 6 forms the covariance of the cavity distribution $\tilde{p}^\bi$: From the elementary relation
      \begin{equation}\label{eqn:cavity}
          G(x,h,K)/G(x,U_i^th_i,U_i^tK_iU_i) = G(x,h-U_i^th_i,K-U_i^tK_iU_i),
      \end{equation}
      we directly get the cavity covariance $C_\bi=(K-U_i^tK_iU_i)^{-1}$. Consequently, by
      Woodbury's formula \cite{Woodbury:1950}
      \begin{equation}\label{eqn:woodbury}
         (A+WCV)^{-1}=A^{-1}-A^{-1}W(C^{-1}+VA^{-1}W)^{-1}VA^{-1},
      \end{equation}
      and letting $\widehat{K}_i=U_iK^{-1}U_i^t$, we arrive at a direct formula for computing
      \begin{equation*}
        \begin{aligned}
           \widehat{C}_i &=U_i^t C_\bi U_i = U_i^t(K-U_i^tK_iU_i)^{-1}U_i^t\\
             & = U_i\left[K^{-1}-K^{-1}U_i^t(-K_i^{-1}+\widehat{K}_i)^{-1}U_iK^{-1}\right]U_i^t\\
             & = \left[I+\widehat{K}_iK_i(I-\widehat{K}_iK_i)^{-1}\right]\widehat{K}_i = (I-\widehat{K}_iK_i)^{-1}\widehat{K}_i.
        \end{aligned}
      \end{equation*}
      Now inverting the matrix gives the
      cavity covariance $\widehat{C}_i^{-1}=\widehat{K}_i^{-1}-K_i$.
   \item  Step 7 forms the cavity mean $\widehat{\mu}_i=U_i\mu_\bi$. In view of the identity
      \eqref{eqn:cavity}, we have $\mu_\bi=(K-U_i^tK_iU_i)^{-1}(h-U_i^th_i)$, and consequently
      \begin{equation*}
         \begin{aligned}
            \widehat{\mu}_i &= U_i(K-U_i^tK_iU_i)^{-1}(h-U_i^th_i)\\
              & = U_i\left[K^{-1}-K^{-1}U_i^t(-K_i^{-1}+\widehat{K}_i)^{-1}U_iK^{-1}\right](h-U_i^th_i)\\
              & = \left[I-\widehat{K}_i(-K_i^{-1}+\widehat{K}_i)^{-1}\right](U_iK^{-1}h-\widehat{K}_ih_i)\\
              & = \left[I+\widehat{K}_iK_i(I-\widehat{K}_iK_i)^{-1}\right](U_iK^{-1}h-\widehat{K}_ih_i)\\
              & = (I-\widehat{K}_iK_i)^{-1}(U_iK^{-1}h-\widehat{K}_ih_i).
         \end{aligned}
      \end{equation*}
   \item Steps 8 and 9 updates the parameter $(K_i,h_i)$ of the Gaussian approximation $\tilde{t}_i(U_ix)$.
       It relies on the Gaussian approximation $p^\ast(x)=N(x;\mu^\ast,C^\ast)$ to $t_i(U_ix)\tilde{p}^\bi(x)$, i.e.,
       \begin{equation*}
         \tilde{t}_i(U_ix)\propto p^\ast(x)/\tilde{p}^\bi(x).
       \end{equation*}
       Thus one first needs to derive the mean $\mu^\ast$ and the
       covariance $C^\ast$ in Theorem \ref{thm:kld}. By Theorem \ref{thm:proj}, with $Z=\int t_i(s_i)
       N(s_i,\widehat{\mu}_i,\widehat{C}_i)\ds_i$ and
       \begin{equation*}
         \bar{p}_i(s_i)=Z^{-1}t_i(s_i)N(s_i,\widehat{\mu}_i,\widehat{C}_i),
       \end{equation*}
       we have
       \begin{equation}\label{eqn:epapprox}
         \begin{aligned}
            \mu^\ast &= \mu_\bi+ C_\bi U_i^t\widehat{C}_i^{-1}(E_{\bar{p}_i}[s_i]-U_i\mu_\bi),\\
            C^\ast &= C_\bi + C_\bi U_i^t\widehat{C}_i^{-1}(\mathrm{Var}_{\bar{p}_i}[s_i]-\widehat{C}_i)\widehat{C}_i^{-1}U_iC_\bi.
         \end{aligned}
       \end{equation}
       Therefore, the Gaussian factor $\tilde{t}_i$, i.e., $(h_i,K_i)$, is determined by ${C^\ast}^{-1}-{C_\bi}^{-1}$
       and $(C^\ast)^{-1}\mu^\ast-(C_\bi)^{-1}\mu_\bi$, respectively.
       To this end, we first let
       \begin{equation*}
         \begin{aligned}
           C &= (\mathrm{Var}_{\bar{p}_i}[s_i]-\widehat{C}_i)^{-1}+\widehat{C}_i^{-1}U_iC_\bi (C_\bi)^{-1} C_\bi U_i^t\widehat{C}_i^{-1}
            & = (\mathrm{Var}_{\bar{p}_i}[s_i]-\widehat{C}_i)^{-1} + \widehat{C}_i^{-1},
         \end{aligned}
       \end{equation*}
       where we have used the definition  $\widehat{C}_i=U_iC_\bi U_i^t$.
       Then it follows from \eqref{eqn:epapprox} and Woodbury's formula \eqref{eqn:woodbury} that
       \begin{equation*}
         \begin{aligned}
              (C^\ast)^{-1}-(C_\bi)^{-1}
             =& \left(C_\bi + C_\bi U_i^t\widehat{C}_i^{-1}(\mathrm{Var}_{\bar{p}_i}[s_i]-\widehat{C}_i)\widehat{C}_i^{-1}U_iC_\bi\right)^{-1}-(C_\bi)^{-1}\\
             =& -(C_\bi)^{-1}C_\bi U_i^t\widehat{C}_i^{-1}C^{-1}\widehat{C}_i^{-1}U_iC_\bi (C_\bi)^{-1}\\
             =& -U_i^t\widehat{C}_i^{-1}[(\mathrm{Var}_{\bar{p}_i}[s_i]-\widehat{C}_i)^{-1}+\widehat{C}_i^{-1}]^{-1}\widehat{C}_i^{-1}U_i\\
             =& U_i^t(\mathrm{Var}_{\bar{p}_i}[s_i]^{-1}-\widehat{C}_i^{-1})U_i.
         \end{aligned}
       \end{equation*}
       Analogously, we deduce from \eqref{eqn:epapprox} and the above formula that
       \begin{equation*}
          \begin{aligned}
             (C^\ast)^{-1}\mu^\ast-(C_\bi)^{-1}\mu_\bi
             =&U_t(\mathrm{Var}_{\bar{p}_i}[s_i]^{-1}E_{\bar{p}_i}[s_i]-(\widehat{C}_i)^{-1}U_i\mu_\bi).
          \end{aligned}
       \end{equation*}
       By comparing these identities with the representation \eqref{eqn:ttilde} for
       $\tilde{t}(U_ix)$, we arrived at the desired updating formulae in Steps 8--9  for the tuple $(K_i,h_i)$:
       \begin{equation*}
          \begin{aligned}
             K_i & = \mathrm{Var}_{\bar{p}_i}[s_i]^{-1}-\widehat{C}_i^{-1},\\
             h_i & = \mathrm{Var}_{\bar{p}_i}[s_i]^{-1}E_{\bar{p}_i}[s_i]-(\widehat{C}_i)^{-1}U_i\mu_\bi.
          \end{aligned}
       \end{equation*}
	\item Steps 10--11 update the parameters $(K,h)$ of the global approximation with the current parameters
	of the $i$th factor. The update of $K$ is implemented as a Cholesky up/downdate to the Cholesky factor of $K$.
    For example, 	computing the Cholesky factor $\widehat{L}$ of the rank-one update $A \pm xx^t$ from
    the Cholesky factor $L$ of $LL^t = A$ by Cholesky up/downdate is much faster than computing the Cholesky factor
    $\widehat{L}$ directly \cite{GillGolubMurraySaunders:1974}. In particular, for one-dimensional projection
    space, we use a Cholesky update with $x=U_i^t\sqrt{K_i}$ if $K_i > 0$ and a Cholesky downdate
    with $x=U_i^t\sqrt{-K_i}$ if $K_i < 0$. A Cholesky downdate may fail if the resulting matrix is no longer
    positive definite; see Theorem \ref{thm:log_concave_well_defined} for conditions ensuring positive definiteness.
\end{enumerate}

The major computational efforts in the inner loop lie in computing the mean $\widehat{\mu}_i$ and the
inverse covariance $\widehat{C}_i^{-1}$. The mean $E_{\bar{p}_i}[s_i]$ and covariance $\mathrm{Var}_{\bar{p}_i}[s_i]$ necessitate
integral, which is tractable if and only if the projections $\{U_i\}$ have very
low-dimensional ranges. Our numerical examples with sparsity constraints in Section \ref{sec:numer} involve only one-dimensional
integrals.

Algorithm \ref{alg:epproj} is computable in exact arithmetic as long
as the inverses can be computed. This hinges on the positive definitieness
of $K$ and $(U_iK^{-1}U_i^t)^{-1}-K_i$. Note that $\mathrm{Var}_{\bar{p}_i}
[s_i]$ is positive definite by Lemma \ref{lem:posdefvariance} if the factors
are nondegenerate. However, numerical instabilities can occur if those
expressions become too close to singular. Choosing initial values for $K_i$
gives some control over the condition of $K = K_0 + \sum_i U_i^tK_iU_i$
in the first iteration. However, general bounds for the conditioning of $K$
at later iterations are still unknown.

\begin{algorithm}[h!]
   \caption{EP for nonlinear problem $F(x)=b$.}\label{alg:epnonlin}
   \begin{algorithmic}[1]
      \STATE Initialize with $\mu = \mu^0$
      \FOR{$k = 0:\mathrm{MaxIter}$}
        \STATE Linearize around $\mu^k$ such that $F(x)\approx F(\mu^k)+F'(\mu^k)(x-\mu^k) + O(\|x-\mu^k\|_2^2)$
		\STATE Use Algorithm \ref{alg:epproj} to approximate the linearized problem with a Gaussian $N(\mu^k_\ast,C^k_\ast)$
		\STATE Compute Barzilai-Borwein step size $\tau^k$
		\STATE Update $\mu^{k+1} = \mu^k + \tau^k (\mu^k_\ast - \mu^k)$
		\STATE Check the stopping criterion.
      \ENDFOR
      \STATE Output the covariance $C^\ast$ and the mean $\mu^\ast$.
   \end{algorithmic}
\end{algorithm}

Now we discuss the case of a nonlinear operator $F$. To this end, we employ the idea of recursive
linearization. The complete procedure is given in Algorithm \ref{alg:epnonlin},
where $\mathrm{MaxIter}$ is the maximum number of iterations. Like most optimization algorithms based on linearization,
step size control is often necessary. Only for problems with low degree of nonlinearity, the step
size selection can be dropped by setting $\mu^{k+1} = \mu^k_\ast$, i.e., a step size of $1$. Generally, it
improves the robustness of the algorithm and avoids jumping back and forth between two
states. There are many possible rules,
and we only mention the Barzilai-Borwein rule here.
It was proposed for gradient type methods in \cite{BarzilaiBorwein:1988}.
It approximates the Hessian by a scalar multiple of the identity matrix based on
the last two iterates and descent directions so as to exploit the curvature
information to accelerate/stabilize the descent into the minimum.
Here we adopt the following simple update rule
\begin{equation*}
  \mu^{k+1} = \mu^k + \tau^k (\mu^k_\ast - \mu^k),
\end{equation*}
and thus $\mu^k_\ast - \mu^k$ serves as a ``descent'' direction and $\mu^k$ acts as the
iterate in the Barzilai-Borwein rule. The step size $\tau^k$ is computed by
\begin{equation*}
  \tau^k=\frac{\langle\mu^k-\mu^{k-1},(\mu^k_\ast-\mu^k)-(\mu^{k-1}_\ast-\mu^{k-1})\rangle}{
  \langle \mu^k-\mu^{k-1},\mu^k-\mu^{k-1}\rangle},
\end{equation*}
which approximately solves (in a least-squares sense) the following equation
\begin{equation*}
   \tau^k (\mu^k-\mu^{k-1}) \approx (\mu^k_\ast - \mu^k) - (\mu^{k-1}_\ast - \mu^{k-1}).
\end{equation*}

\section{Numerical experiments and discussions}\label{sec:numer}
In order to illustrate the feasibility and efficiency of the proposed method, in this part, we present
numerical experiments with electrical impedance tomography (EIT) of recovering the conductivity
distribution from voltage measurements on the boundary based on sparsity constraints. EIT is one typical
parameter identification problem for partial differential equations, and sparsity is a recent promising imaging
technique \cite{JinMaass:2012ip,JinKhanMaass:2012}. The reconstructions are obtained from experimental
data with a water tank, immersed with plastic/metallic bars.

\subsection{Mathematical model: Electrical impedance tomography}
Electrical Impedance Tomography (EIT) is a non-destructive, low-cost and portable
imaging modality developed for reconstructing the conductivity distribution in the interior
of a concerned object. One typical experimental setup is as follows.
One first attaches a set of metallic electrodes to the surface of the object, and then injects
electrical currents into the object through these electrodes. The resulting potentials are then measured
on the same set of electrodes. The inverse problem is then to infer the electrical
conductivity distribution from these noisy measurements. It has found applications
in many areas, including medical imaging and nondestructive evaluation.

The most accurate model for an EIT experiment, the complete electrode model (CEM), reads
\begin{equation}\label{eqn:cem}
   \left\{\begin{aligned}
         \begin{array}{ll}
           -\nabla\cdot(\sigma\nabla v)=0 & \mbox{ in }\Omega,\\[2ex]
           u+z_l\sigma\frac{\partial v}{\partial n}=V_l& \mbox{ on } e_l, l=1,\ldots,L,\\[2ex]
           \int_{e_l}\sigma\frac{\partial v}{\partial n}ds =I_l& \mbox{ for } l=1,\ldots, L,\\ [2ex]
           \sigma\frac{\partial v}{\partial n}=0&\mbox{ on } \Gamma\backslash\cup_{l=1}^Le_l,
         \end{array}
  \end{aligned}\right.
\end{equation}
where $\Omega$ denotes the domain occupied by the object and $\Gamma$ is the boundary,
$\sigma$ is the electrical conductivity distribution, $v\in H^1(\Omega)$ is the electric
potential, $\{e_l\}_{l=1}^L\subset\Gamma$ denote the surfaces under the $L$ electrodes and
$n$ is the unit outward normal to the boundary $\Gamma$. In the model \eqref{eqn:cem}, $z_l$
denotes contact impedance of the $l$th electrode caused by the resistive layer at the
interface between the $l$th electrode and the object, and $I_l$ and $V_l$ are the electrode
current and voltage, respectively, on the $l$th electrode. The currents $\{I_l\}_{l=1}^L$
satisfy the charge conservation law, and hence $\sum_{l=1}^LI_l=0$. Upon introducing the
subspace $\RS^L=\{ w \in \R^L : \sum_{l=1}^L w_l = 0 \}$, we have $I:=(I_1,\ldots,I_L)\in \RS^L$.
Meanwhile, to ensure uniqueness of the solution to the model \eqref{eqn:cem}, the ground potential
has to be fixed; this can be achieved, e.g., by requiring $\sum_{l=1}^LV_l=0$, i.e., $V:=(V_1,\ldots,V_L)\in\RS^L$.
The model was first introduced in \cite{ChengIsaacsonNewellGisser:1989}, and mathematically studied
in \cite{SomersaloCheneyIsaacson:1992,JinMaass:2012,GehreKluthLipponenJinSeppanenKaipioMaass:2012}.
In particular, for every strictly positive conductivity $\sigma \in L^\infty(\Omega)$ and
current pattern $I\in \RS^L$ and positive contact impedances $\{z_l\}_{l=1}^L$,
there exists a unique solution $(v,V) \in H^1(\Omega)\times\RS^L$ \cite{SomersaloCheneyIsaacson:1992}.
For a fixed current pattern $I$, we denote by $F(\sigma) = V(\sigma)$ the forward operator, and in
the case of multiple input currents, we define $F(\sigma)$ by stacking respective potentials vectors.
In the computation, we discard the potentials measured on current-carrying electrodes, which makes
the model more robust with respect to contact impedances $\{z_l\}$.
The inverse problem consists of inferring the physical conductivity $\sigma^\dagger$ from the
measured, thus noisy, electrode voltages $V^\delta$.

There are many deterministic imaging algorithms for the EIT inverse problem. The Bayesian approach has gained increasing
popularity in the last few years. The first substantial Bayesian modeling in EIT imaging was presented
by Nicholls and Fox \cite{NichollsFox:1998} and Kaipio et al \cite{KaipioKolehmainenSomersaloVauhkonen:2000}.
In particular, in \cite{KaipioKolehmainenSomersaloVauhkonen:2000}, Kaipio et al applied the Bayesian
approach to EIT imaging, and systematically discussed various computational issues. West et al
\cite{WestAykroydMengWilliams:2004} studied the Bayesian formulation in medical applications of EIT,
with an emphasis on incorporating explicit geometric information, e.g., anatomical structure and
temporal correlation. The proper account of discretization errors and modeling errors has also
received much recent attention; see e.g., \cite{Arridge:2006,NissinenHeikkinenKaipio:2008}. We refer
to \cite{WatzenigFox:2009} for an overview of the status of statistical modeling in EIT imaging.

In this work, we adopt an approach based on sparsity constraints, which have demonstrated significant potentials in
EIT imaging \cite{JinKhanMaass:2012} and other parameter identification problems \cite{JinMaass:2012ip}. Now we build the probabilistic
model as follows. First, we specify the likelihood function $p(V^\delta|\sigma)$. To this end, we assume
that the noisy voltages are contaminated by additive noises, i.e., $V^\delta = V(\sigma^\dagger) + \zeta$ for
some noise vector $\zeta $, and further, the noise components follow an independent and identically distributed
Gaussian distribution with zero mean and inverse variance $\alpha$. Then the likelihood function $p(V^\delta|\sigma)$ reads
\begin{equation*}
  p(V^\delta|\sigma) \propto e^{-\frac{\alpha}{2}\|F(\sigma)-V^\delta\|^2_2}
\end{equation*}
The prior distribution $p(\sigma)$ encodes our a priori knowledge or physical constraints on the
sought-for conductivity distribution $\sigma$. First of all, the conductivity is pointwise non-negative due
to physical constraints. Mathematically, the forward model \eqref{eqn:cem} is not well defined
for non-positive conductivities, and it degenerates if the conductivity tends to zero. Thus we enforce
strict positivity on the conductivity $\sigma$ using the characteristic function $\chi_{[\Lambda,\infty]}(\sigma_i)$
on each component with a small positive constant $\Lambda$. Second, we assume that the conductivity
consists of a known (but possibly inhomogeneous)
background $\sigma^\text{bg}\in\R^n$ plus some small inclusions. As is now widely accepted,
one can encapsulate this ``sparsity'' assumption probabilistically by using a Laplace prior on the
difference $\sigma-\sigma^\text{bg}$ of the sought-for conductivity $\sigma$ from the background $\sigma^\text{bg}$:
\begin{equation*}
   \tfrac{\lambda}{2} e^{-\lambda||\sigma-\sigma^\text{bg}||_{\ell^1}},
\end{equation*}
where $\lambda>0$ is a scale parameter, playing the same role of a regularization parameter in regularization methods.
Intuitively, the $\ell_1$ norm penalizes many small deviations from $\sigma^\text{bg}$ stronger than a few big
deviations, thereby favoring small and localized inclusions. We note that one may incorporate an additional smooth prior
to enhance the cluster structure of typical inclusions 

The Bayesian formulation involves two hyperparameters ($\alpha$, $\lambda$), and
their choice is important for the reconstruction resolution. The inverse variance
$\alpha$ should be available from the specification of the measurement equipment.
The choice of the scale parameter $\lambda$ (in the Laplace distribution) is much
more involved. This is not surprising since it plays a role analogous to the
regularization parameter in Tikhonov regularization, which is known to be highly
nontrivial. One systematic way is to use hierarchical models \cite{WangZabaras:2005},
i.e., viewing the parameters as unknown random variables with their own (hyper-)priors.
However, in general, an automated choice of the hyperparameters is still an actively
researched topic. Here we have opted for an ad hoc approach: an initial value for
$\lambda$ was first estimated using a collection of typical reconstructions and
then fine tuned by trial and error.

By Bayes' formula, this leads to the following posterior distribution $p(\sigma|V^\delta)$
of the sought-for conductivity distribution $\sigma$
\begin{equation*}
   p(\sigma|V^\delta) \sim e^{-\tfrac{\alpha}{2}||F(\sigma)-V^\delta||_2^2}  e^{-\lambda||\sigma-\sigma^\text{bg}||_{\ell^1}}
   \prod_{i=1}^n\chi_{[\Lambda,\infty]}(\sigma_i).
\end{equation*}

To apply the EP algorithm to the posterior $p(\sigma|V^\delta)$, we factorize it into
\begin{align*}
t_0(\sigma) &= e^{-\tfrac{\alpha}{2}||F(\sigma)-V^\delta||_2^2}, \\
t_i(\sigma) &= e^{-\lambda|\sigma_i-\sigma_i^\text{bg}|} \chi_{[\Lambda,\infty]}(\sigma_i),\quad i=1,\ldots,n.
\end{align*}
Then the factor $t_0$ is Gaussian (after linearization,
see Algorithm \ref{alg:epnonlin}), and the remaining factors $t_1,\ldots,t_n$ each depend only on one
component of $\sigma$. This enables us to rewrite the necessary integrals, upon appealing to
Theorem \ref{thm:proj}, into one-dimensional ones. They have to be carried out semi-analytically,
though, because a purely numerical quadrature may fail to deliver the required accuracy due to the singular
shape of the Laplace prior, which is plagued with possible cancelation and underflow.

Finally, we briefly comment on the computational complexity of
the EP algorithm. It needs $n$ evaluations of the forward operator $F$ to
compute the linearized model $F'$ in the outer iteration. Each evaluation of $F$ is dominated
by solving a linear problem of size $n\times n$, and for each outer evaluation the stiffness
matrix remains unchanged and one can employ Cholesky decomposition to reduce the
computational efforts. Further, it can be carried out in parallel if desired.
Each inner iteration (linear EP) involves $n$ EP updates, which is dominated by one
Cholesky up/down-date of $K$ and one forward/backward substitution. The
cost of performing the low dimensional integration is negligible compared with other pieces.
With $k$ outer iterations and $l$ inner iterations for each outer iteration, this gives a
total cost of
\[ O( kn^3/3 + kln(n + n^2/2)). \]

\subsection{Numerical results and discussions}
In this part, we illustrate the performance of the EP algorithm with real experimental data on a water tank.
The measurement setup is illustrated in Fig. \ref{fig:exp}. The diameter of the
cylindrical tank was $28$ cm and the height was $10$ cm. For the EIT measurements, sixteen
equally spaced metallic electrodes (width $2.5$ cm, height $7$ cm) were attached to the
inner surface of the tank. The tank was filled with tap water, and objects of different
shapes and materials (either steel or plastic) were placed in the tank. We consider six
cases listed in Table \ref{tab:exam}. All the inclusions
were symmetric in height. In all cases, the excess water was removed from the tank, so
that the height of water level was also 7 cm, the same as the height of the electrodes.
The EIT measurements were conducted with the KIT 2 measurement system \cite{Savolainen2003}.
In these experiments, fifteen different current injections were carried out, between
all adjacent electrode pairs. For each injection, 1 mA of current was injected into
one electrode and the other electrode was grounded. Then the potentials of all other
fifteen electrodes with respect to the grounded electrode were measured.
\begin{table}[h!]
   \centering
   \caption{Description of experiments.} \label{tab:exam}
  \begin{tabular}{c|l}
     \hline
     case & description\\
     \hline
     1 & one plastic cylinder\\
     2 & two plastic cylinders\\
     3 & two neighboring plastic cylinders\\
     4 & three neighboring plastic cylinders\\
     5 & one plastic and one metallic cylinders\\
     6 & two plastic and one metallic cylinders\\
  \hline
  \end{tabular}
\end{table}

Due to cylindrical symmetry of the inclusions and the water tank, a two-dimensional CEM model was adequate
for describing the solution of the EIT forward problem. We note that with the two-dimensional approximation
of a cylindrical target, the conductivity $\sigma$ represents the product $\gamma h$, where $\gamma$ is the
(cylindrically symmetric) three-dimensional conductivity distribution and $h$ is the height of the cylinder.
Accordingly, $z_l$ represents $z_l=\xi_l/h$, where $\xi_l$ is the contact impedance in a three-dimensional
model. The model was discretized by the piecewise linear finite element method, since the forward solution
$u(\sigma)$ has only limited regularity, especially in the regions close to surface electrodes. The domain
$\Omega$ was triangulated into a mesh with $424$ nodes and $750$ triangle elements, and the mesh is locally refined around
the surface electrodes; see Fig. \ref{fig:exp} for a schematic illustration. The conductivity is only computed
on $328$ inner nodes since it is fixed on the boundary. This is reasonable when the inclusions are in the
interior, which is the case for our experiments. Further, the same finite element space was used for discretizing both the
potential $ u(\sigma)$ and the conductivity $\sigma$.

\begin{figure}
\centering
\begin{tabular}{cc}
\includegraphics[trim = 4.5cm 2cm 4cm 3.5cm, clip=true,height=5cm]{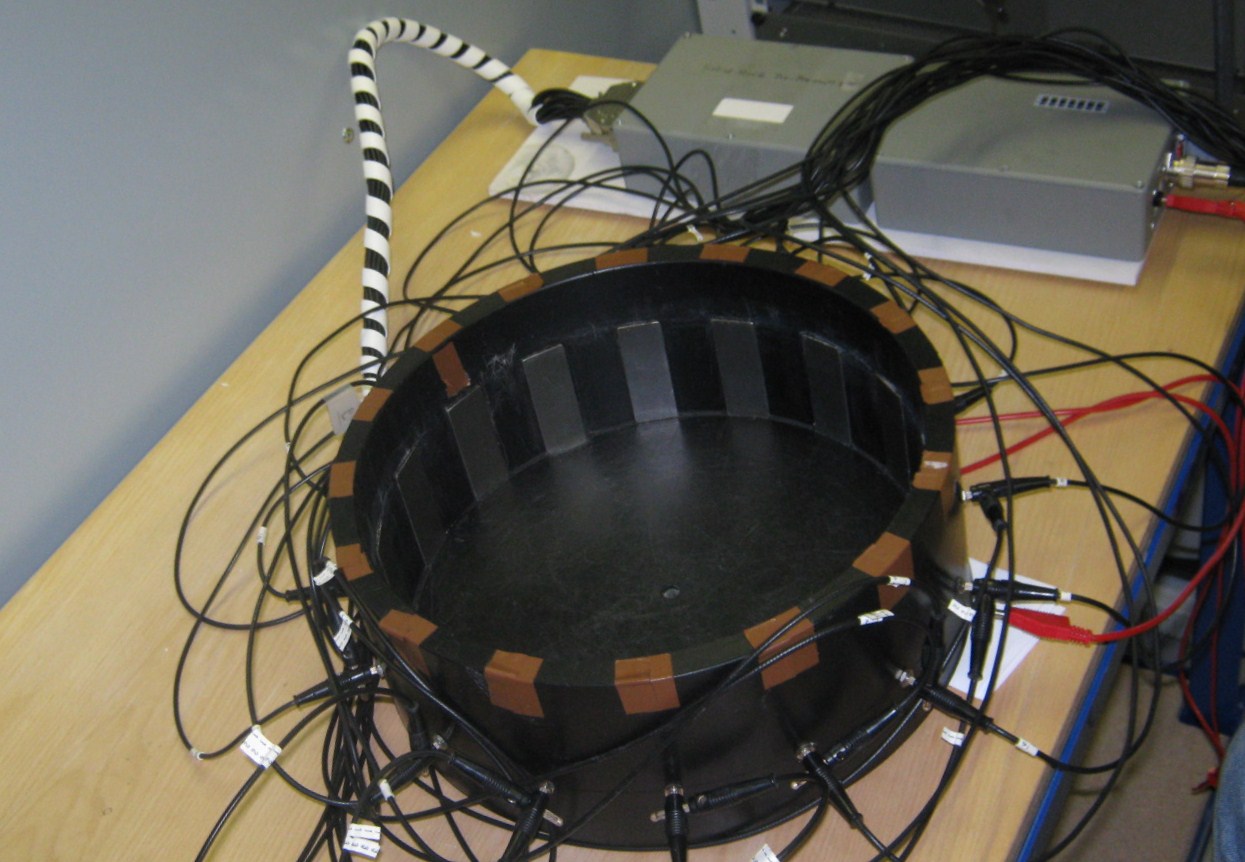} &
\includegraphics[height=5cm]{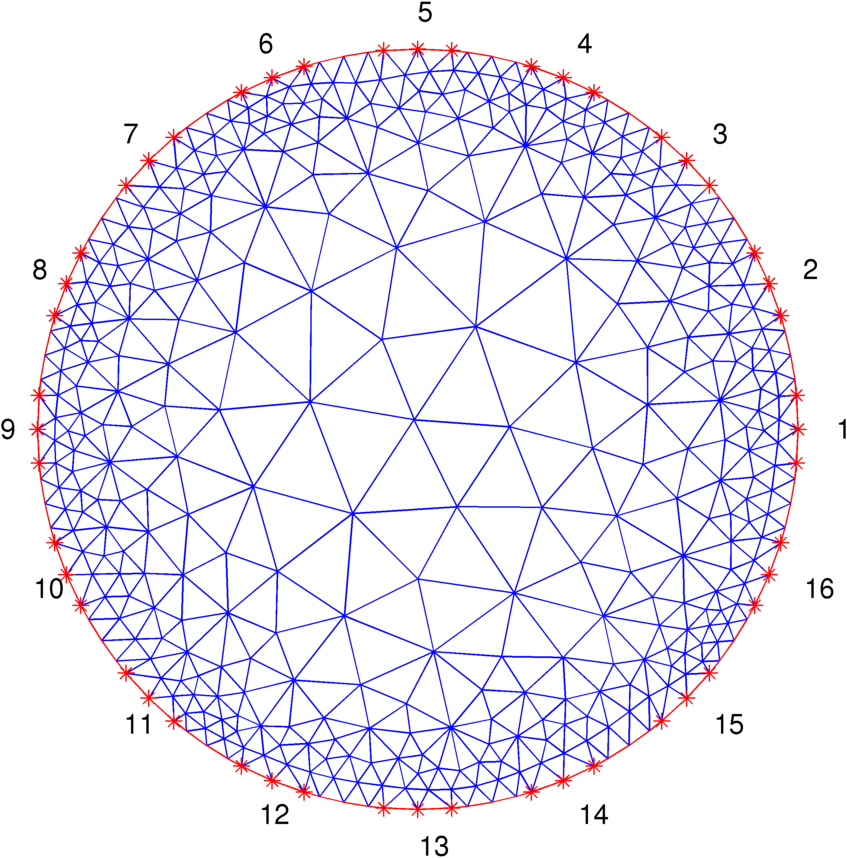}\\
experimental setup & finite element mesh
\end{tabular}
\caption{The experimental setup and the finite element mesh for forward problems.}
\label{fig:exp}
\end{figure}

The contact impedances $\{z_l\}_{l=1}^L$ and the background conductivity
$\sigma^\text{bg}$ were determined from a separate experiment. That is, we carried out additional EIT
measurements using tank filled solely with tap water. The estimation of the contact
impedances $\{z_l\}_{l=1}^L$ and the constant background conductivity $\sigma^\text{bg}$
is a well-posed problem; here, $\{z_l\}_{l=1}^L$ and $\sigma^{\rm
bg}$ were reconstructed using the standard least-squares fitting \cite{VilhunenKaipioVauhkonen:2002}.
The estimated background conductivity $\sigma^\text{bg}$ was $1.41 \cdot
10^{-3}~\Omega^{-1}$. In real (3D) conductivity units, this is $\gamma_\text{bg} =
\sigma^\text{bg}/h = 1.41 \cdot 10^{-3}~\Omega^{-1}/0.07~\mathrm{m} = 0.02
~\Omega^{-1}\mathrm{m}^{-1}$. This value is plausible -- according to literature, the
conductivity of drinking water varies between $0.0005 ~\Omega^{-1}\mathrm{m}^{-1}$ and
$0.05 ~\Omega^{-1}\mathrm{m}^{-1}$. The estimated contact impedances $\{z_l\}_{l=1}^L$ are shown
in Table \ref{tab:contactimp}. Conversion to contact impedances in real units is obtained by
multiplying the listed values by $h=0.07~\mathrm{m}$.

\renewcommand{\arraystretch}{1.2}
\setlength{\tabcolsep}{3pt}
\begin{table}[h!]
  \centering
  \caption{Estimated contact impedances $\{z_l\}$.}
  \begin{tabular}{c|cccccccccccccccc}
  \hline
  Electrode                        &1    &2   &3   &4   &5   &6   &7   &8   &9   &10  &11  &12  &13  &14  &15  & 16 \\
  $z_{l}(10^{-4}\Omega\mathrm{m})$ &2.64 &3.00&2.76&4.27&3.50&4.30&3.91&2.35&2.01&2.21&2.04&1.43&2.98&2.78&2.92&3.40\\
  \hline
  \end{tabular}\label{tab:contactimp}
\end{table}

The nonlinear EP algorithm was initialized with the estimated background conductivity $\sigma^\text{bg}$
of tap water, and the Barzilai-Borwein stepsize $\tau_k$ was backtracked as $\tau_k=\max(0,\min(\tau_k,1))$.
The parameters of the distributions were set to $\alpha = 6.9 \times10^4$ and $\lambda = 3.0\times10^4$ for all cases.
To illustrate the accuracy of the EP algorithm, we present also the results by a standard
Metropolis-Hastings MCMC algorithm (see Algorithm \ref{alg:mcmc} for the complete procedure)
with a normal random walker proposal distribution \cite{Liu:2008},
with a sample size of $1.0\times10^8$. The random walk stepsize is chosen
such that the acceptance ratio is close to 0.234, which is known to be optimal \cite{GelmanRobertsGilks:1996}.
It is well known that the convergence of MCMC methods is nontrivial to assess. Hence, we run eight chains
in parallel with random and overdispersed initial values and different random seeds, and then compute the accuracy
from these parallel chains. This enables us to directly compare the results of all chains by computing the maximum
relative error of each chain from the mean of all chains. We also compute the Brooks-Gelman statistic
\cite{BrooksGelman:1998}, which compares the within-chain variance to the between-chain variance, and
the statistics tends to one for convergent chains. In Table \ref{tab:mcmc}, we present the relative
errors of the mean and the standard deviation, and the Brooks-Gelman statistics. It is observe that
the chains for all cases, except case 1, converge reasonably well. In the first case, we could not obtain convergence,
and hence the MCMC results for case 1 should be interpreted with care. In Fig. \ref{fig:mcmc_accuracy}, we show
the results of the eight MCMC chains for case 1. The means of the means and standard deviations for all eight
chains are shown in blue, where the error bar is drawn as $1.96$ times the standard deviation componentwise,
which corresponds to the $95\%$ credible interval for a one-dimensional Gaussian distribution. The componentwise
minimum/maximum of expectation and standard  deviation are drawn in red. The means agree well with each other,
which further confirms the convergence of the chains. However, the accuracy of the standard deviation is not
very high, even with the large number of MCMC samples.

\begin{figure}[h!]
\centering
\begin{tabular}{ccc}
\includegraphics[height=3cm]{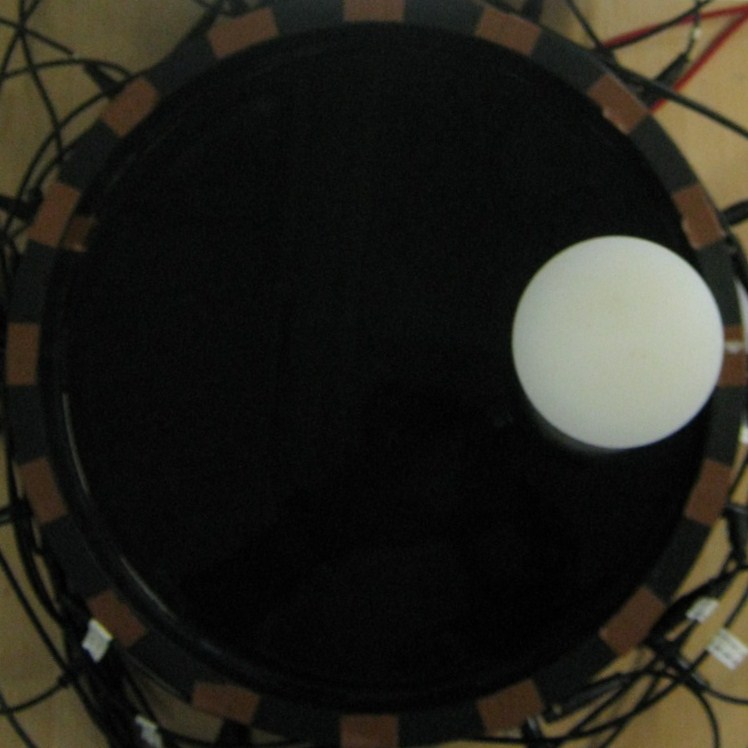}& \includegraphics[height=3cm]{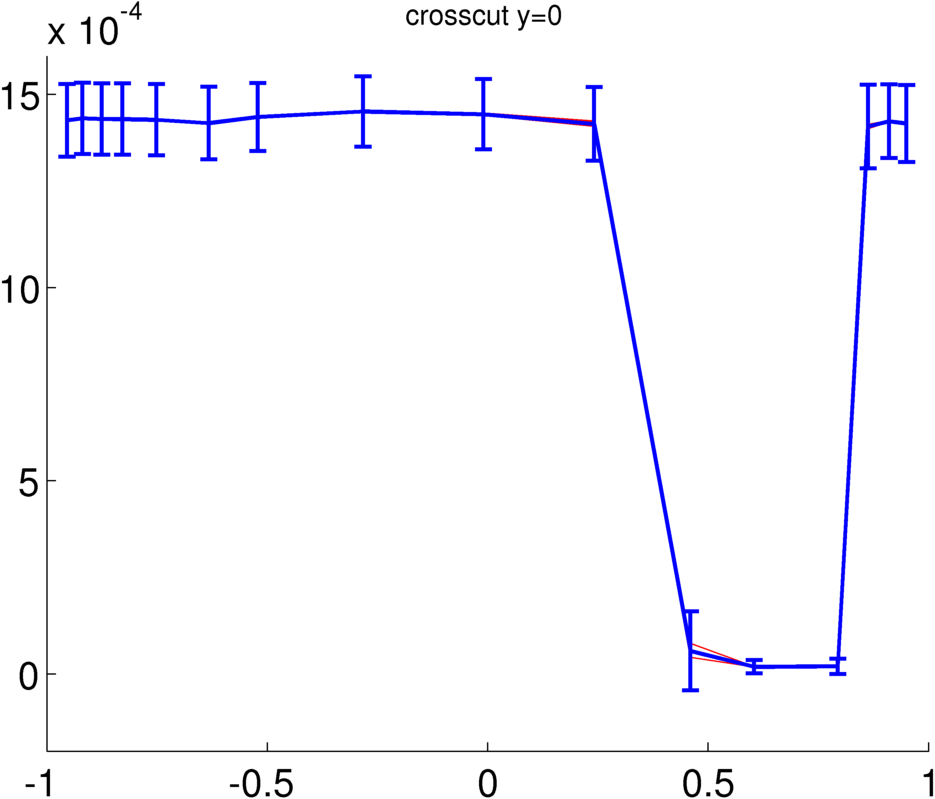}\\
\includegraphics[height=3cm]{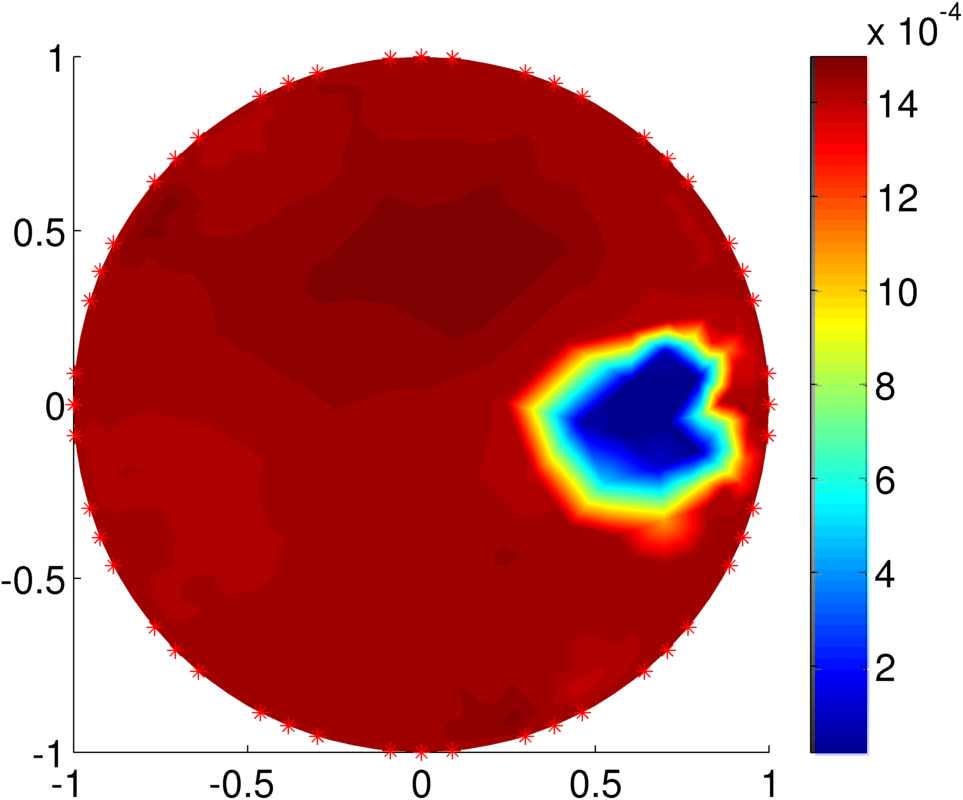}&\includegraphics[height=3cm]{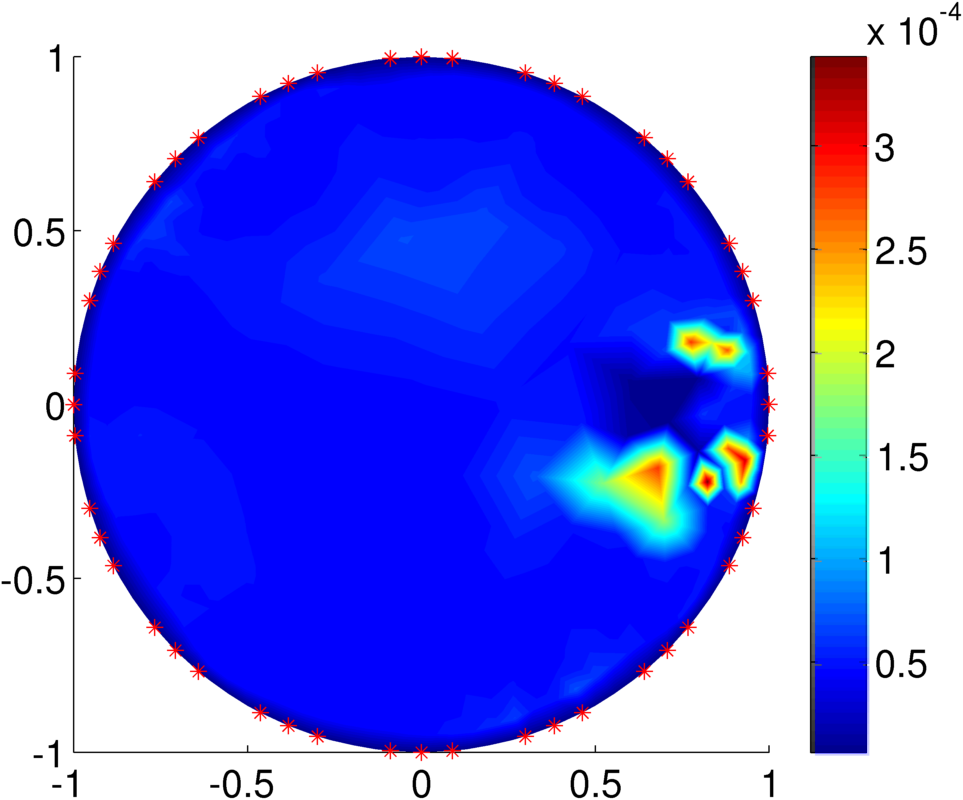}
\end{tabular}
\caption{MCMC accuracy for case 1, one plastic inclusion. Top: Photography and horizontal
cross-section with componentwise 95\% intervals. Bottom: The mean of expectation and standard deviation of all eight chains.}
\label{fig:mcmc_accuracy}
\end{figure}

\begin{table}[h!]
  \centering
  \caption{The accuracy of eight MCMC chains: relative errors in the mean
  and standard deviation, and Brooks-Gelman statistic.}
  \begin{tabular}{c|cccccc}
  \hline
  case               & 1        & 2      & 3      &4       &5       &6\\
  \hline
  mean               & 1.68e+0 &8.60e-2 &2.86e-2 &7.03e-2 &1.89e-2 &8.78e-2\\
  standard deviation & 9.72e-1 &7.72e-2 &7.30e-2 &1.53e-1 &7.35e-2 &1.51e-1\\
  Brooks-Gelman      & 561.17  &1.0069  &1.0031  &1.0045  &1.0034  &1.0139\\
  \hline
  \end{tabular}\label{tab:mcmc}
\end{table}

The reconstructions using the serial EP algorithm took about 2 minutes with up to $7$
outer iterations and a total of $22$ inner iterations on a $2.6$ GHz CPU, the MCMC algorithm
took more than one day on a $2.4$ GHz CPU. To show the convergence of the algorithm, in Fig.
\ref{fig:EP_iterates} we show the mean of the approximation $\tilde{p}$ at intermediate iterates,
where each row refers to one outer iteration.
Further, in Table \ref{tab:EPiterates} we show the changes of the mean and covariance
in the complete EP algorithm, i.e., Algorithm \ref{alg:epnonlin}. In the table, we present the 2-norm
of the difference at each iteration relative to the previous iterate and the last iterate
(the converged approximation), which are indicated as $e_{p}(\mu)$ and $e_{p}(C)$ for the error relative to
the previous iterate (respectively $e_f(\mu)$ and $e_f(C)$ for that relative to the last iterate).
It is observed from the table that the errors decrease steadily as the EP iteration proceeds, and, like
the MCMC algorithm, the convergence of the covariance is slower than that of the mean.

\begin{figure}[h!]
\centering
\begin{tabular}{cccccc}
\hline
1&\includegraphics[height=2cm]{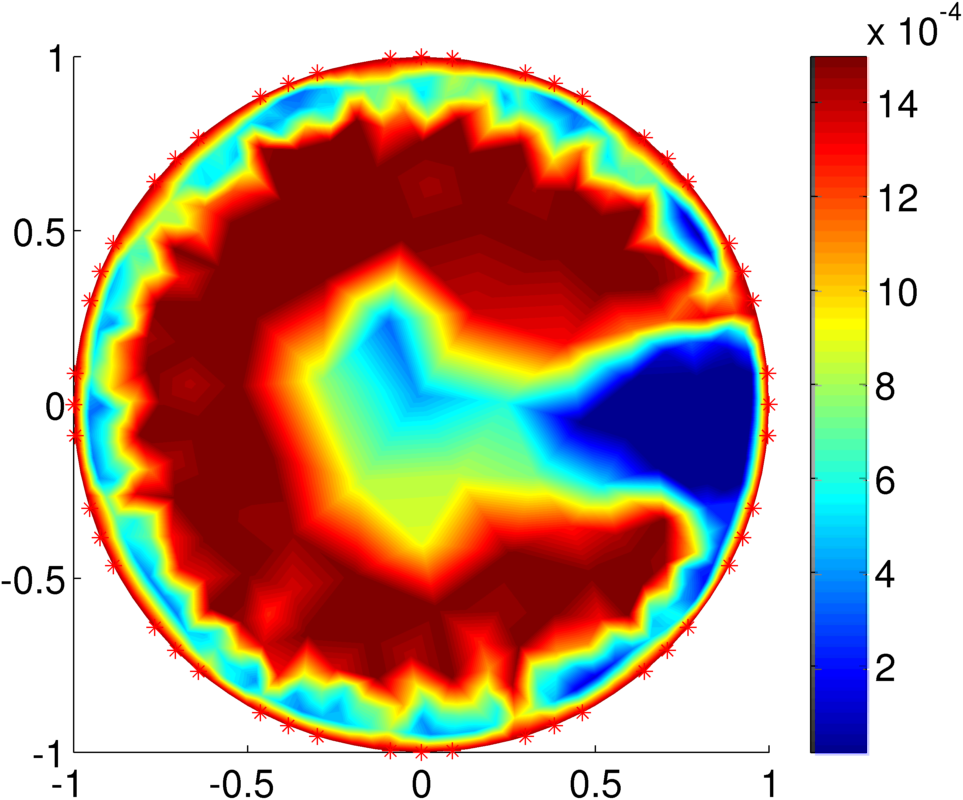}&
\includegraphics[height=2cm]{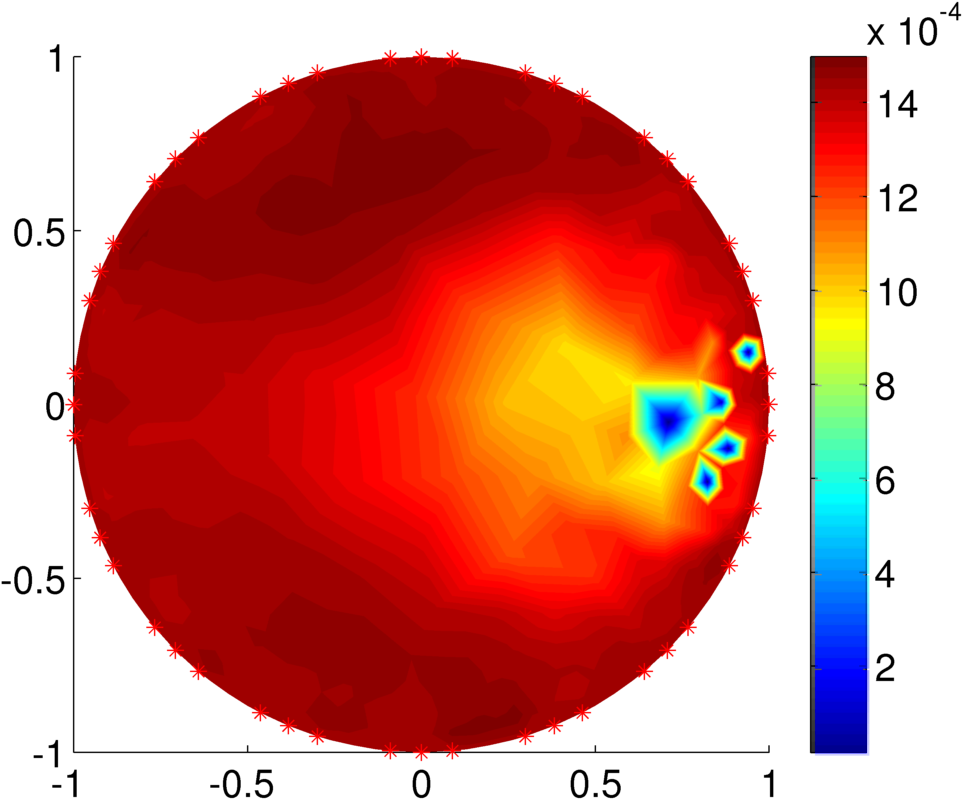}&
\includegraphics[height=2cm]{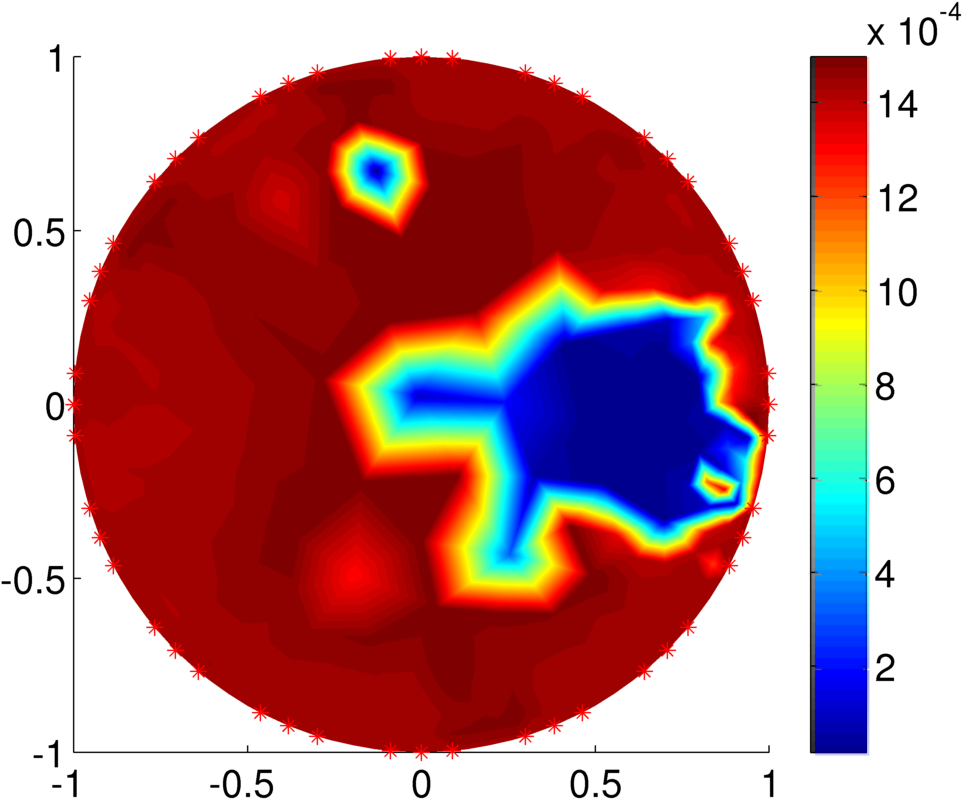}&
\includegraphics[height=2cm]{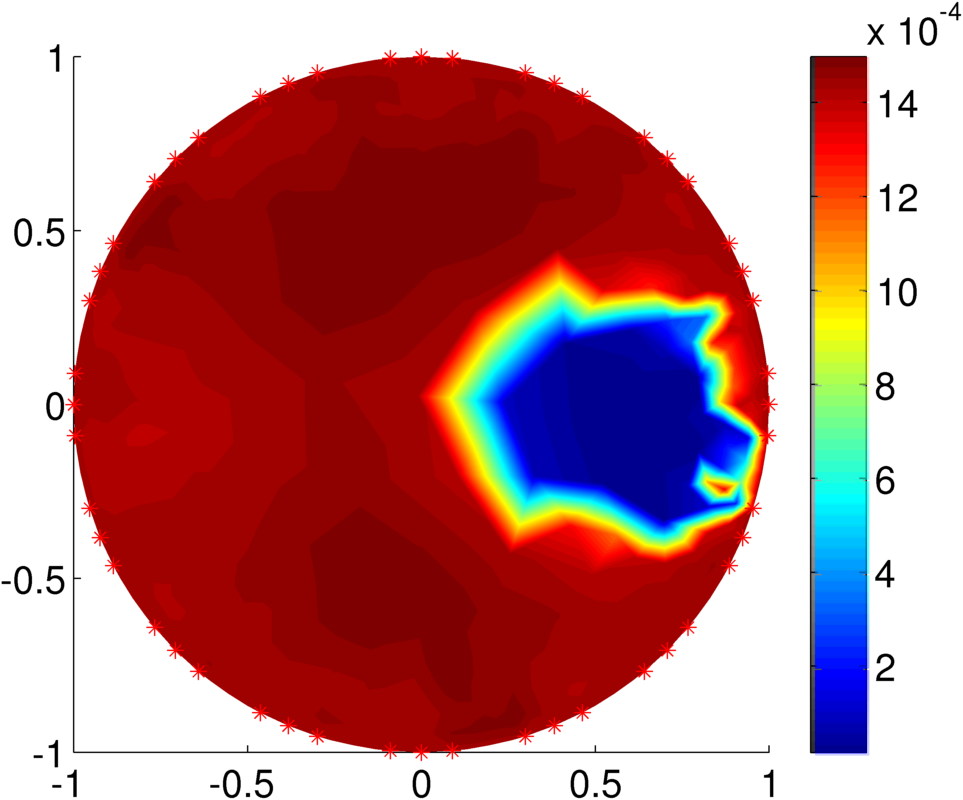}&
\includegraphics[height=2cm]{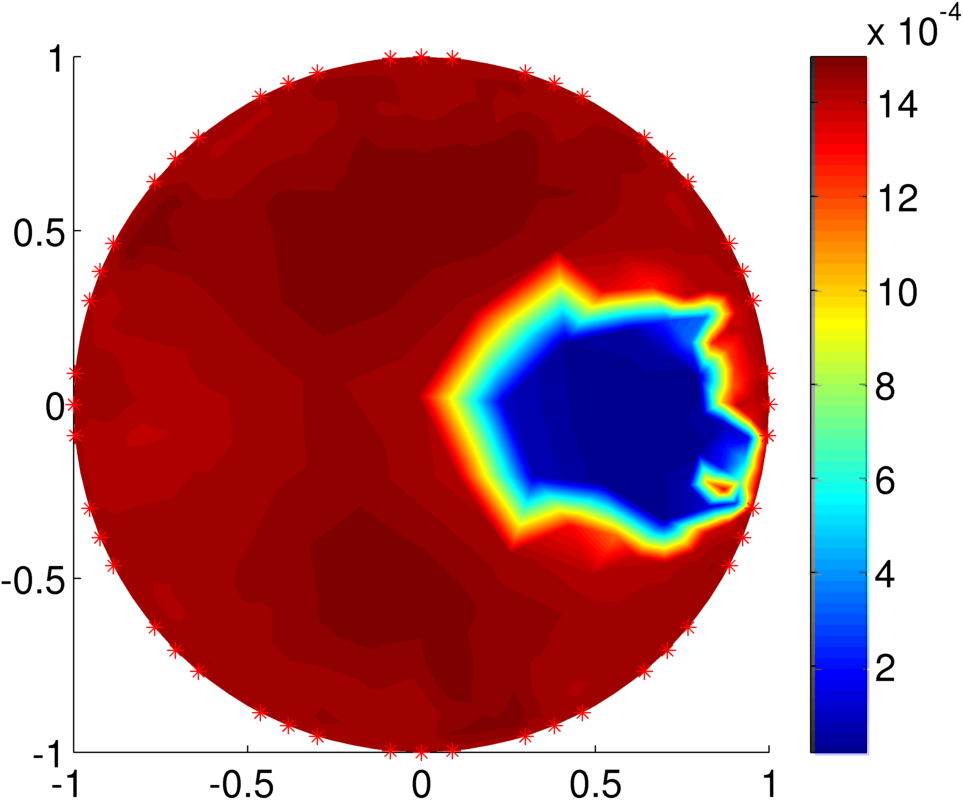}\\
2&\includegraphics[height=2cm]{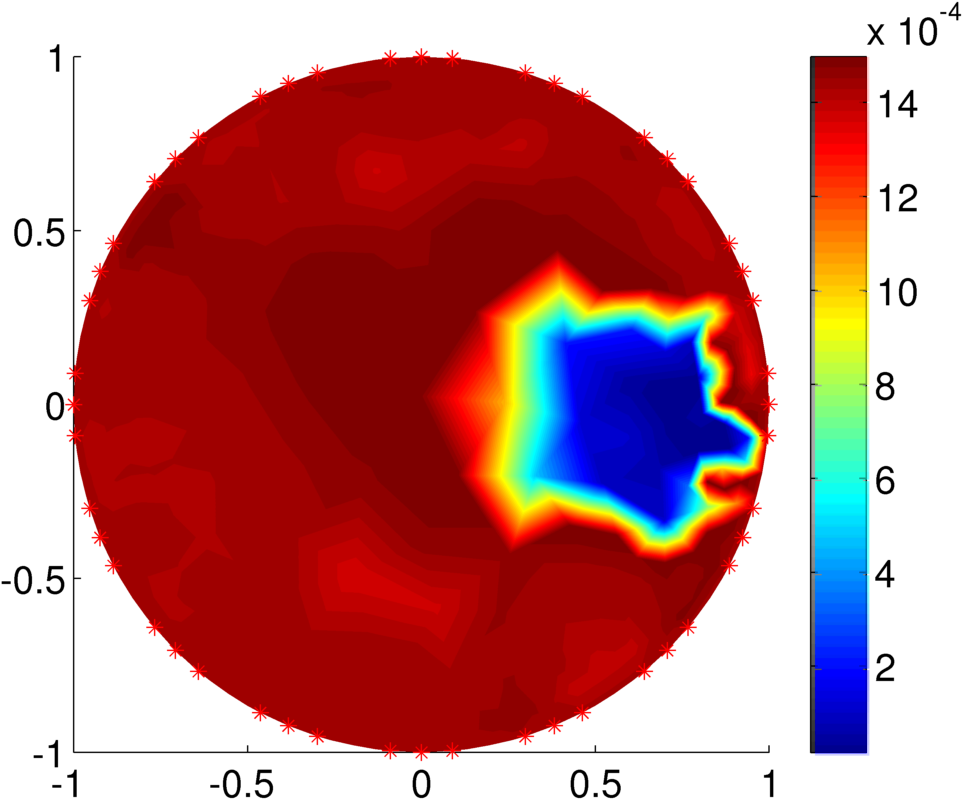}&
\includegraphics[height=2cm]{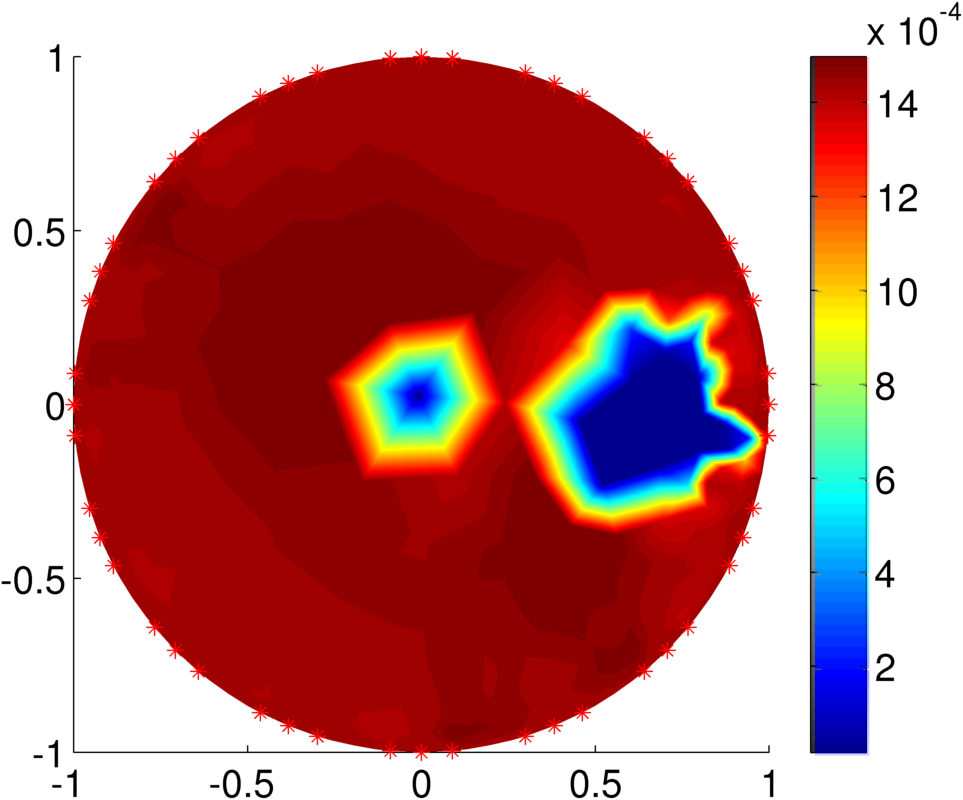}&
\includegraphics[height=2cm]{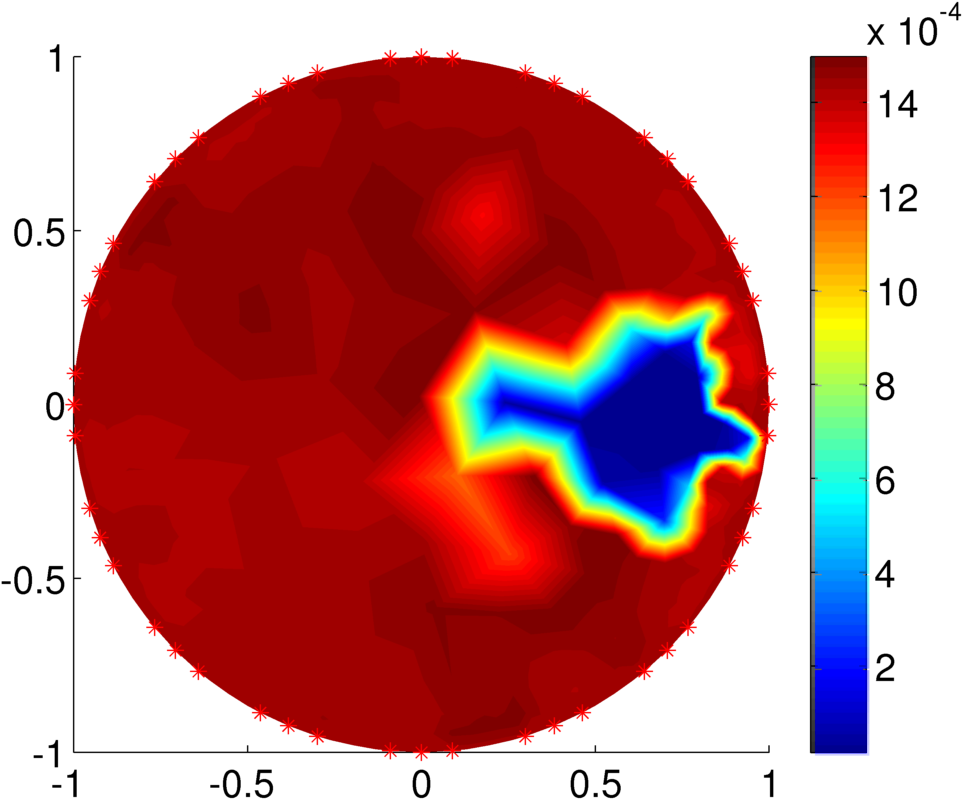}\\
4&\includegraphics[height=2cm]{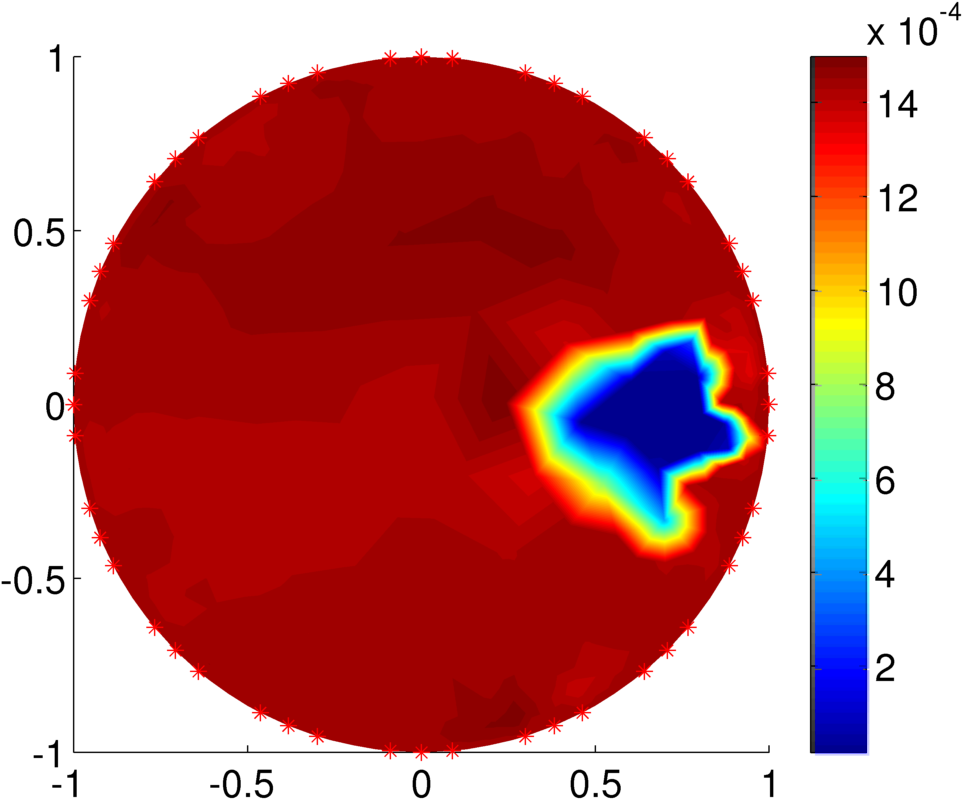}&
\includegraphics[height=2cm]{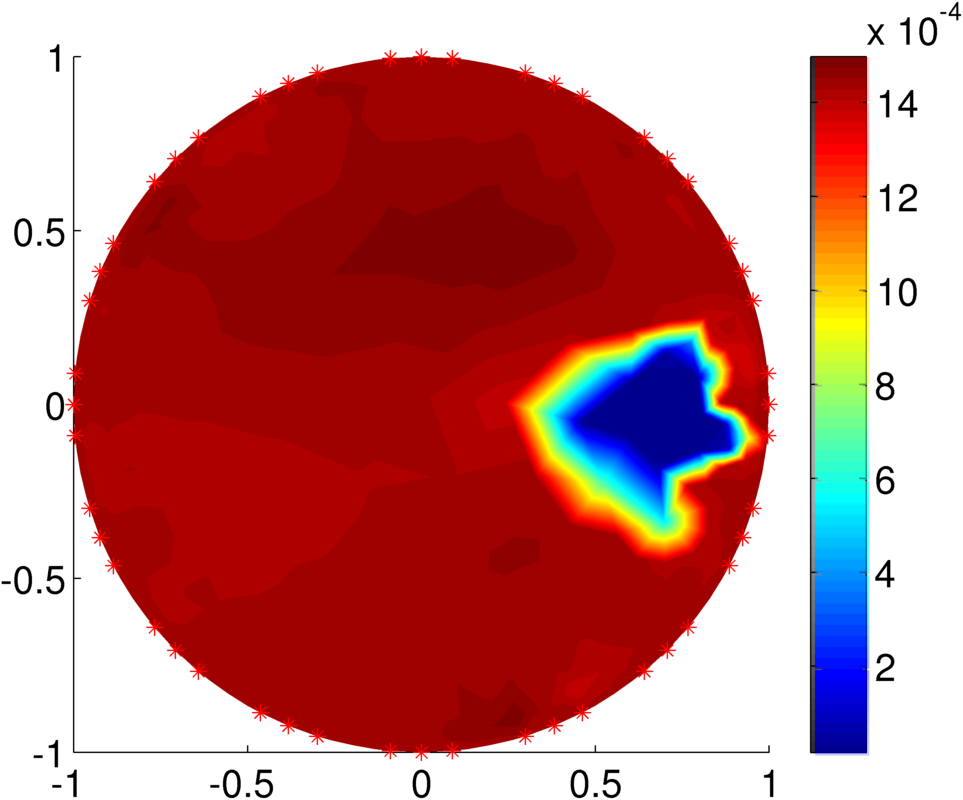}&
\includegraphics[height=2cm]{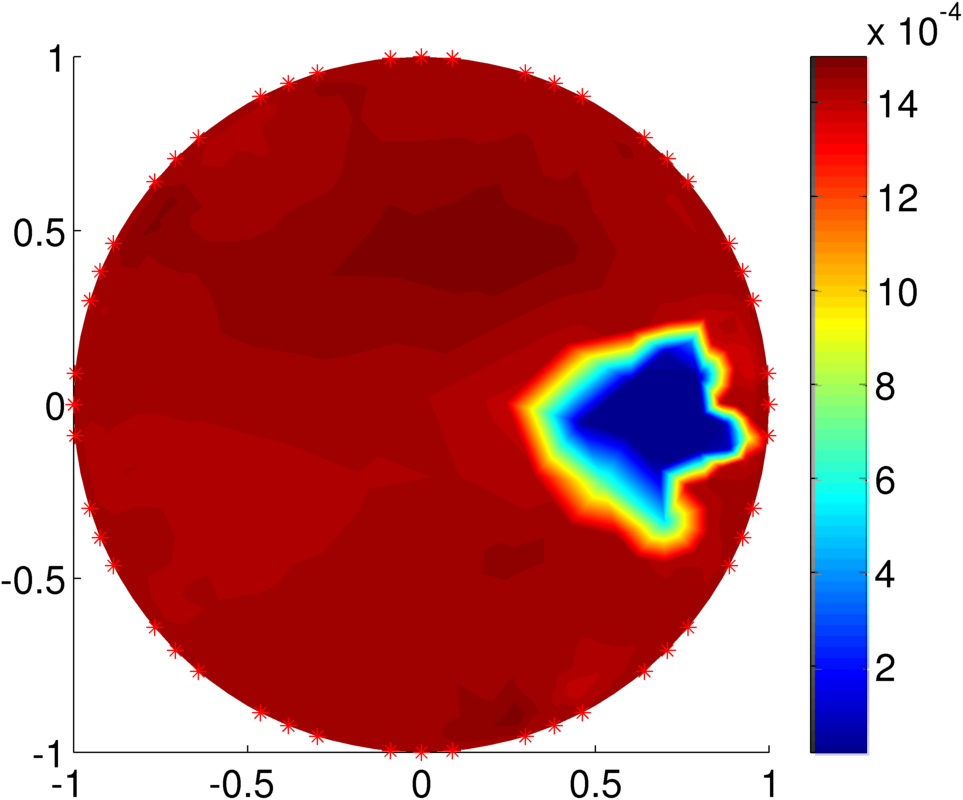}\\
5&\includegraphics[height=2cm]{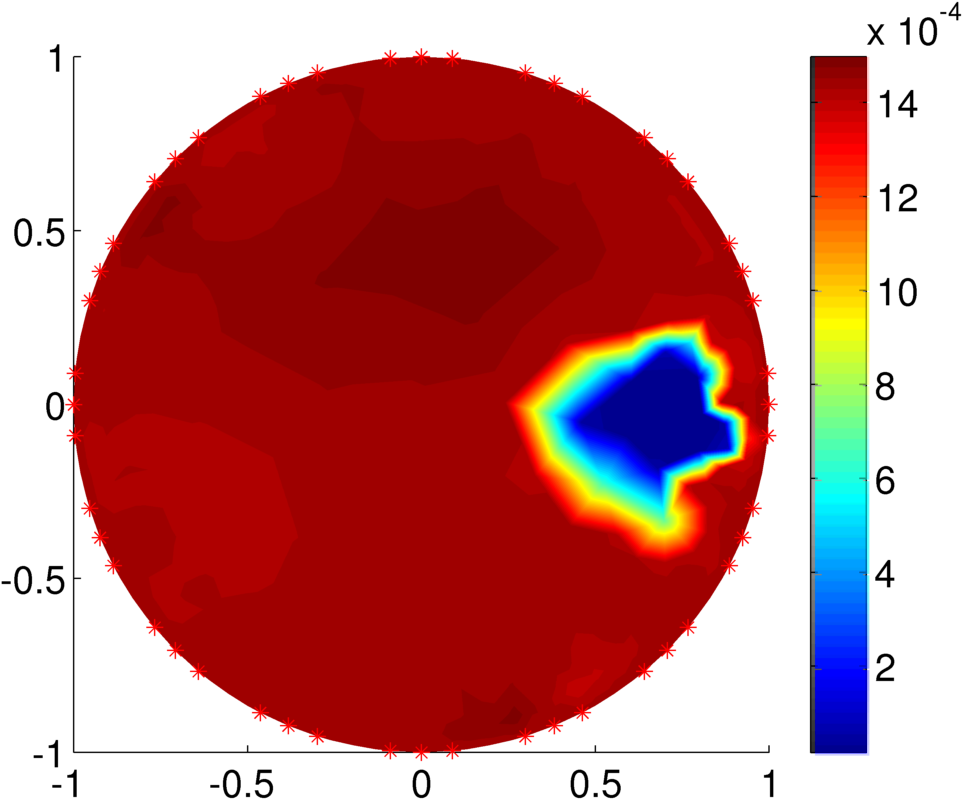}&
\includegraphics[height=2cm]{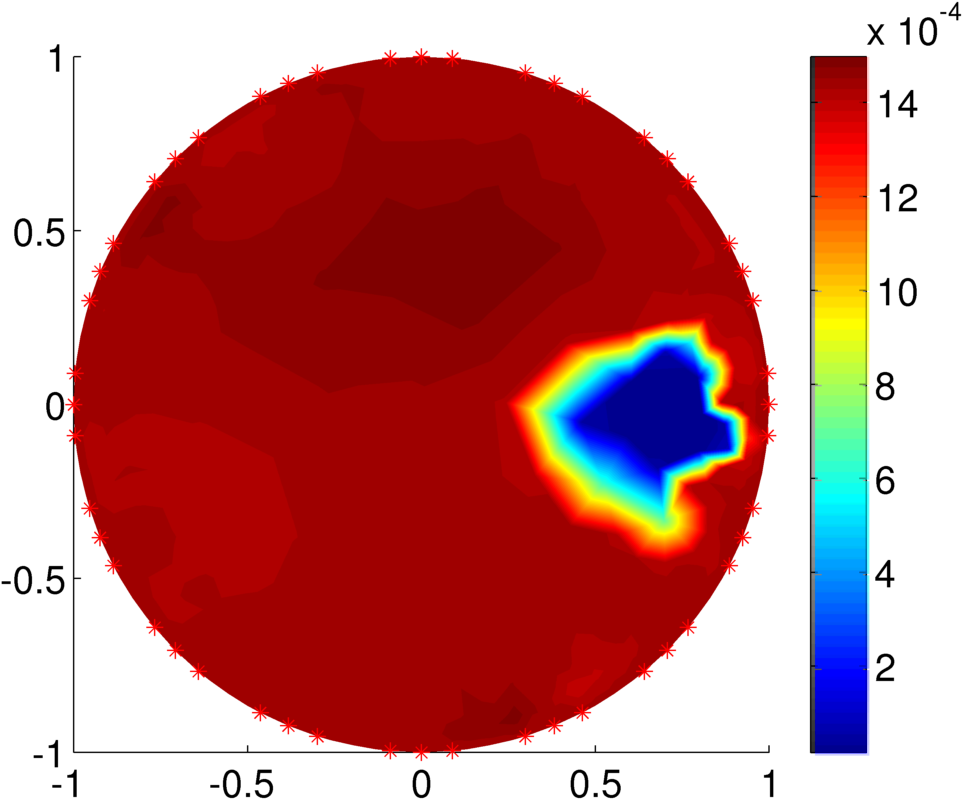}\\
\hline
$k$ & 1st EP & 2nd EP & 3rd EP & 4th EP& 5th EP\\
\hline
\end{tabular}
\caption{An illustration of the EP iterates for case 1. Plotted are the mean of each
inner EP iterate, and each row refers to one outer iteration ($k$ is the $k$th outer iteration).}
\label{fig:EP_iterates}
\end{figure}

\begin{table}[h!]
  \centering
  \caption{Convergence behavior of the EP algorithm for case 1. Here $k$ and $j$ refer respectively to the
  outer and inner iteration index, $e_{p}(\mu)$ and $e_{p}(C)$ denote the error relative of the mean $\mu$ and
  covariance $C$ relative to the previous iterate, and $e_f(\mu)$ and $e_f(C)$ relative to the last iterate. }
  \begin{tabular}{c|c|cccc}
  \hline
     $k$& $j$ & $e_{p}(\mu)$ & $e_f(\mu)$ & $e_{p}(C)$  & $e_f(C)$\\
  \hline
   &1 &6.95e-1 &6.21e-1 &1.87e1  &1.87e1\\
   &2 &7.23e-1 &4.74e-1 &1.87e1  &1.82e0\\
  1 &3&5.29e-1 &2.53e-1 &6.40e0  &6.36e0\\
   &4 &1.71e-1 &1.78e-1 &5.57e0  &3.69e0\\
   &5 &1.98e-3 &1.78e-1 &6.76e-1 &3.54e0\\
   \hline
   &1 &2.40e-1 &2.26e-1 &1.15e0  &4.03e0\\
   &2 &2.58e-1 &1.71e-1 &3.96e0  &2.55e0\\
  2 &3&1.60e-1 &1.04e-1 &2.60e0  &3.07e0\\
   &4 &2.86e-2 &9.63e-2 &2.54e0  &3.56e0\\
   &5 &1.58e-3 &9.63e-2 &2.75e-1 &3.67e0\\
      \hline
   &1 &1.46e-1 &1.21e-1 &1.23e0  &3.85e0\\
   &2 &1.07e-1 &5.74e-2 &3.54e0  &8.67e-1\\
  3 &3&4.45e-2 &3.39e-2 &6.83e-1 &7.92e-1\\
   &4 &2.32e-2 &4.47e-2 &4.48e-1 &4.62e-1\\
   &5 &4.39e-3 &4.34e-2 &3.06e-1 &7.06e-1\\
      \hline
   &1 &3.11e-2 &2.36e-2 &1.04e-1 &6.97e-1\\
  4 &2&1.11e-2 &1.74e-2 &4.56e-1 &2.90e-1\\
   &3 &1.91e-3 &1.79e-2 &2.17e-1 &3.35e-1\\
     \hline
  5 &1&1.80e-2 &1.28e-3 &1.94e-1 &1.58e-1\\
   &2 &1.28e-3 &0.00e0  &1.58e-1 &0.00e0\\
  \hline
  \end{tabular}\label{tab:EPiterates}
\end{table}

The results are shown in Figs. \ref{fig:exp3}-\ref{fig:exp16}. In the figures, the left column
shows a photography of the experiment, i.e., the watertank with the inclusions.  The middle and right
columns present the reconstructions by the EP algorithm and the MCMC method, respectively; the first
and second rows present the posterior mean and the diagonal of the posterior standard deviation,
respectively.

In all the experiments, EP and MCMC usually show a very good (mostly indistinguishable) match on the
mean. For the variance, EP usually captures the same structure as MCMC, but the magnitudes are slightly
different. The variance generally gets higher towards the center of the watertank (due to less
information encapsulated in the measurements), and at the edges of detected inclusions (due to
uncertainity in the edge's exact position). Further, compared with cases 1-4, the standard deviation of
cases 5-6, which contains metallic bars, is much larger, i.e., the associated uncertainties are larger.

Finally, we observe that the Bayesian reconstructions with the Laplace prior contain
many close-to-zero components, however it is not truly sparse but only approximately sparse.
This is different from the Tikhonov formulation, which yields a genuinely sparse reconstruction,
as long as the regularization parameter is large enough; see the reconstructions in
\cite{GehreKluthLipponenJinSeppanenKaipioMaass:2012}.
While the Tikhonov formulation yields only a point estimate, the EP reconstruction gives information
about the mean, thus a weighted superposition of all possible reconstructions.
The visually appealing sparsity of a Tikhonov minimizer can be misleading, because it only shows one solution
in an ensemble of very different but (almost) equally probable reconstructions.

\begin{figure}
\centering
\begin{tabular}{ccc}
\includegraphics[height=3cm]{exp3.jpg}& \includegraphics[height=3cm]{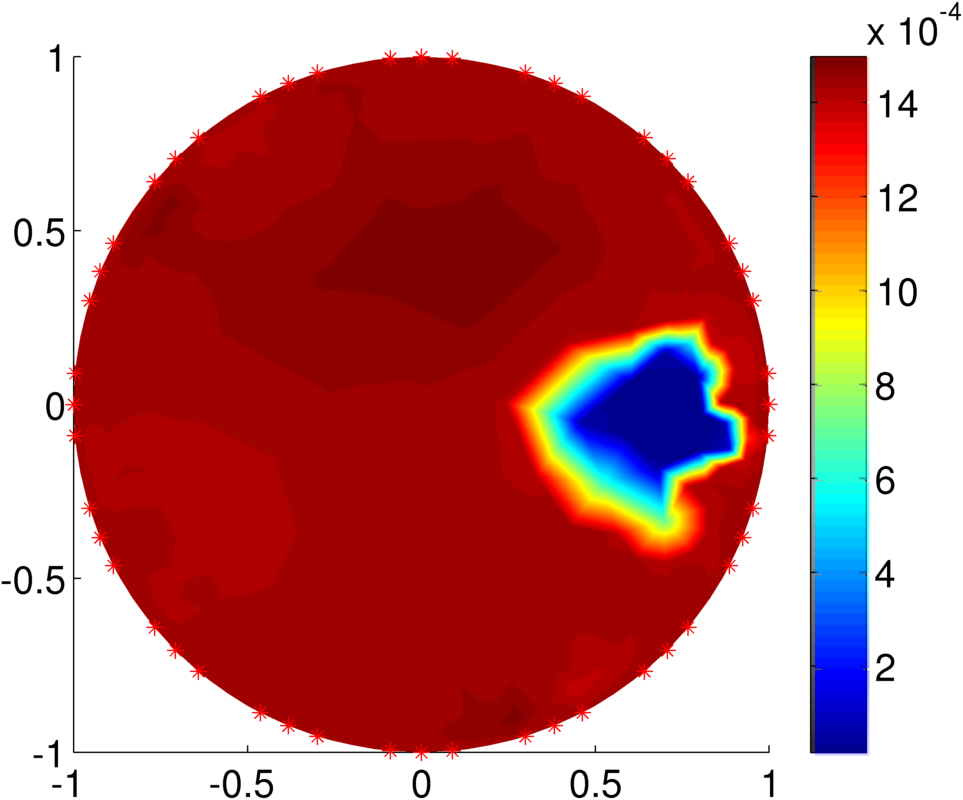} & \includegraphics[height=3cm]{cem-exp3-mcmc-exp.png}\\
        & \includegraphics[height=3cm]{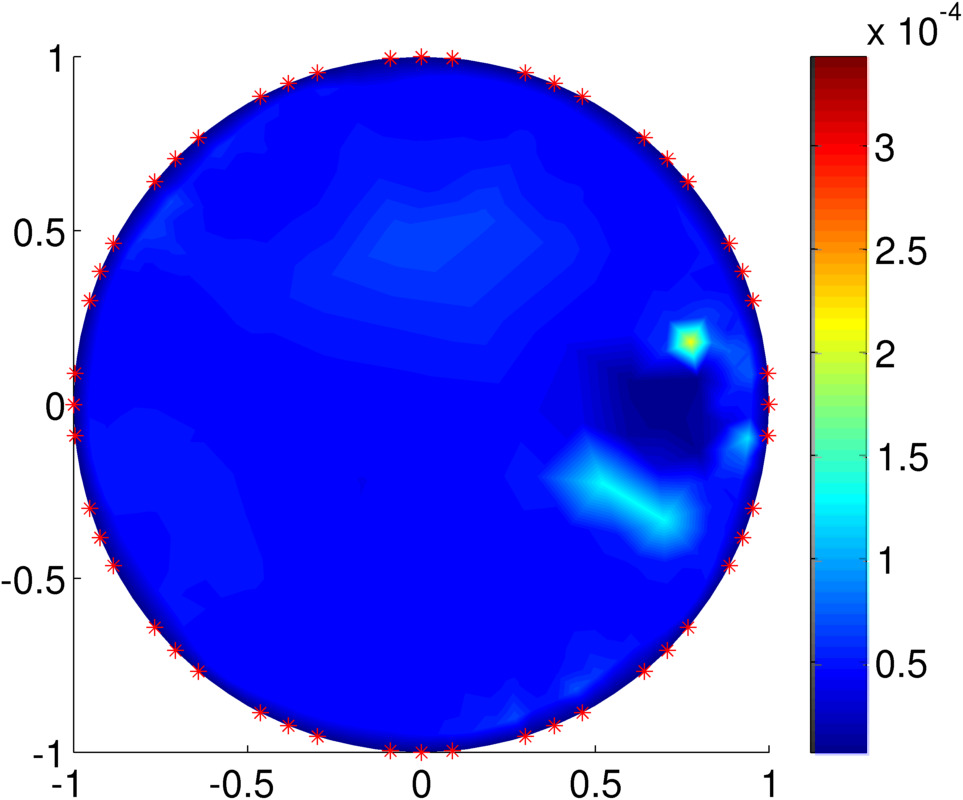} & \includegraphics[height=3cm]{cem-exp3-mcmc-std.png}\\
        & \includegraphics[height=3cm]{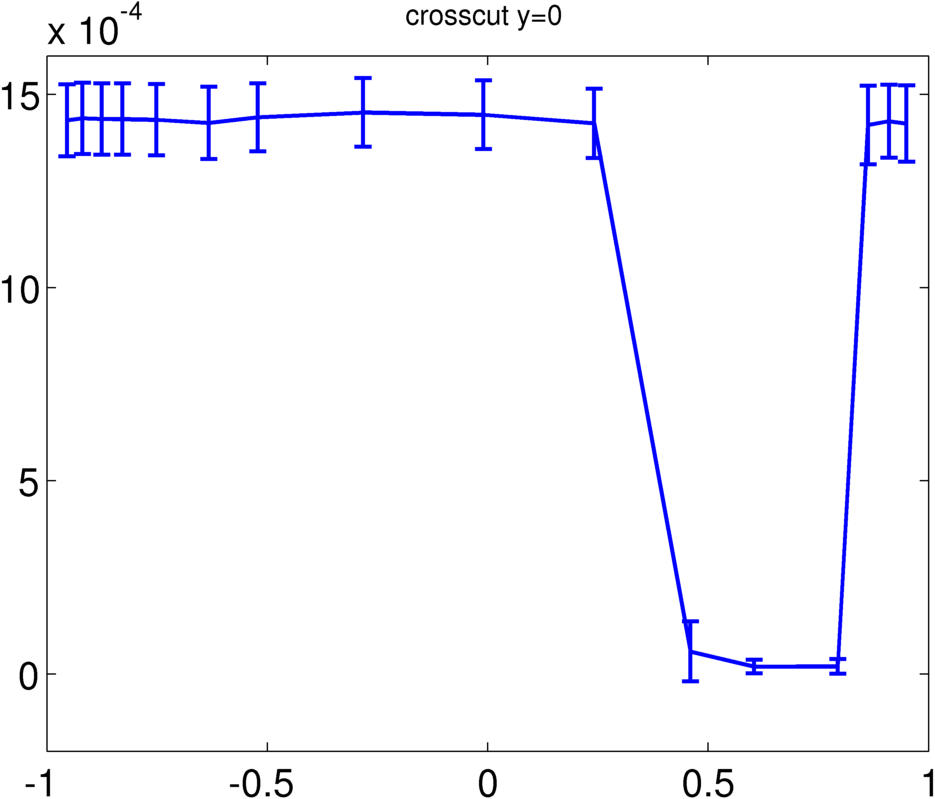} & \includegraphics[height=3cm]{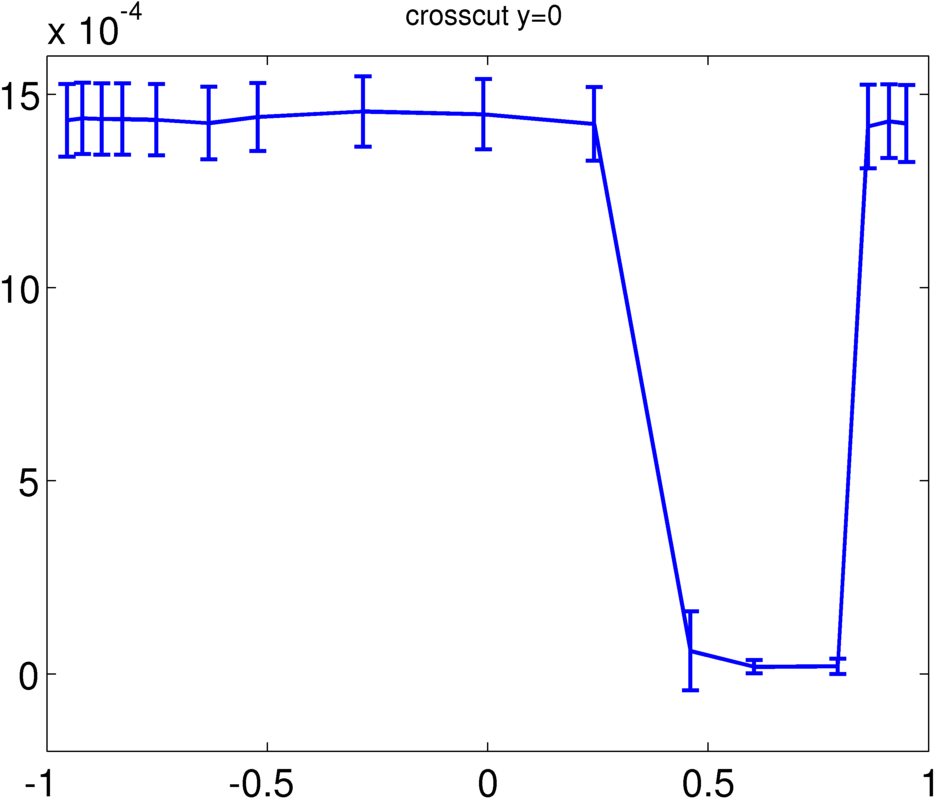}\\
        & EP  & MCMC
\end{tabular}
\caption{Numerical results for case 1, one plastic inclusion.}
\label{fig:exp3}
\end{figure}

\begin{figure}
\centering
\begin{tabular}{ccc}
\includegraphics[height=3cm]{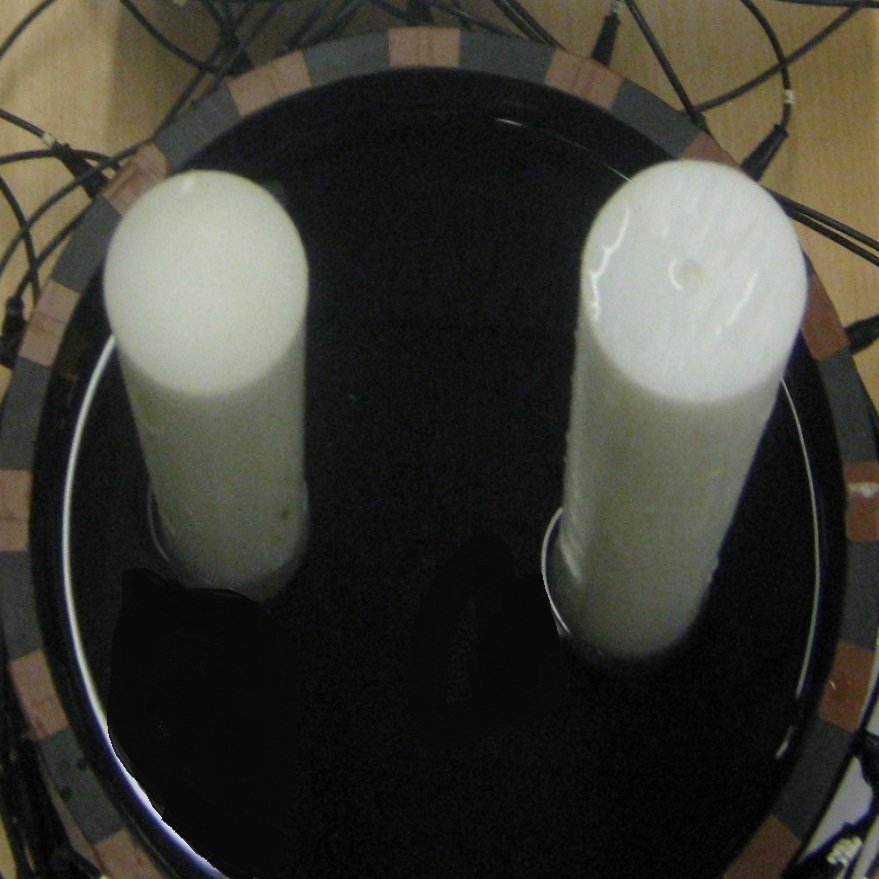} & \includegraphics[height=3cm]{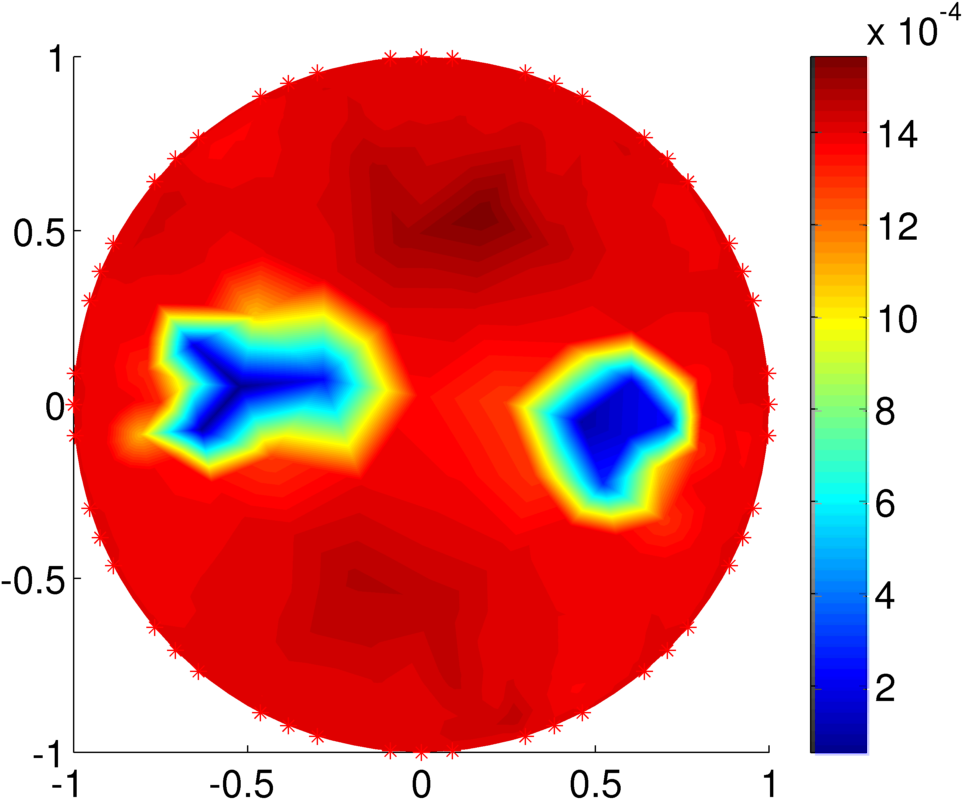} & \includegraphics[height=3cm]{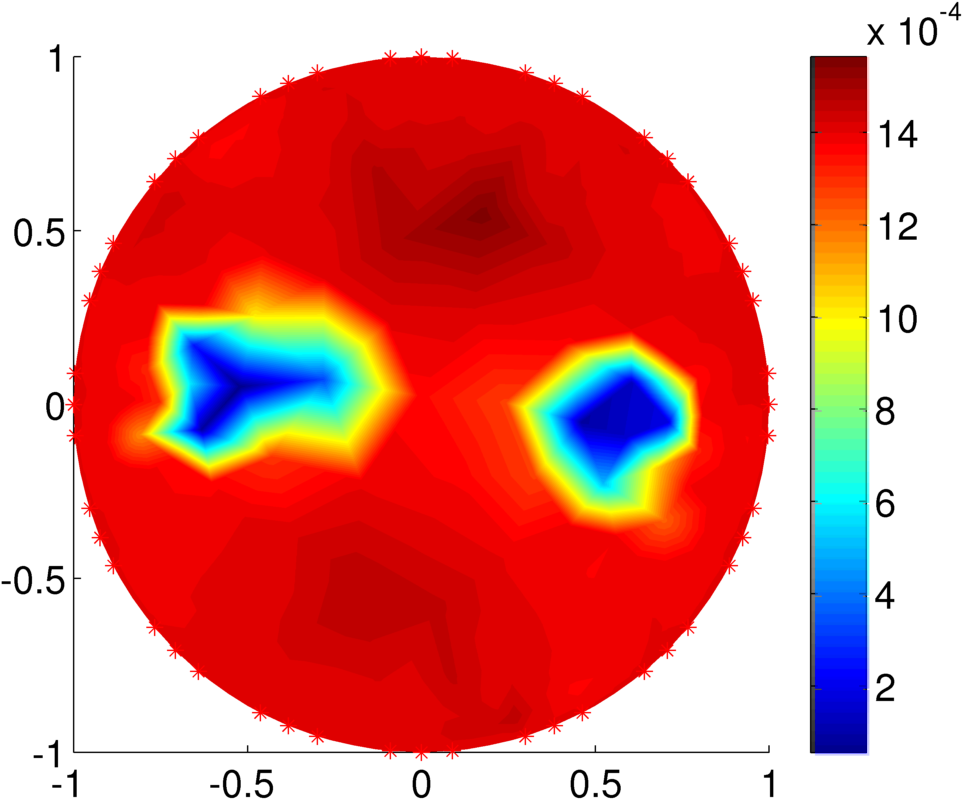}\\
         & \includegraphics[height=3cm]{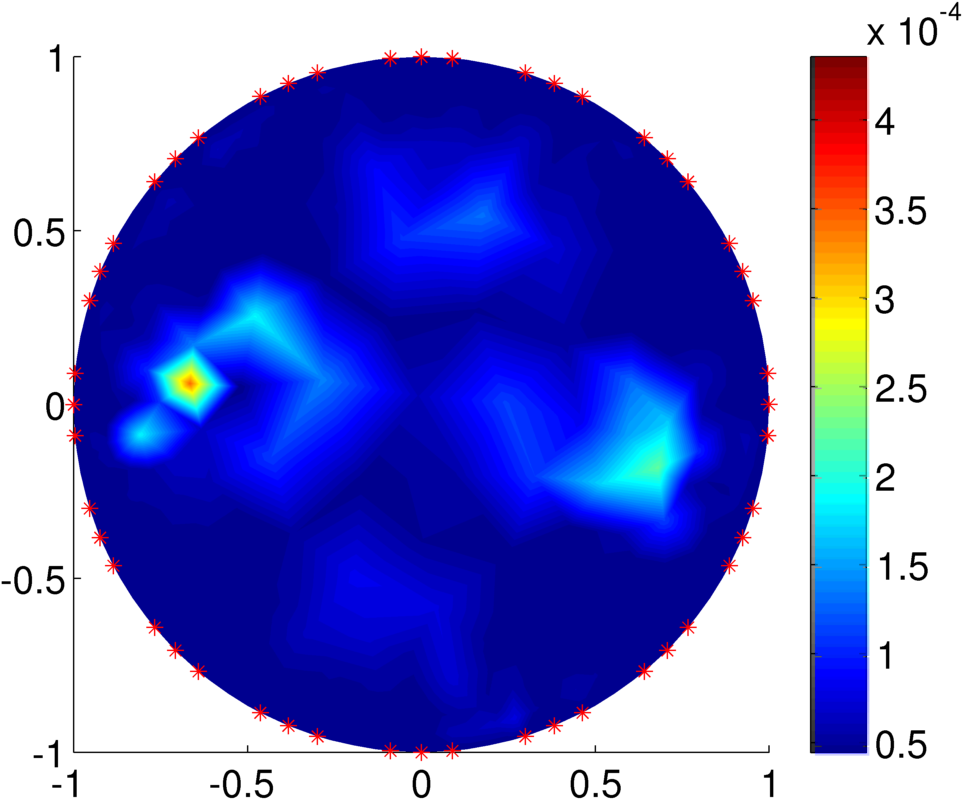} &  \includegraphics[height=3cm]{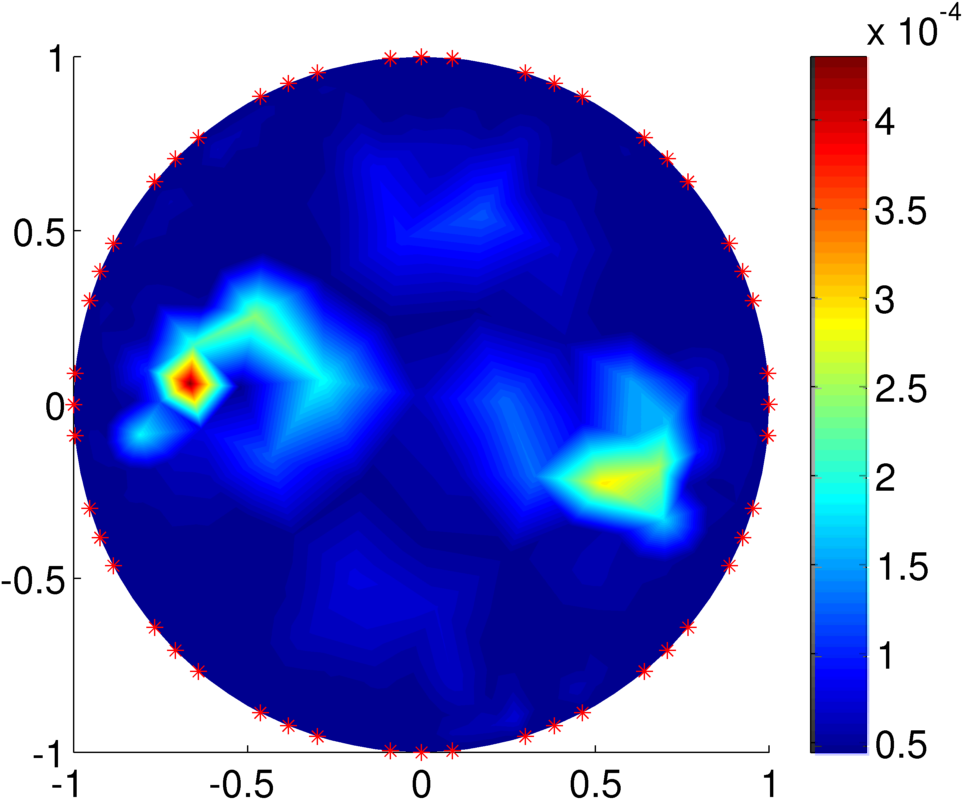}\\
         & \includegraphics[height=3cm]{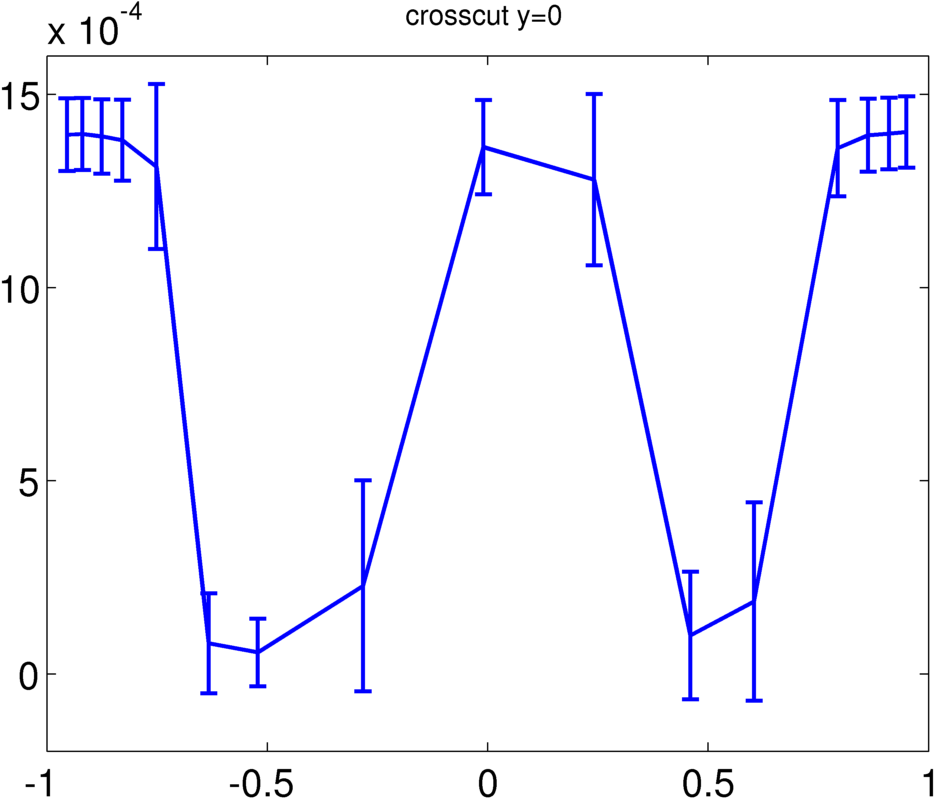} & \includegraphics[height=3cm]{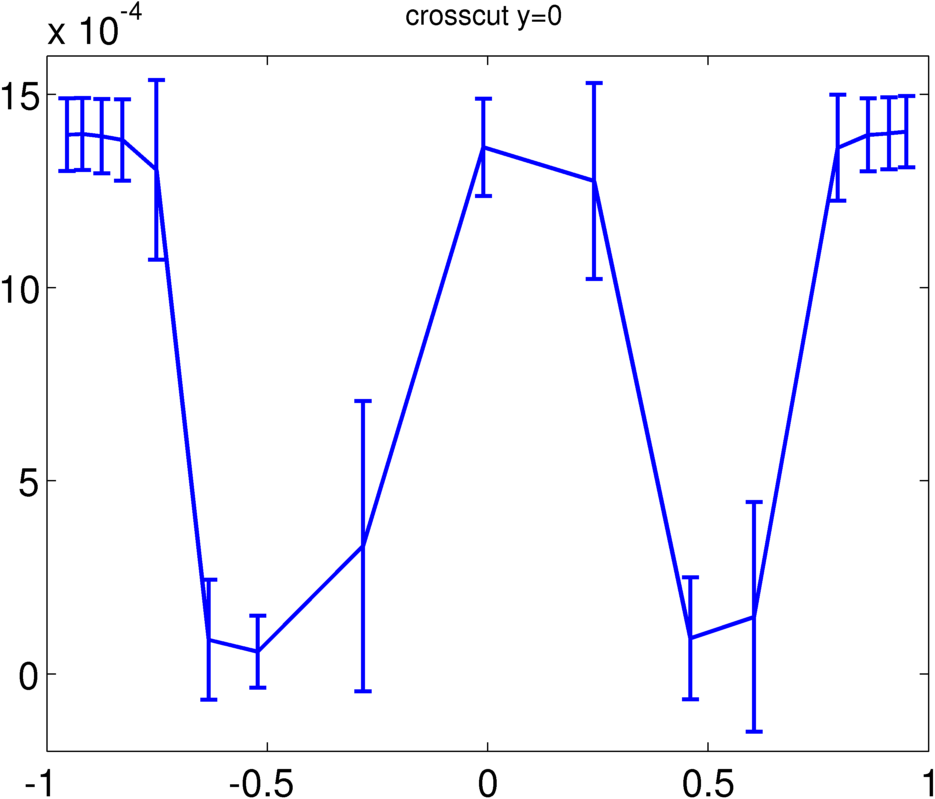}\\
         & EP & MCMC
\end{tabular}
\caption{Numerical results for case 2, two plastic bars.}
\label{fig:exp5}
\end{figure}

\begin{figure}
\centering
\begin{tabular}{ccc}
\includegraphics[height=3cm]{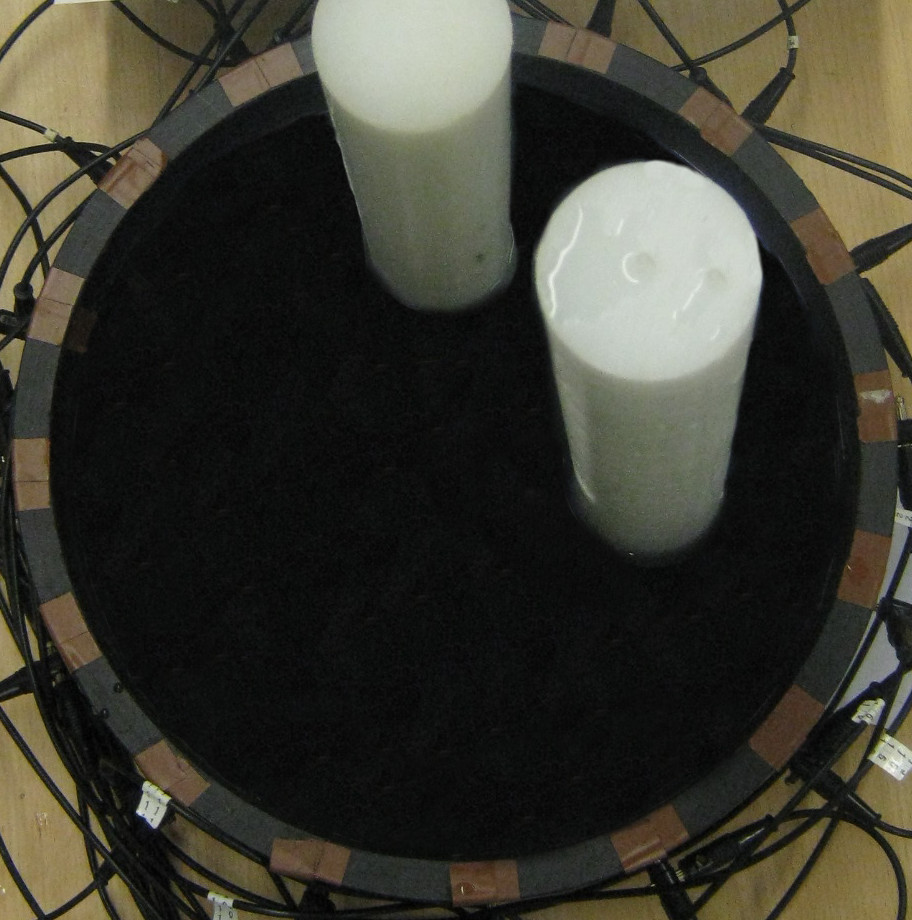} & \includegraphics[height=3cm]{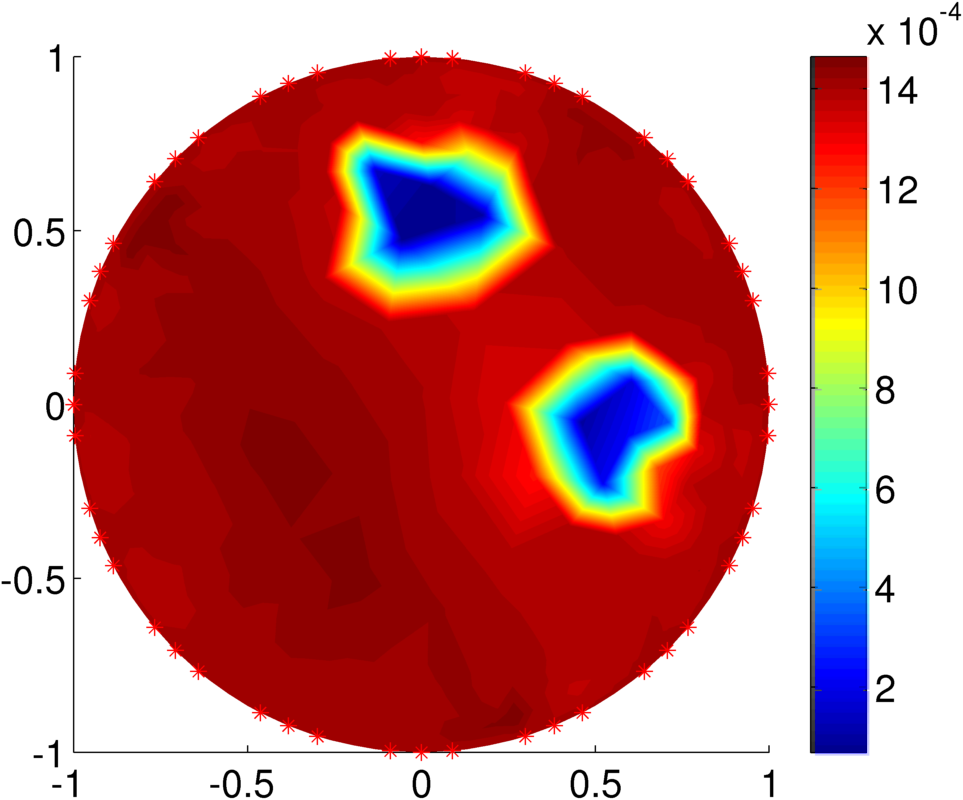} & \includegraphics[height=3cm]{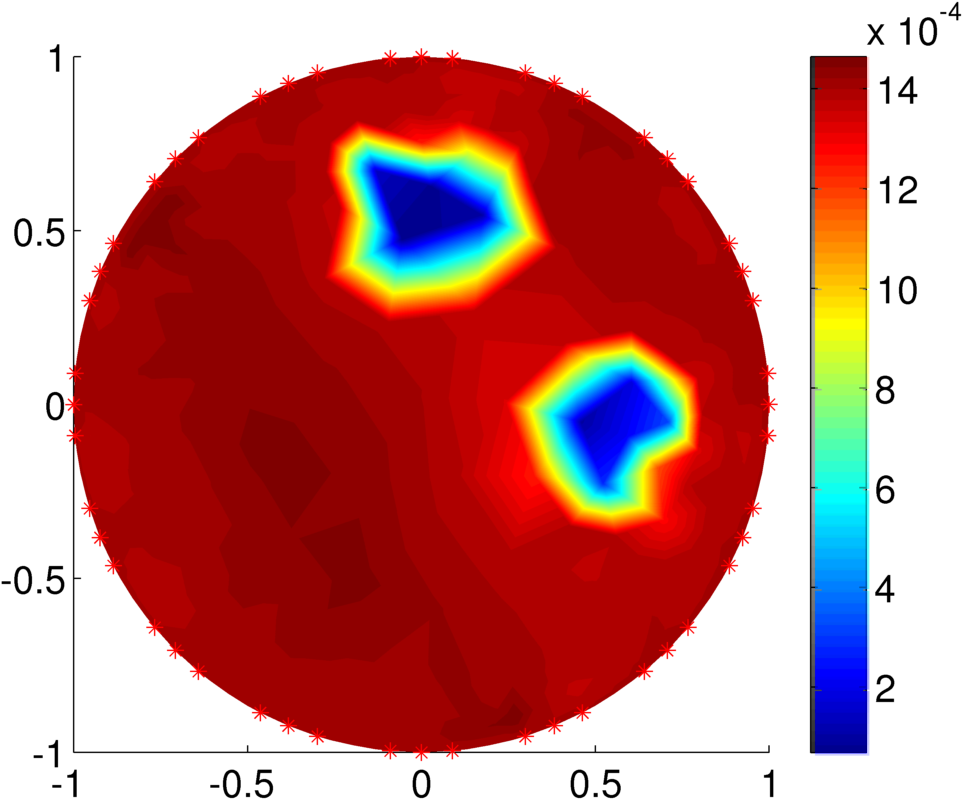}\\
       & \includegraphics[height=3cm]{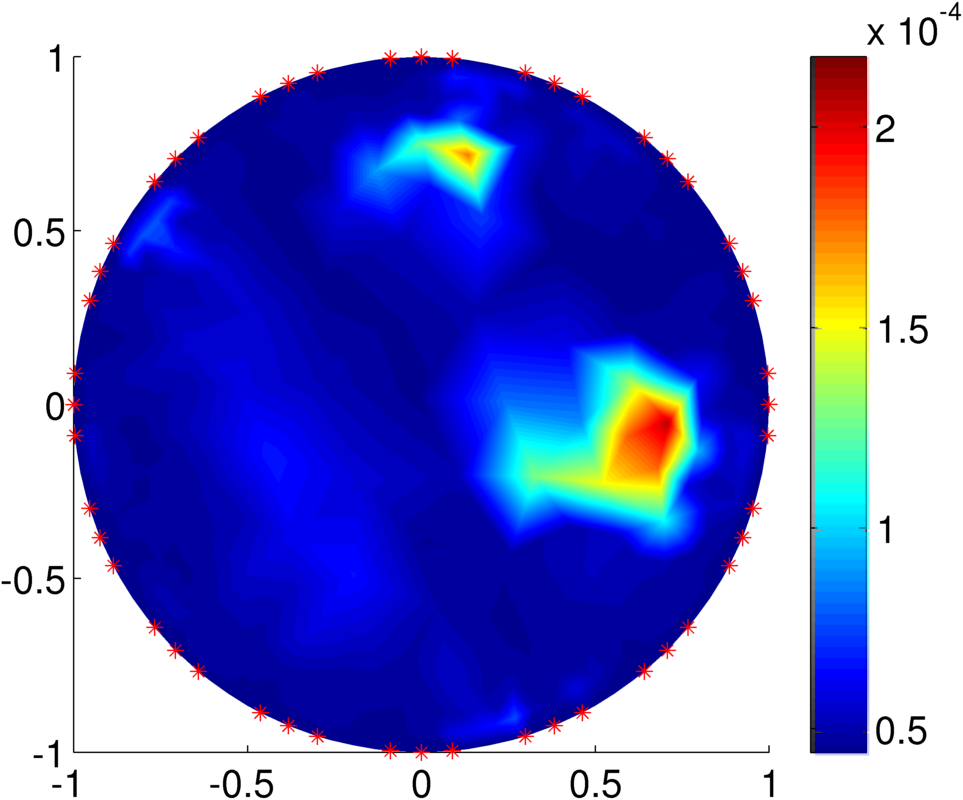}  & \includegraphics[height=3cm]{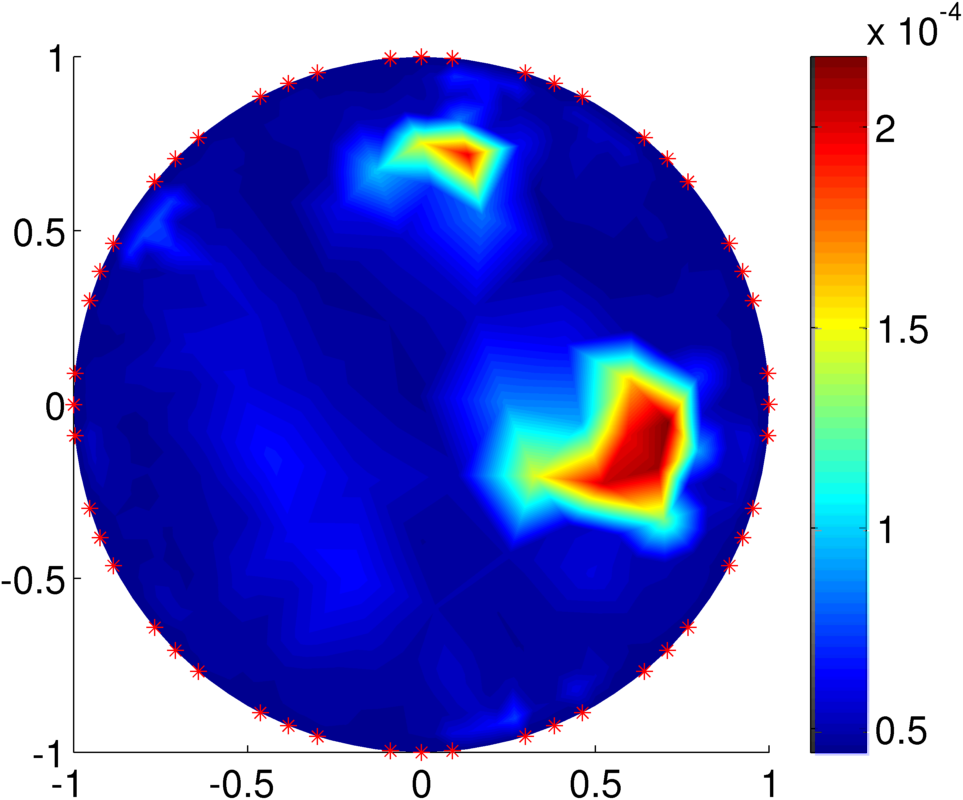}\\
       & \includegraphics[height=3cm]{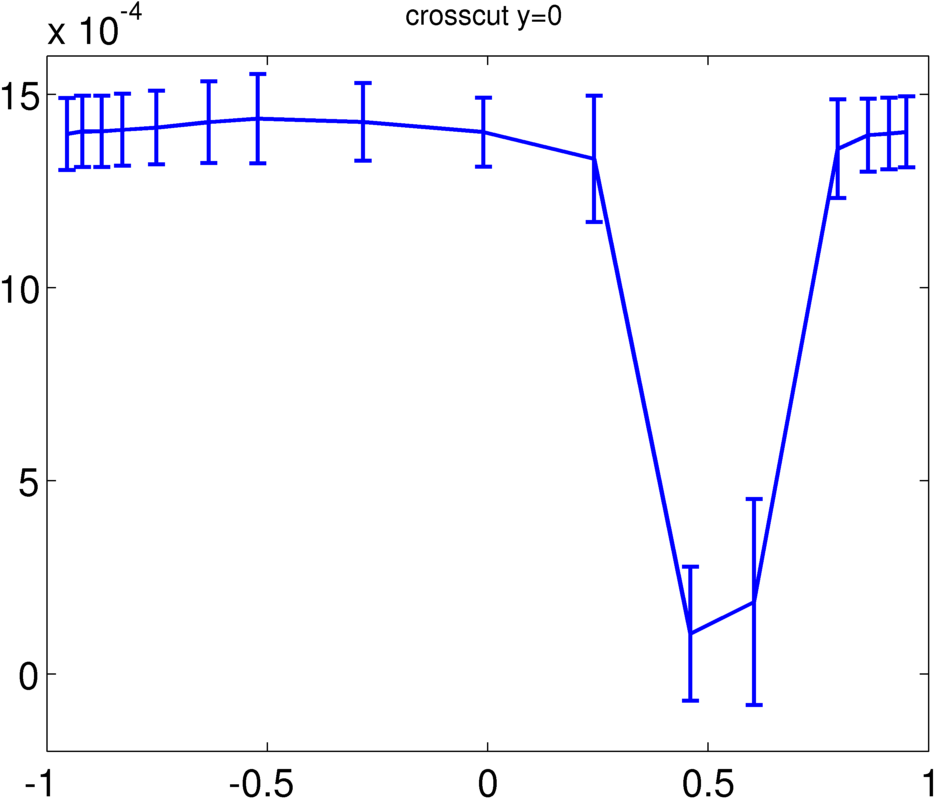} & \includegraphics[height=3cm]{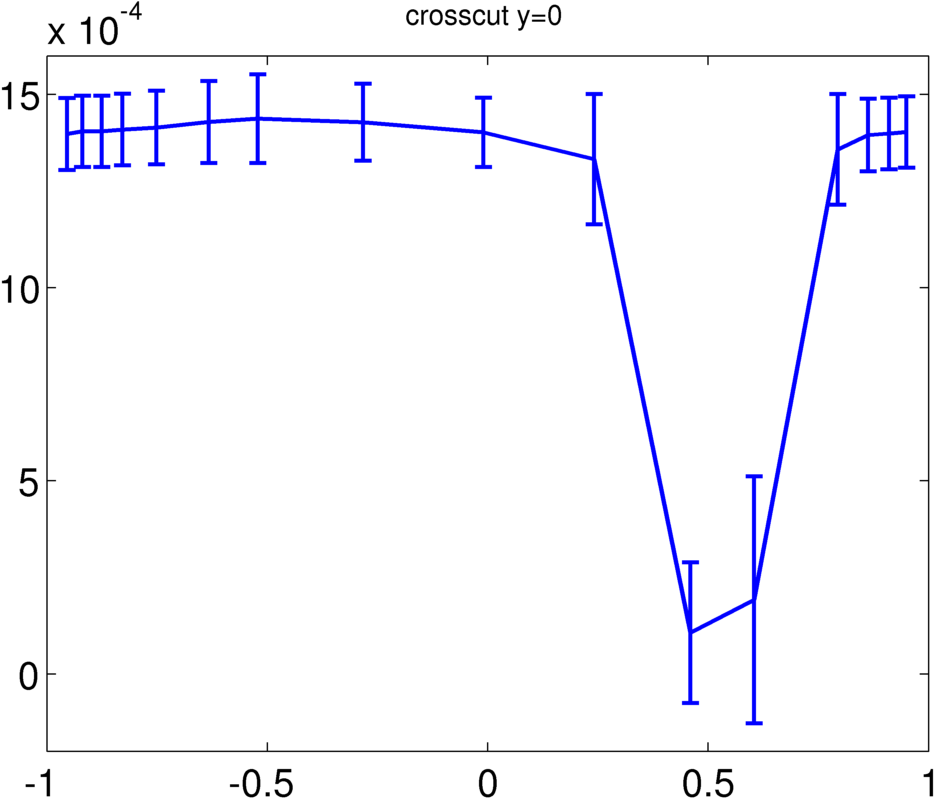}\\
       & EP & MCMC
\end{tabular}
\caption{Numerical results for case 3, two neighboring plastic bars.}
\label{fig:exp6}
\end{figure}

\begin{figure}
\centering
\begin{tabular}{ccc}
\includegraphics[height=3cm]{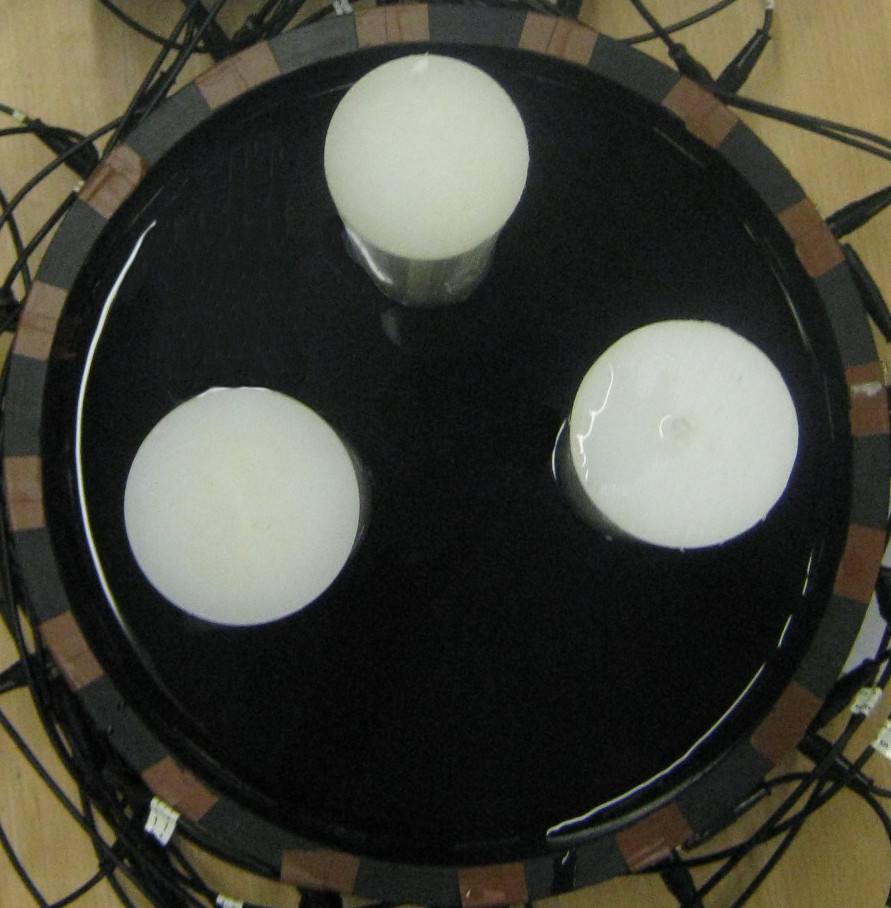} & \includegraphics[height=3cm]{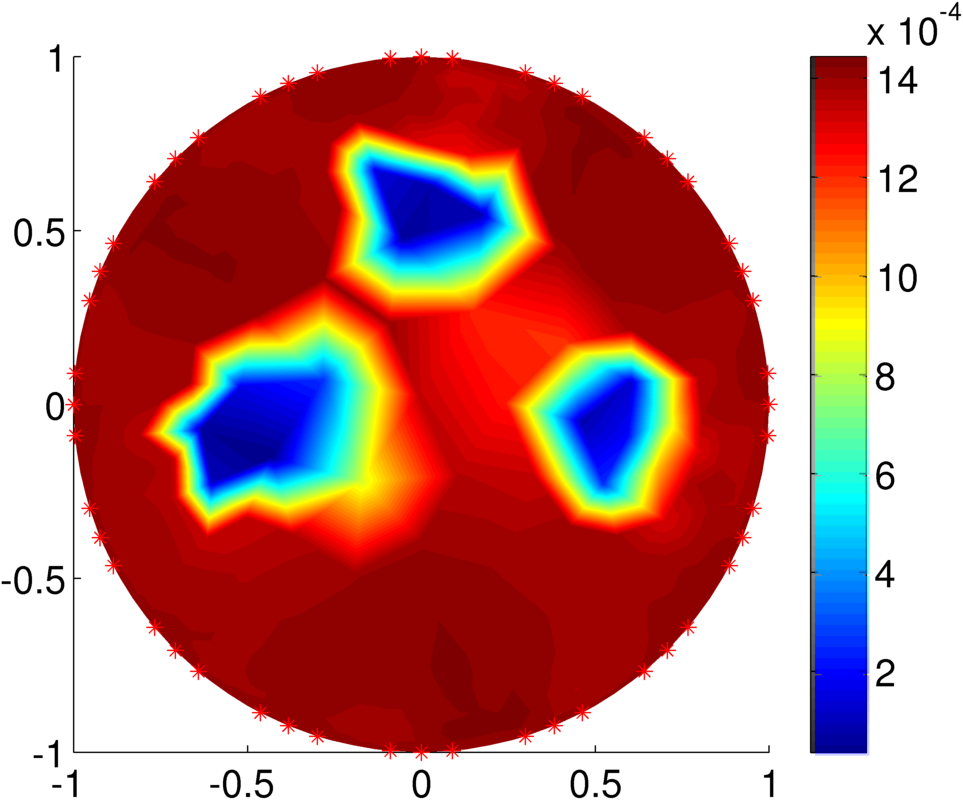} & \includegraphics[height=3cm]{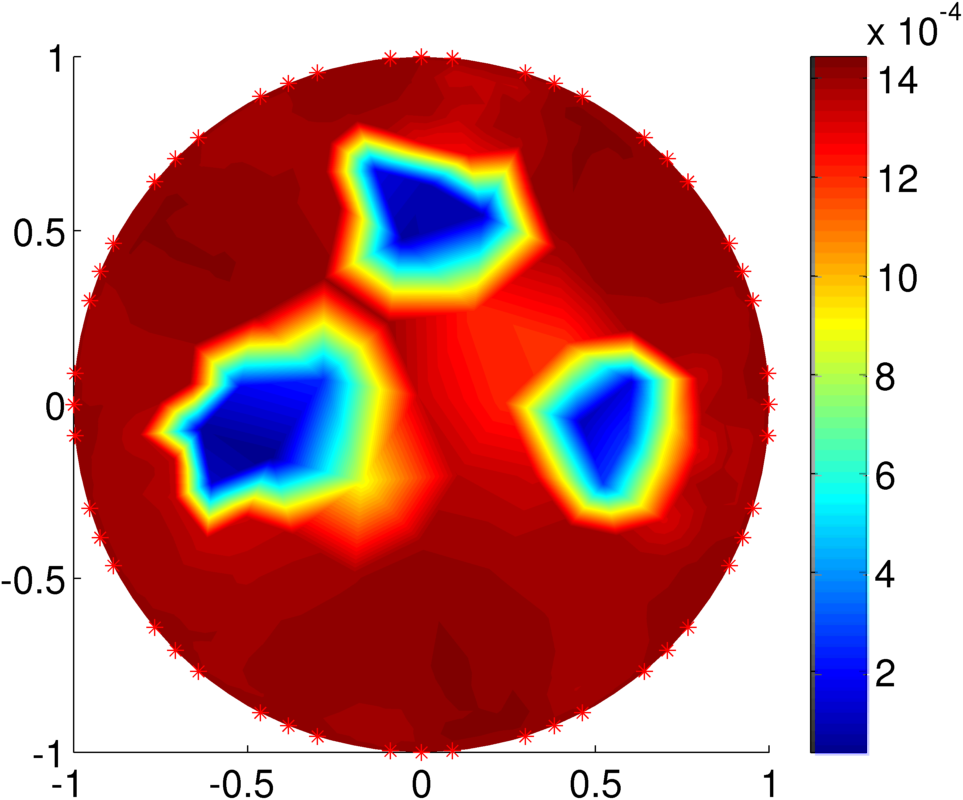}\\
           & \includegraphics[height=3cm]{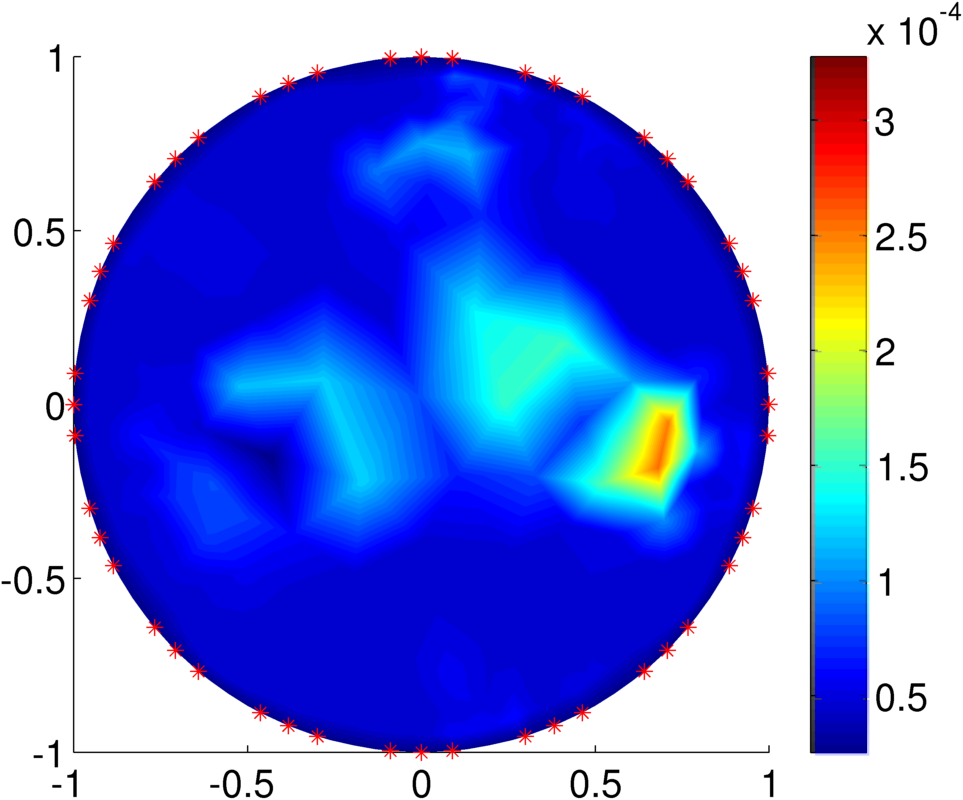} & \includegraphics[height=3cm]{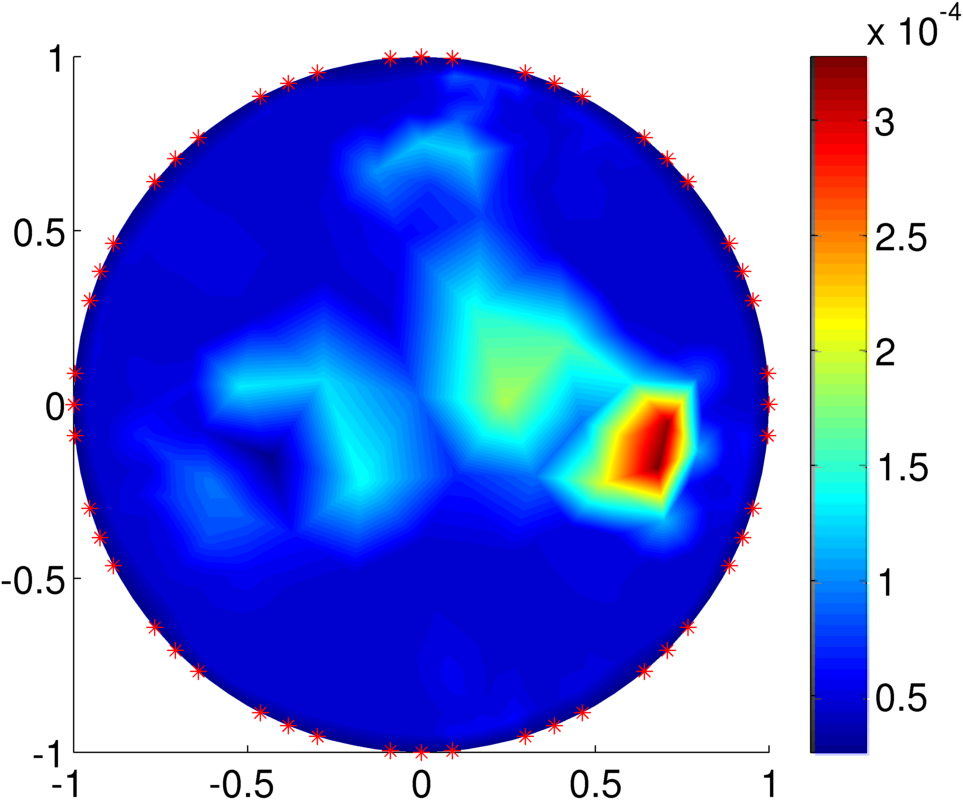}\\
           & \includegraphics[height=3cm]{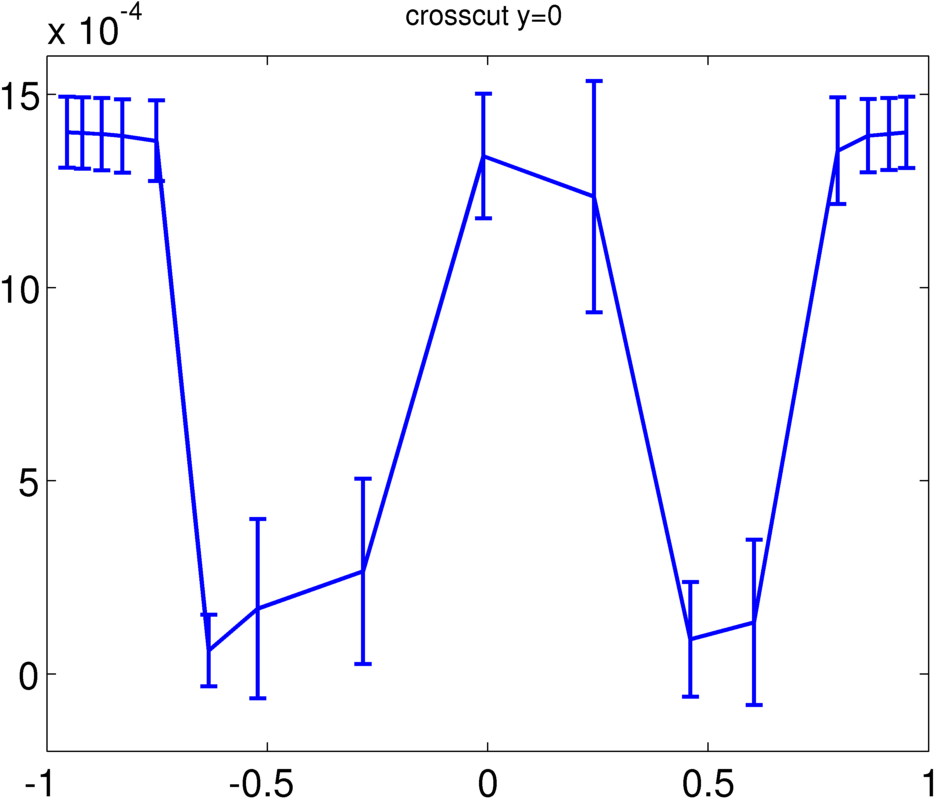} & \includegraphics[height=3cm]{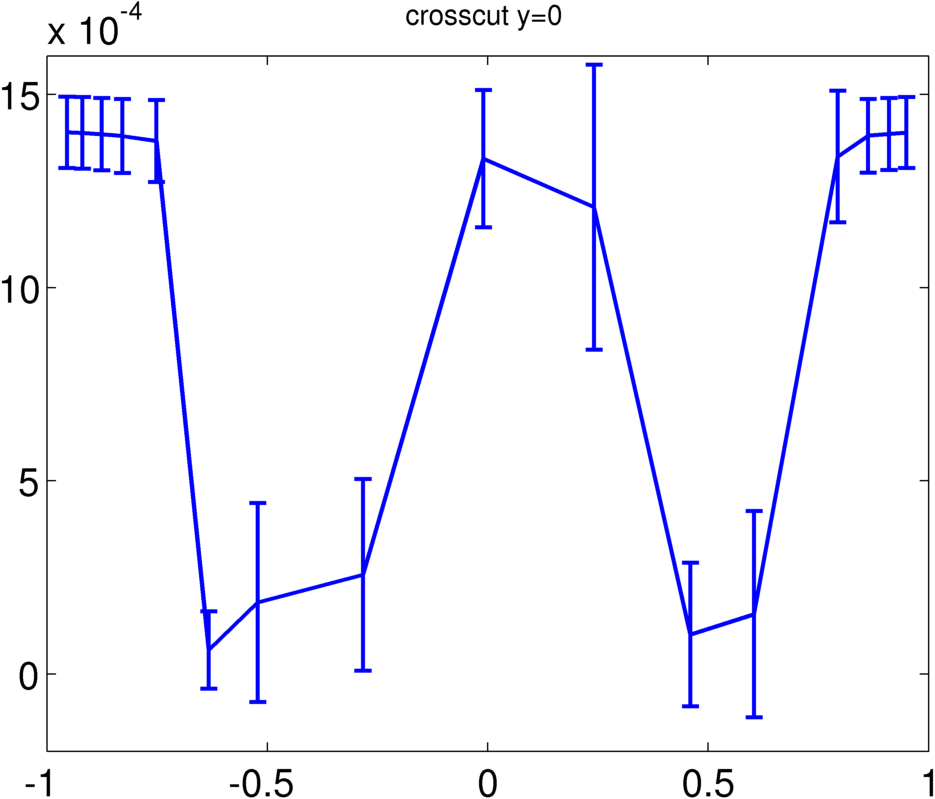}\\
           & EP & MCMC
\end{tabular}
\caption{Numerical results for case 4, three neighboring plastic inclusions.}
\label{fig:exp7}
\end{figure}

\begin{figure}
\centering
\begin{tabular}{ccc}
\includegraphics[height=3cm]{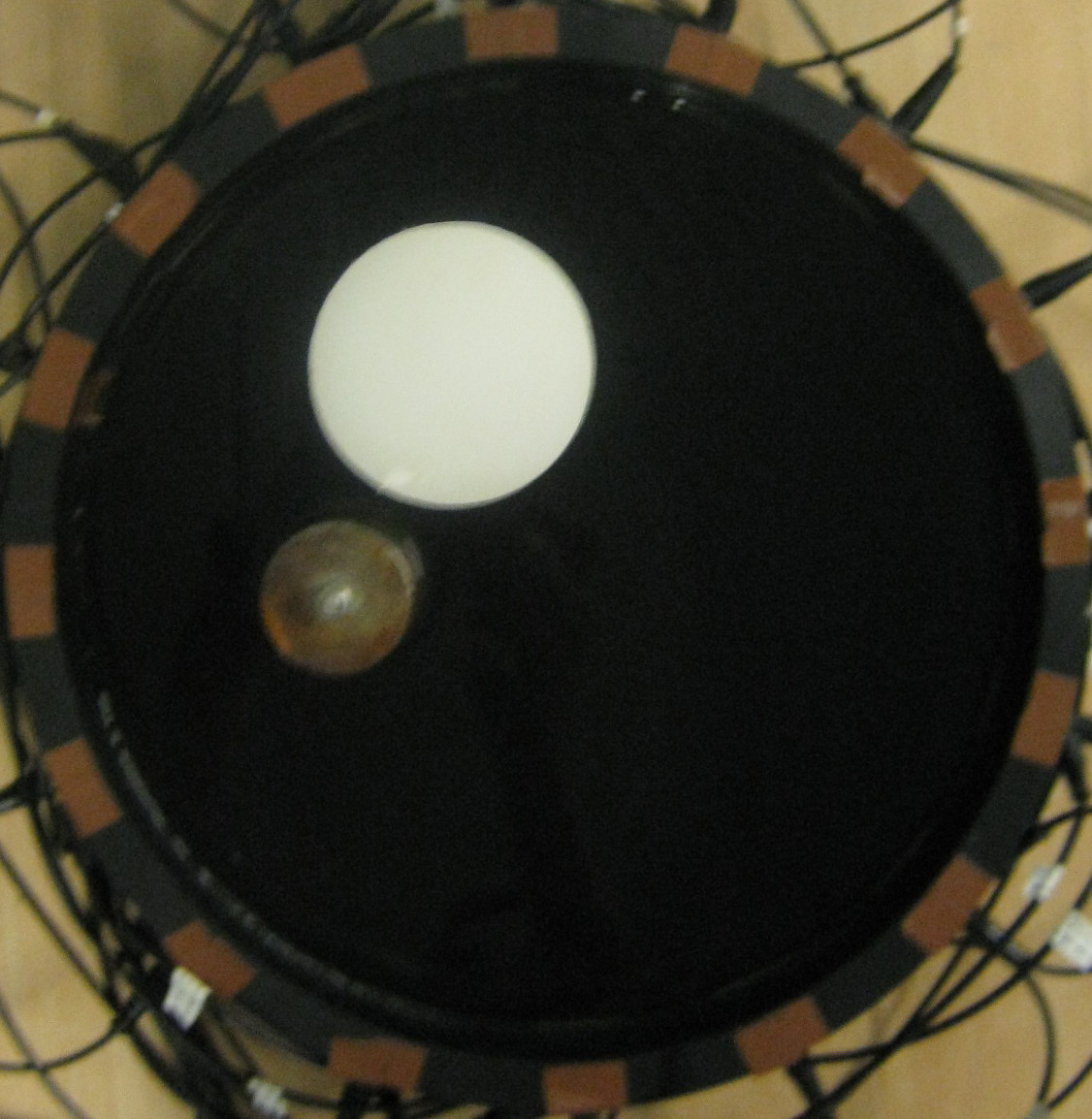}&\includegraphics[height=3cm]{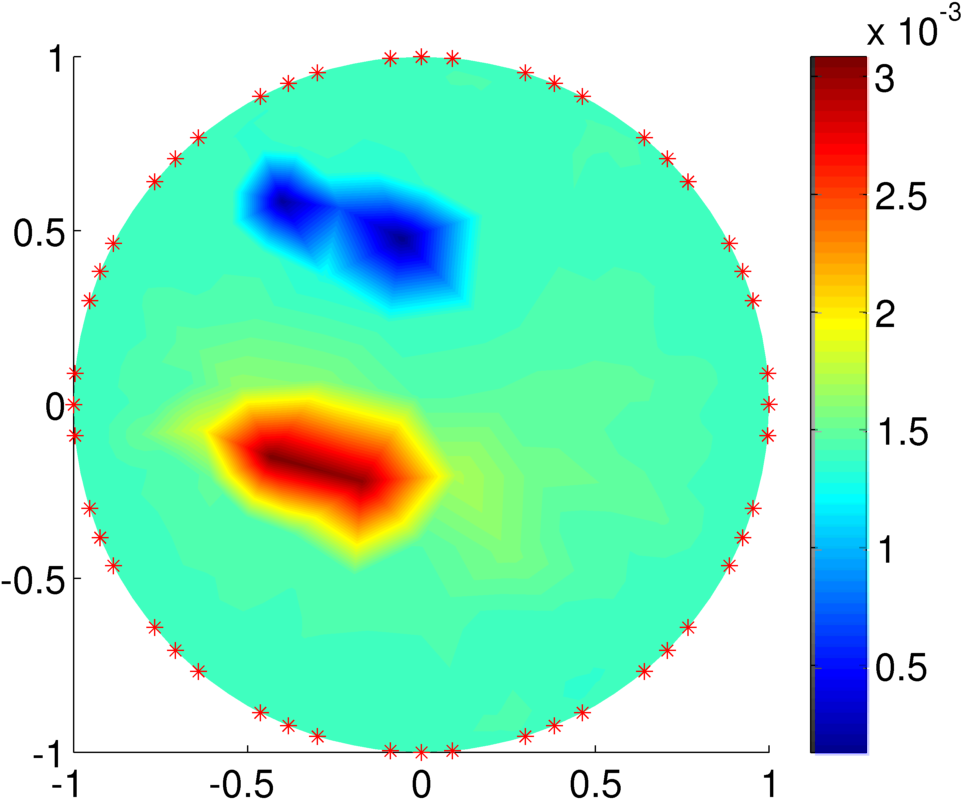} & \includegraphics[height=3cm]{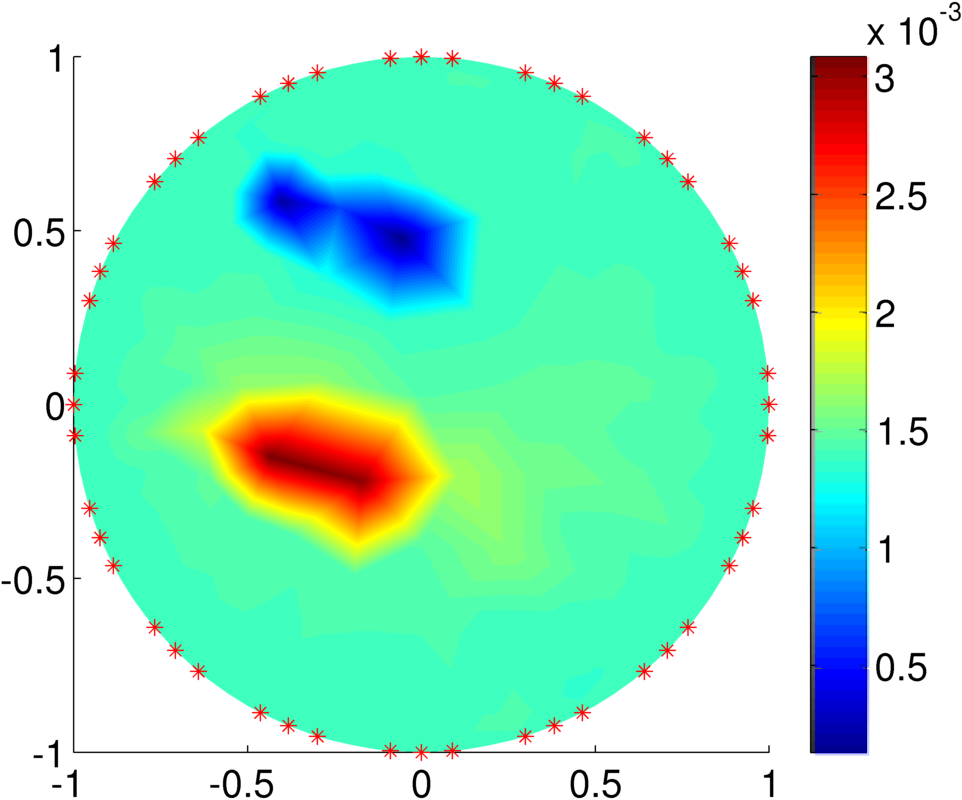}\\
    & \includegraphics[height=3cm]{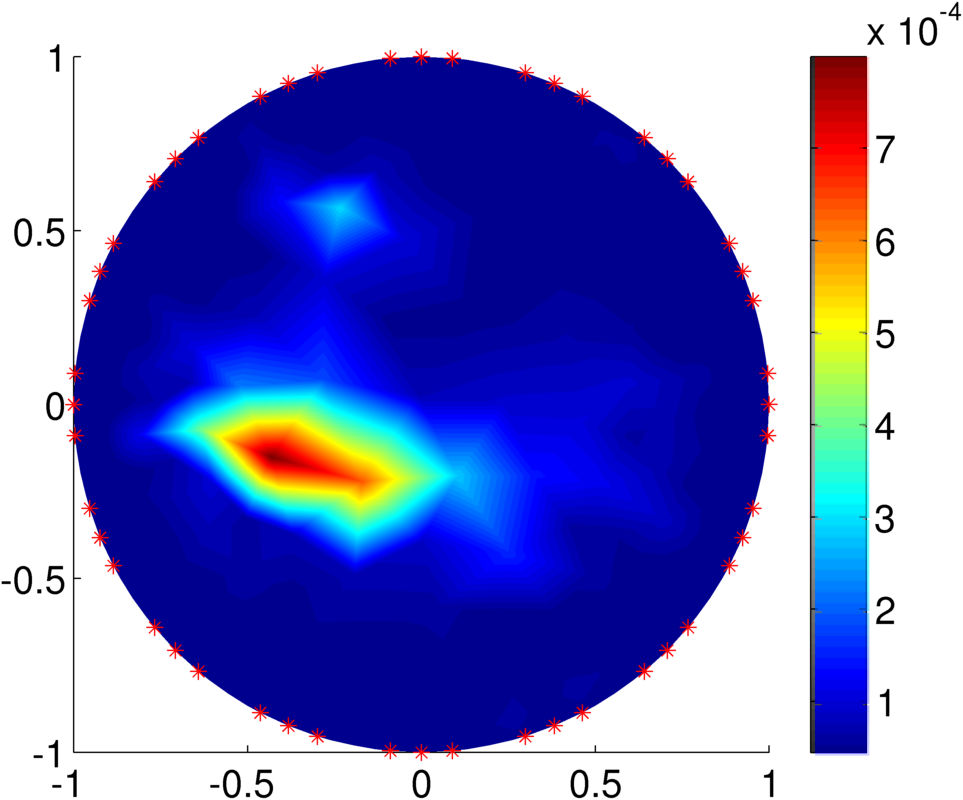} & \includegraphics[height=3cm]{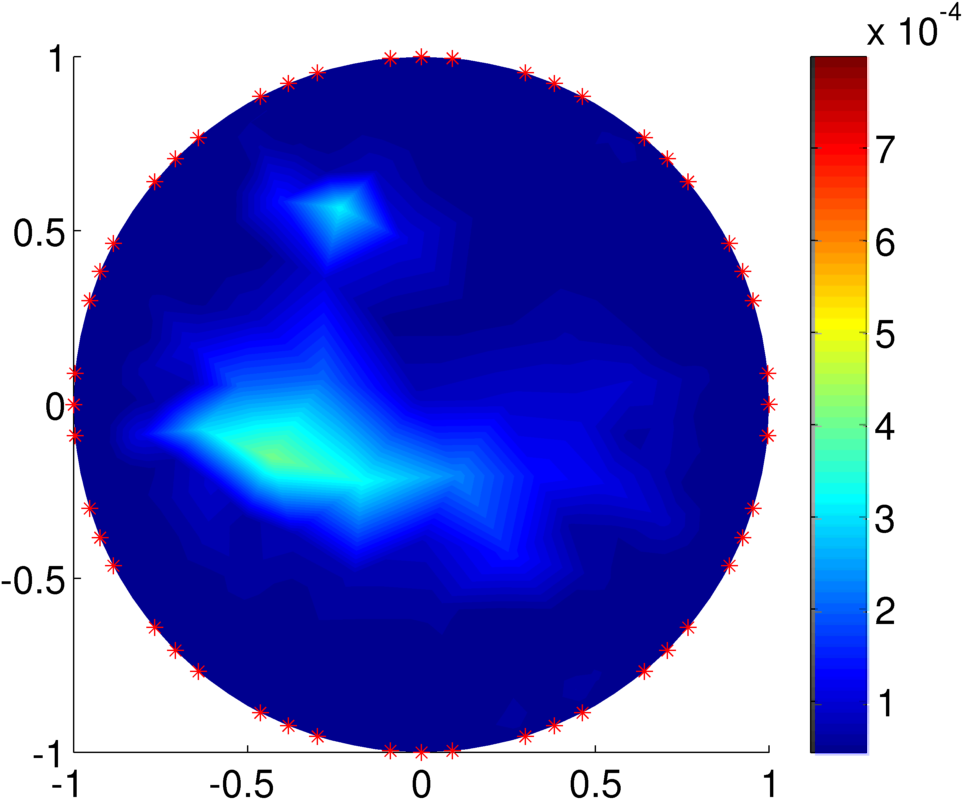}\\
    & \includegraphics[height=3cm]{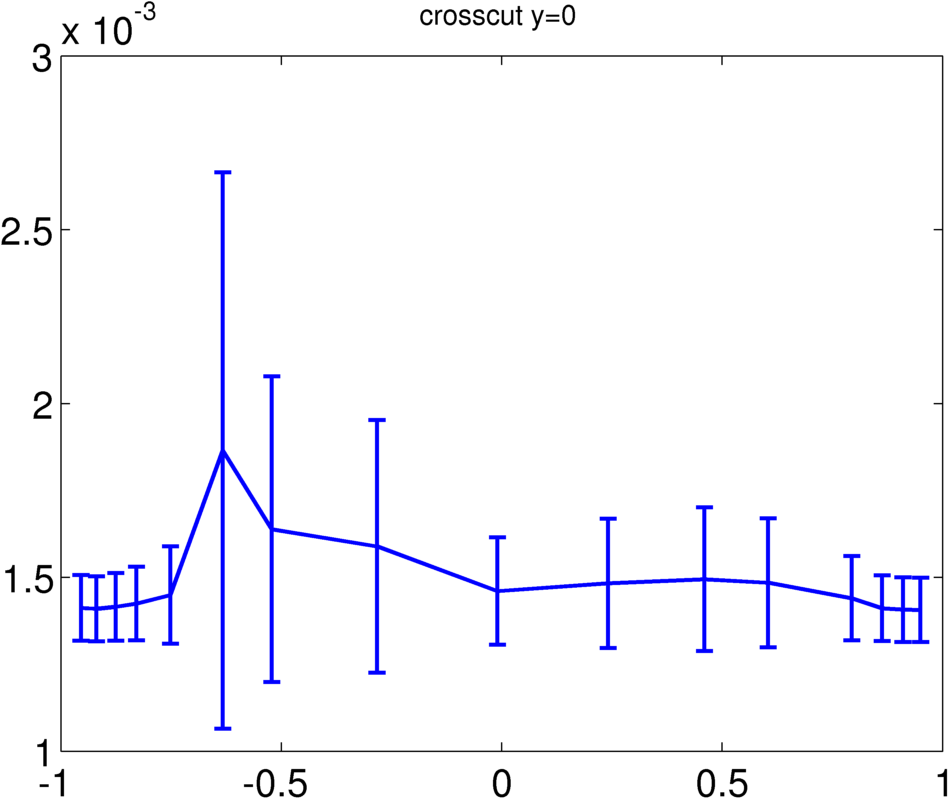} & \includegraphics[height=3cm]{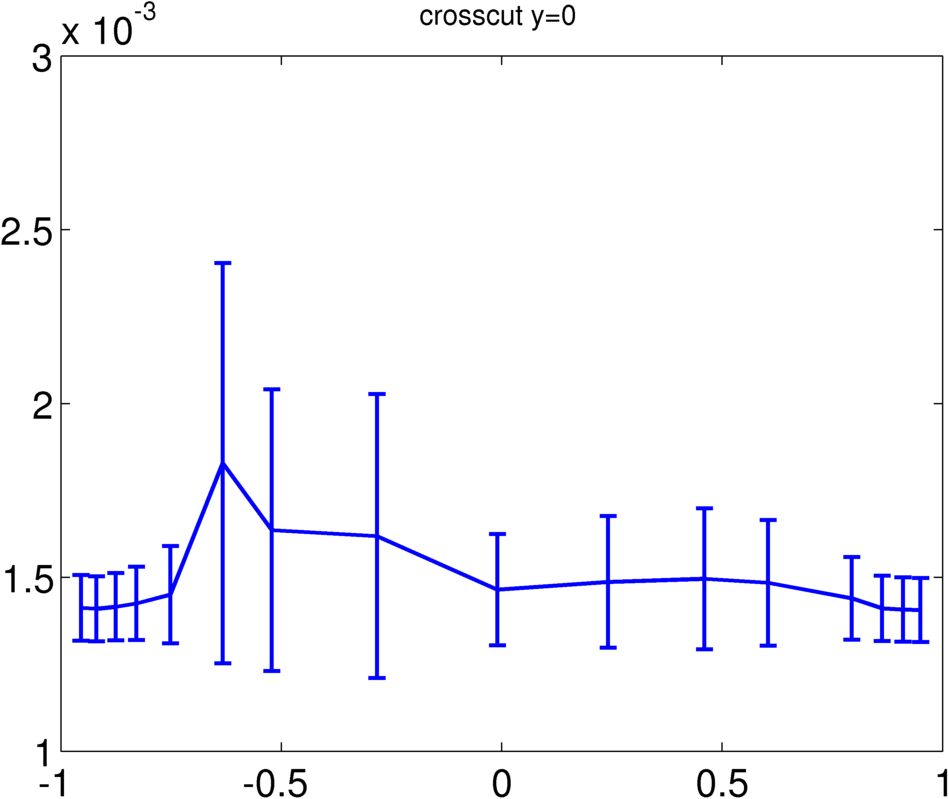}\\
    & EP  & MCMC
\end{tabular}
\caption{Numerical results for case 5, one plastic inclusion and one metallic inclusion.}
\label{fig:exp14}
\end{figure}

\begin{figure}
\centering
\begin{tabular}{ccc}
\includegraphics[height=3cm]{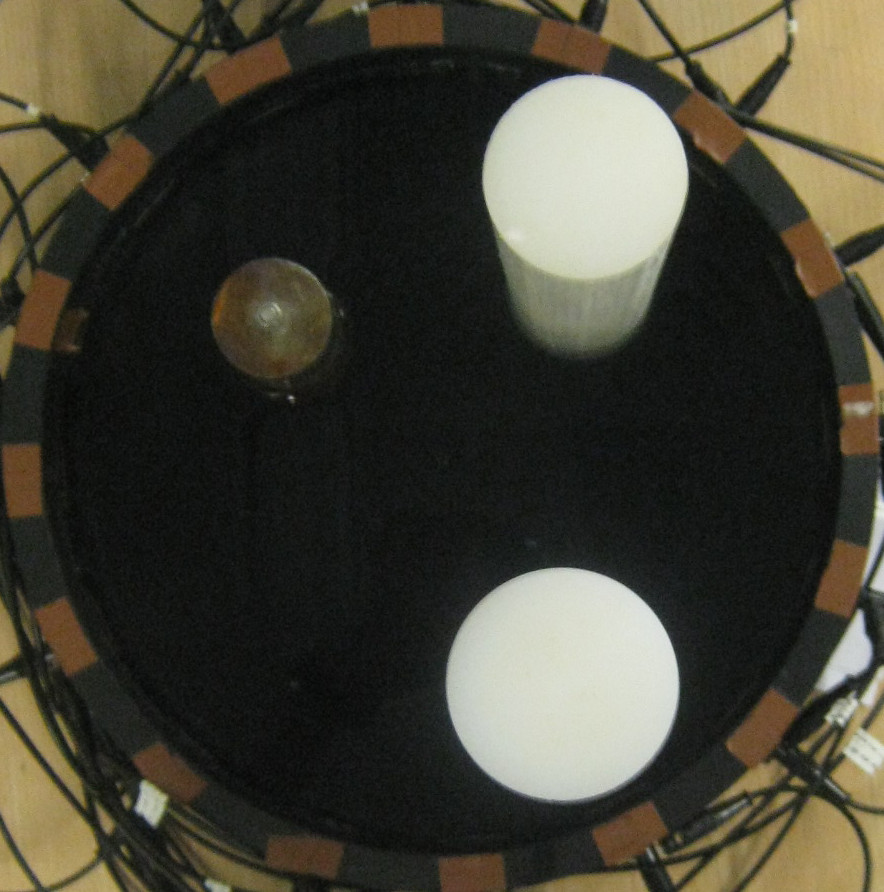}&\includegraphics[height=3cm]{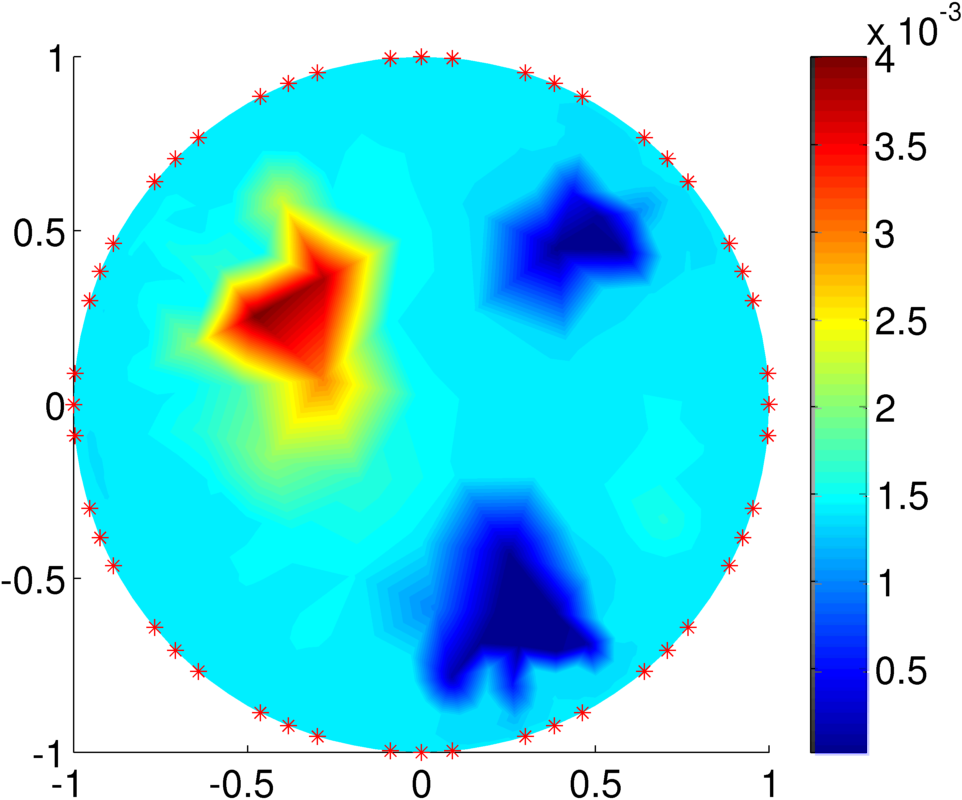} & \includegraphics[height=3cm]{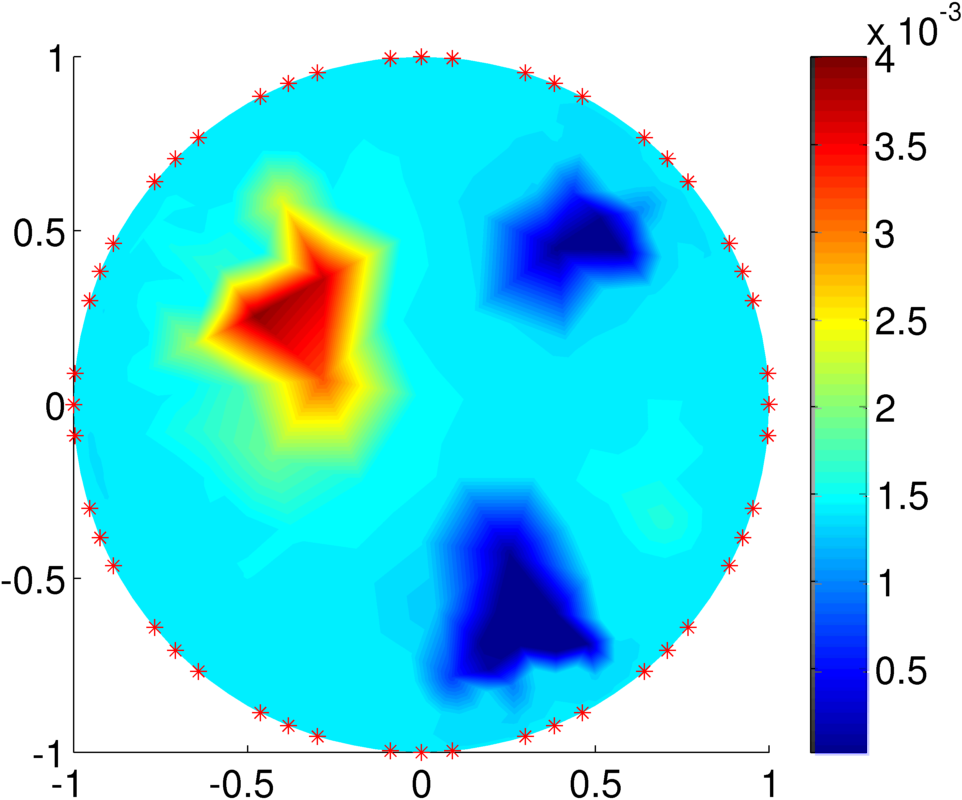}\\
      & \includegraphics[height=3cm]{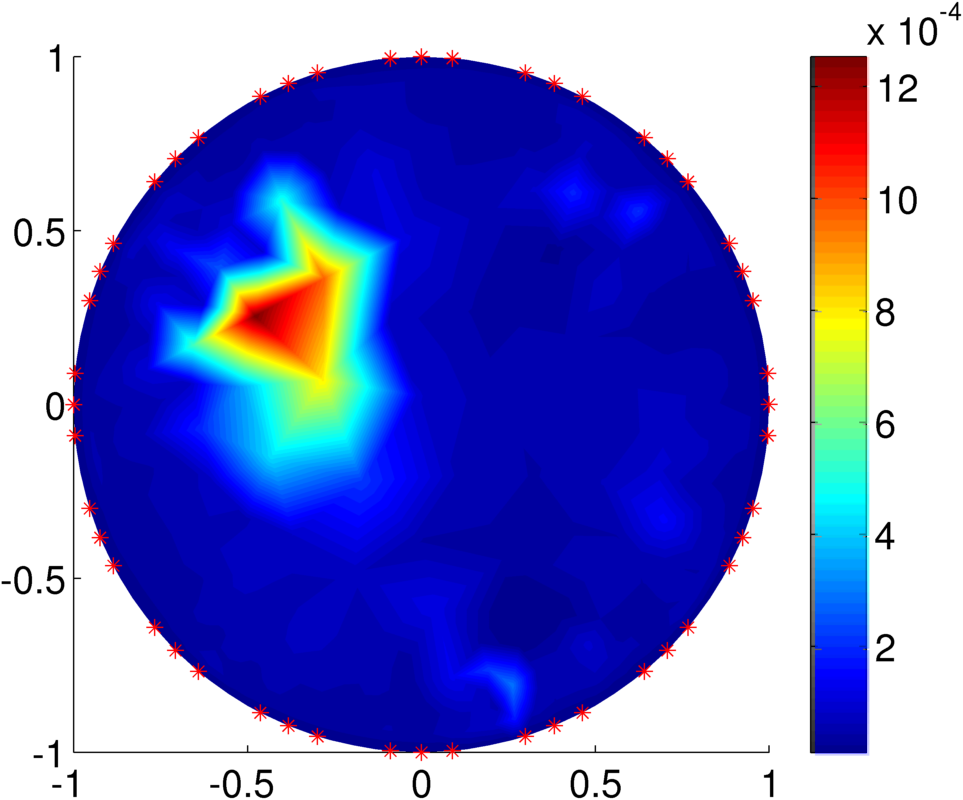} & \includegraphics[height=3cm]{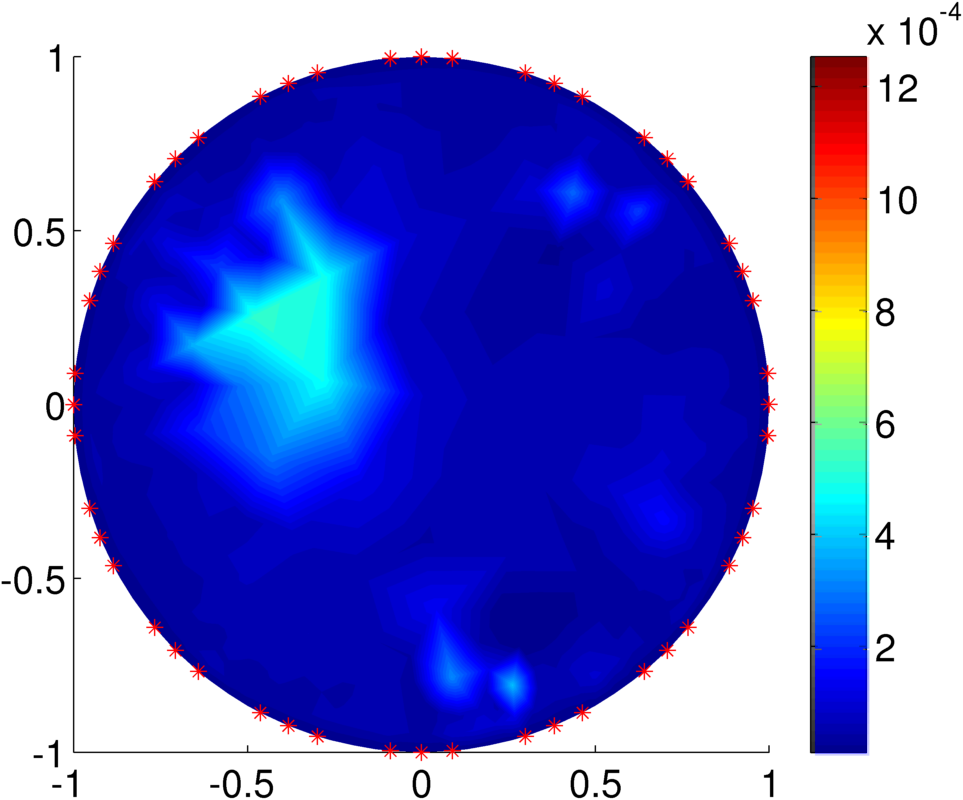}\\
       & \includegraphics[height=3cm]{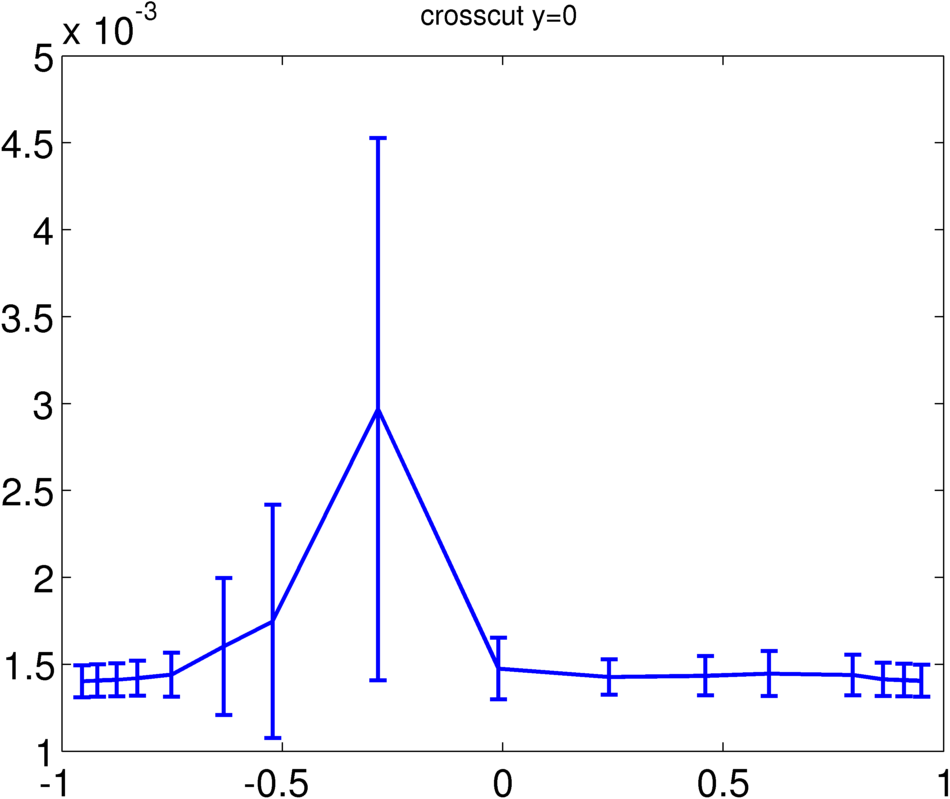} & \includegraphics[height=3cm]{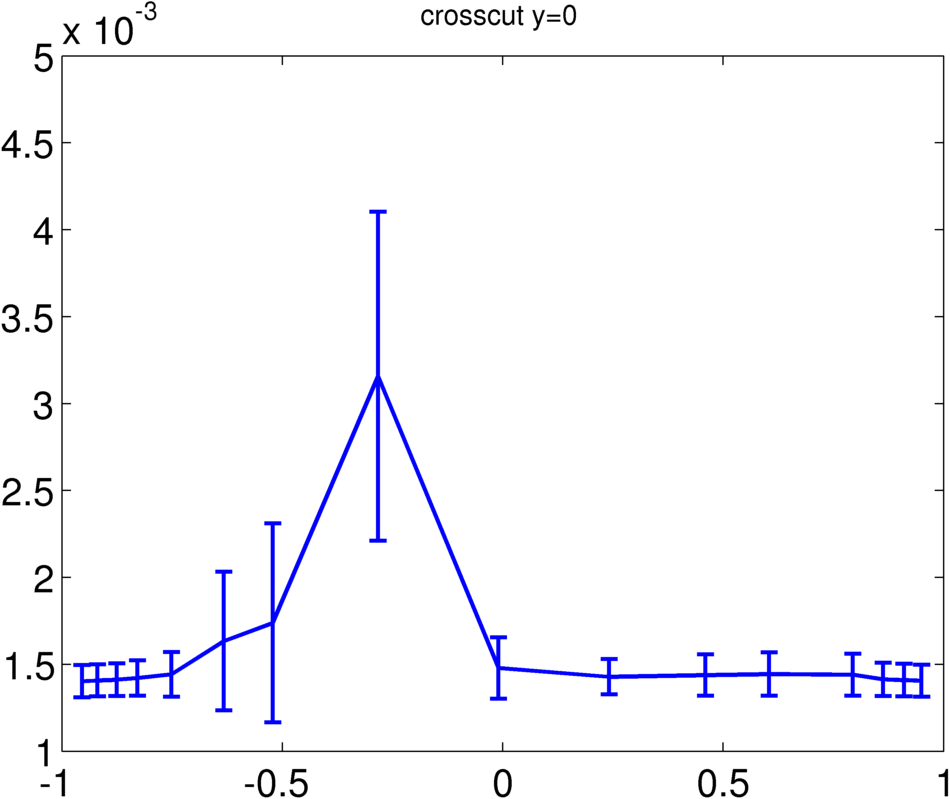}\\
      &  EP & MCMC
\end{tabular}
\caption{Numerical results for case 6, two plastic inclusions and one metallic inclusion.}
\label{fig:exp16}
\end{figure}

\section{Conclusions}
We presented a variational method for approximating the posterior probability distribution based on
expectation propagation. The algorithm is of iterative nature, with each iteration matching the mean and
covariances, by minimizing the Kullback-Leibler divergence. Some basic properties (convergence and
stability) of the algorithm have been discussed. The algorithm is explicitly illustrated on a special
class of posterior distributions which are of projection type. Numerically, we demonstrated the applicability
to electrical impedance tomography with the complete electrode model on real experiments with a water tank
immersed with plastic/metallic bars, under sparsity constraints. The numerical experiments indicate
that compared with Markov chain Monte Carlo methods, the expectation propagation algorithm is computationally
very efficient, and reasonably accurate.

Even though we showcase only the electrical impedance tomography, expectation propagation clearly can be
applied to a wide variety of other nonlinear inverse problems, e.g., optical tomography and inverse
scattering. In addition, there are several avenues for further research. First, it is important to extend
the algorithm to other important prior distributions (e.g., total variation and nonlog-concave
sparsity-promoting priors) and likelihood functions (e.g., Laplace/$t$ distributions for robust
formulations). We note that these extensions are nontrivial for parameter identifications. For example,
for the total variation prior, one complicating factor is the presence of box constraints, whose treatment
is not straightforward; whereas the absence of log concavity poses severe numerical stability issues.
Second, we have described some basic properties
of the EP algorithm. However, many fundamental properties, e.g., its global and local convergence
properties, and stability with respect to data perturbation, remains elusive. Third, the computational
efficiency of the EP algorithm requires careful implementation, e.g., computing the diagonal elements of
the inverse of large (and possibly dense) inverse covariance matrix and accurate numerical quadrature of
low-dimensional but nonsmooth integrals. This necessitates the study of relevant numerical issues, e.g., semi-analytic formulas
and error estimates.

\section*{Acknowledgements}
The authors would like to thank the referees for their helpful comments, and thank Dr. Aku Sepp\"{a}nen
of University of Eastern Finland for kindly providing the experimental data. The work was initiated
during MG's visit at Texas A\&M University, supported partially by
a scholarship by Deutscher Akademischer Austausch Dienst (DAAD) and a visitor programme of
Institute for Applied Mathematics and Computational Science, Texas A\&M University.
The work of BJ was partially supported by NSF Grant DMS-1319052.

\appendix
\section{Proof of Proposition \ref{lem:convergence_blockdiagonal}}\label{sec:app:proof}
\begin{proof}
We recast the posterior $p(x)$ into the formulation \eqref{eq:epwithproj} for EP with projection
\begin{equation*}
  p(x) = \prod t_i(U_i x),
\end{equation*}
where the matrices $U_i$ consist of disjoint rows of the identity matrix $I_n \in \R^{n \times n}$.
Upon relabeling the variables, we may assume $ [U_1^t\, \ldots\, U_k^t] = I_n. $
Then the Gaussian approximations $\{\tt_i\}$ take the form
\begin{equation*}
  \tt_i(x) = G(x,U_i^th_i,U_i^tK_iU_i)\sim e^{-\onehalf x^t U_i^t K_i U_i x + h_iU_ix},
\end{equation*}
where $K_i \in \R^{r\times r}$ and $h_i \in \R^r$, with $r$ being the cardinality
of $x_i$. The Gaussian approximation $\tilde{p}(x)=G(x,h, K)$ to $p(x)$ is given by
the product of factor approximations, i.e., $h= \sum_i U_i^th_i$ and $K = \sum_i
U_i^tK_iU_i$. By the construction of $U_i$, the matrix $K$ is block diagonal with
blocks $K_i$. 
Further, to update the $i$-th factor, steps 6 and 8 of Algorithm \ref{alg:epproj} yield
$\widehat{C}_i^{-1} = (U_iK^{-1}U_i^t)^{-1}-K_i  = 0 $ and hence
\begin{equation*}
   K_i = \mathrm{Var}_{Z^{-1}t_i(s_i)N(s_i;\widehat{\mu}_i,\widehat{C}_i)}[s_i]^{-1} -
   \widehat{C}_i^{-1} = \mathrm{Var}_{Z^{-1}t_i(s_i)}[s_i]^{-1}.
\end{equation*}
Thus the update $K_i$ depends only on $t_i$, but not on the current approximation,
from which the one-sweep convergence follows directly.
\end{proof}

\section{Auxiliary lemmas}\label{sec:app:lemmas}
Below, we recall some results on exponential family \cite{Brown:1986}. We shall
only need their specialization to normal distributions, but
we follow the general framework for two reasons:
it is the usual form found in the literature, and it
allows direct generalization from normal distributions to other exponential families.
\begin{definition}
An exponential family is a set $\mathcal{F}$ of distributions with density of the form
\begin{align*}
  p(x|\theta) &=  e^{\theta^t \phi(x) - \Phi(\theta)},\\
  \Phi(\theta) &= \log \int e^{\theta^t\phi(x)} d\nu(x),
\end{align*}
for natural parameter $\theta$ from the natural parameter space $\Theta$.
The exponential family is fully characterized by the sufficient statistics
$\phi$ and the base measure $d\nu$. The natural parameter space $\Theta$
is a convex set of all natural parameters such that $p(x|\theta)$ is
a valid distribution.
\end{definition}

The log partition function $\Phi$ allows computing the mean and variance of $\phi(x)$:
\begin{align}\label{eq:expfam_cumulants}
	\nabla_\theta \Phi = E_{p(x|\theta)}[\phi(x)]\quad \mbox{and}\quad
	\nabla^2_\theta \Phi = \mathrm{Var}_{p(x|\theta)}[\phi(x)].
\end{align}
Let $f(x)$ be a positive function, and $\log E_{p(x|\theta)} [f(x)] + \Phi(\theta)$
exists for every $\theta\in\Theta$. Then it defines a new exponential family with
the base measure $f(x)d\nu(x)$, with the log partition function $\Phi_f(\theta)$
given by $\log E_{p(x|\theta)} [f(x)] + \Phi(\theta)$.

We now consider the exponential family of normal distributions. We denote a normal distribution
with mean $\mu$ and covariance $C$ by $N(x;\mu,C) = (2\pi)^{-\frac{n}{2}} (\det{C})^{-\frac{1}{2}}
e^{-\frac{1}{2} (x-\mu)^tC^{-1}(x-\mu)}$. Hence, the sufficient statistics $\phi(x)$ and the natural
parameter $\theta$ are given by $\phi(x) = (x, xx^t)$, and $\theta =: (h,K) = (C^{-1}\mu,-\frac{1}{2}
C^{-1})$, respectively. This together with \eqref{eq:expfam_cumulants} and the chain rule yields
\begin{align*}
	E_{N(x;\mu,C)}[x] = C \nabla_\mu \Phi(\theta(\mu,C)) \quad\mbox{and}\quad
	\mathrm{Var}_{N(x;\mu,C)}[x] = C \nabla_\mu^2 \Phi(\theta(\mu,C)) C.
\end{align*}

Now we can state a result on a tilted normal distribution $f(x)N(x,\mu,C)$.
\begin{lem}\label{cor:gaussian:tilted_moments}
Let $N(x;\mu,C)$ be a normal distribution and $f$ a positive function. Then
the mean $\mu_f$ and covariance $C_f$ of the tilted distribution
$Z^{-1}f(x)N(x;\mu,C)$ $($$Z=\int f(x)N(x;\mu,C$$)$ are given by
\begin{align*}
	\mu_f &= C \left(\nabla_\mu\log E_{N(x;\mu,C)} [f(x)]\right) + \mu,\\
	C_f &= C \left(\nabla^2_\mu\log E_{N(x;\mu,C)} [f(x)]\right)C + C.
\end{align*}
\end{lem}
\begin{proof}
Let $\Phi$ and $\Phi_f$ be the log partition function of the normal distribution
$p(x|\theta)=N(x;\mu,C)$ and the tilted distribution $\tilde{p}(x|\theta)=f(x)N(x;\mu,C)e^{-\Phi_f}$,
respectively, with $\theta(\mu,C)=(h,K) = (C^{-1}\mu,-\frac{1}{2}C^{-1})$. Then
$\Phi_f(\theta) = \log E_{p(x|\theta)}[f(x)] + \Phi(\theta)$.
Further, the first component of \eqref{eq:expfam_cumulants} reads
\begin{align*}
	\nabla_h \Phi_f(\theta) = E_{\tilde{p}(x|\theta)}[x]\quad\text{and}\quad
	\nabla^2_h \Phi_f(\theta) = \mathrm{Var}_{\tilde{p}(x|\theta)}[x].
\end{align*}
By the chain rule there hold
\begin{align*}
	\nabla_\mu  \log E_{N(x;\mu,C)} [f(x)] &= C^{-1}\nabla_h \log E_{p(x|\theta)}[f(x)],\\
	\nabla_\mu^2\log E_{N(x;\mu,C)} [f(x)] &= C^{-1} \nabla_h^2\log E_{p(x|\theta)}[f(x)] C^{-1}.
\end{align*}
Consequently, we deduce
\begin{align*}
	E_{\tilde{p}(x|\theta)}[x] &= \nabla_h \Phi_f(\theta) = C\nabla_\mu \log E_{N(x;\mu,C)} [f(x)] + \nabla_h \Phi(\theta),\\
	\mathrm{Var}_{\tilde{p}(x|\theta)}[x] &= \nabla^2_h \Phi_f(\theta) = C \nabla^2_\mu \log E_{N(x;\mu,C)} [f(x)] C+\nabla^2_h \Phi(\theta).
\end{align*}
Now the desired assertion follows directly from the relation $\nabla_h \Phi
(\theta)=\mu$ and $\nabla^2_h \Phi(\theta)=C$ using \eqref{eq:expfam_cumulants}.
\end{proof}

\begin{lem}\label{lem:posdefvariance}
Let $p : \R^n \to \R$ be a probability density function with a support of positive measure.
If the covariance $\mathrm{Var}_p[x] = \int p(x) (E[x]-x)(E[x]-x)^t \dx$ exists, then it
is positive definite.
\end{lem}
\begin{proof}
It suffices to show $w^t \mathrm{Var}_p[x] v > 0$ for any nonzero vector $w\in\R^n$.
Let $S = \{ x \in \supp(p) : w^t (E[x]-x) \neq 0 \}$. However, the complement
set $\mathrm{supp}(p)\setminus S=\{x\in \supp(p): w^t(E[x]-x)=0\}$ lies in a
co-dimensional one hyperplane in $\mathbb{R}^n$, thus it has only zero measure. This together with
the positive measure assumption of the support $\supp(p)$, the set $S$ has positive measure.
Therefore, $w^t \Var_p[x] w \ge \int_S p(x) (w^t(E[x]-x))^2 \dx > 0$.
\end{proof}

Finally we recall the concept of log-concavity.  It plays a role in the theory
of expectation propagation as convexity in classical optimization theory.
A nonnegative function $f:V \to \R^+_0$ is log-concave if
\begin{equation*}
 f(\lambda x_1 + (1-\lambda) x_2) \ge f(x_1)^\lambda f(x_2)^{1-\lambda}
\end{equation*}
holds for all elements $x_1,x_2$ from a real convex vector space $V$ and for all $\lambda \in [0,1]$.

\begin{lem}\label{lem:prod_logconcave}
Let $f,g :V \to \R^+_0$ be log-concave. Then the product $fg$ is log-concave.
\end{lem}

Log-concavity is preserved by marginalization by the Pr\'ekopa-Leindler
inequality \cite[Corollary 1.8.3]{Bogachev:1998} \cite[Corollary 3.5]{BrascampLieb:1976}.
\begin{lem}
\label{lem:marg_logconcave}
Let $f: \R^n\times\R^m \to \R^+_0$ be log-concave and bounded. Then the marginalized function
$ g(x) = \int f(x,y) dy $ is log-concave.
\end{lem}

\section{Proof of Theorem \ref{thm:proj}}\label{sec:app:proof:thmproj}

We begin with an elementary lemma, which shows that the normalizing constant
of the product density is invariant under linear transformation.
\begin{lem}\label{lem:proj_z}
Let $x\sim N(x;\mu,C)$ with $C\in\mathbb{R}^{n\times n}$ being symmetric positive definite and
$U\in \R^{l\times n}\ (l \le n)$ be of full row rank. Then
\begin{equation*}
  \int t(Ux) N(x;\mu,C) \dx = \int t(s) N(s;U\mu,UCU^\mathrm{t}) \ds.
\end{equation*}
\end{lem}
\begin{proof}
Let $T\in\R^{n\times n}$ be a bijective completion of $U$,
i.e., there is some $R \in \mathbb{R}^{(n-l)\times n}$, such that $T=\left[\begin{aligned} U\\
R\end{aligned}\right]$. Further, we denote $Tx =\left(\begin{aligned}s= Ux\\ r=Rx \end{aligned}\right) := \widehat{x}
$. Then we have
\begin{equation*}
  \begin{aligned}
   \tfrac{1}{(2\pi)^{\frac{n}{2}}\sqrt{\det{C}}} e^{-\frac{1}{2} (\mu-x)^\mathrm{t}C^{-1}(\mu-x)}
   &= \tfrac{1}{(2\pi)^{\frac{n}{2}}\sqrt{\det{C}}} e^{-\frac{1}{2} (T\mu-Tx)^\mathrm{t}(TCT^\mathrm{t})^{-1}(T\mu-Tx)}
   =|\det{T}|N(Tx;T\mu,TCT^\mathrm{t}).
  \end{aligned}
\end{equation*}
Consequently,
\begin{align*}
   \int t(Ux) N(x,\mu,C) \dx &= \int t(Ux) |\det{T}| N(Tx;T\mu,TCT^\mathrm{t}) \dx
    = \int t(s) N(\widehat{x};T\mu,TCT^\mathrm{t}) \mathrm{d}\widehat{x}.
\end{align*}
Now we split the Gaussian distribution $N(\widehat{x};T\mu,TCT^\mathrm{t})$ into
\begin{equation}\label{eqn:gaussiansplit}
   N(\widehat{x};T\mu, TCT^\mathrm{t}) = N(s;U\mu,UCU^\mathrm{t}) N(r;\widehat{\mu}(s),\widehat{C}),
\end{equation}
where $\widehat{\mu}(s)=R\mu + (UCR^\mathrm{t})^\mathrm{t} (UCU^\mathrm{t})^{-1}(s-U\mu)$,
and $\widehat{C} = RCR^\mathrm{t} - (UCR^\mathrm{t})^\mathrm{t} (UCU^\mathrm{t})^{-1} (UCR^\mathrm{t})$ is the Schur complement of $UCU^\mathrm{t}$.
Therefore,
\begin{align*}
   \int  t(Ux) N(x;\mu,C) \dx &= \int t(s) N(s;U\mu,UCU^\mathrm{t}) \int N(r;\widehat{\mu}(s),\widehat{C})\mathrm{d}r\ds\\
    &= \int t(s) N(s;U\mu,UCU^\mathrm{t})\ds.
\end{align*}
\end{proof}

Now we present the proof of Theorem \ref{thm:proj}.
\begin{proof}
Let $Z=\int t(Ux) N(x;\mu,C) x \dx$. By Lemma \ref{lem:proj_z}, $Z=\int  t(s) N(s;U\mu,UCU^T) \ds$.
Now using the completion $T$ of $U$ and the splitting \eqref{eqn:gaussiansplit}, we can write
\begin{equation*}
\begin{aligned}
T\int Z^{-1}t(Ux) N(x;\mu,C) x \dx 
  =&\int  Z^{-1}t(s) N(s;U\mu,UCU^\mathrm{t}) \int N(r;\widehat{\mu}(s),\widehat{C}) \widehat{x} \mathrm{d}r \ds\\
  =&\int Z^{-1}t(s) N(s;U\mu,UCU^\mathrm{t}) \left(\begin{array}{c} s\\ \widehat{\mu}(s)\end{array}\right)\ds\\
  =& \left(\begin{array}{c}\displaystyle E_{Z^{-1}t(s) N(s;U\mu,UCU^\mathrm{t})} [s]\\
  R\mu + (UCR^\mathrm{t})^\mathrm{t}(UCU^\mathrm{t})^{-1}(E_{Z^{-1}t(s) N(s;U\mu,UCU^\mathrm{t})} [s]-U\mu)
\end{array}\right).
\end{aligned}
\end{equation*}
The unique solution $\mu^\ast := E_{Z^{-1}t(Ux) N(x;\mu,C)} [x]$ of this matrix equation is given by
\begin{align*}
  \mu^\ast = \mu+ CU^\mathrm{t}(UCU^\mathrm{t})^{-1}(\overline{s}-U\mu),
\end{align*}
where $\overline{s} = E_{Z^{-1}t(s) N(s;U\mu,UCU^\mathrm{t})} [s]$.
This shows the first identity.

We turn to the covariance $C^*= E_{Z^{-1}t(Ux) N(x;\mu,C)} [(x-\mu^*)(x-\mu^*)]$. With a change of variable and
the splitting (with $\bar{s}=U\mu^*$ and $\bar{r}=R\mu^*$)
\begin{equation*}
  (\widehat{x}-T\mu^*)(\widehat{x}-T\mu^*)^\mathrm{t} = \left(\begin{array}{cc} (s-\bar{s})(s-\bar{s})^\mathrm{t} &(s- \bar{s})(r-\bar{r})^\mathrm{t}\\
     (r-\bar{r})(s-\bar{s})^\mathrm{t} &(r-\bar{r})(r-\bar{r})^\mathrm{t} \end{array}\right),
\end{equation*}
we deduce that
\begin{align*}
TC^* T^\mathrm{t}
=& \int Z^{-1}t(s) N(s;U\mu,UCU^\mathrm{t}) \int N(r;\widehat{\mu}(s),\widehat{C}) (\widehat{x}-T\mu^*)(\widehat{x}-T\mu^*)^\mathrm{t}\mathrm{d}r\ds\\
=& \int Z^{-1}t(s) N(s;U\mu,UCU^\mathrm{t})
\left(\begin{array}{cc} (s-\bar{s})(s-\bar{s})^\mathrm{t} & (s-\bar{s})(\widehat{\mu}(s)-\bar{r})^\mathrm{t}\\
			(\widehat{\mu}(s)-\bar{r})(s-\bar{s})^\mathrm{t}	& \widehat{C} +
    (\widehat{\mu}(s)-\bar{r})(\widehat{\mu}(s)-\bar{r})^\mathrm{t}\end{array}\right) \ds.
\end{align*}
Further, it follows from \eqref{eqn:gaussiansplit} that
$
     \widehat{\mu}(s)-\bar{r}= (UCR^\mathrm{t})^\mathrm{t}(UCU^\mathrm{t})^{-1}(s-\bar{s}):=L(s-\bar{s}).
$
Consequently, the covariance $C^* $ satisfies
\begin{align*}
   \left(\begin{array}{cc} UC^*U^\mathrm{t} & UC^*R^\mathrm{t}\\ RC^*U^\mathrm{t} &RC^*R^\mathrm{t}\end{array}\right)
   = \left(\begin{array}{cc} \overline{C} & \overline{C}L^\mathrm{t}\\ L\overline{C} & L\overline{C}L^\mathrm{t} +\widehat{C}\end{array}\right)
\end{align*}
where $\overline{C} = E_{Z^{-1}t(s) N(s;U\mu,UCU^\mathrm{t})} [(s-\bar{s})(s-\bar{s})^\mathrm{t}]$.
The unique solution $C^*$ to the equation is given by
\begin{equation*}
  C^\ast = C + CU^\mathrm{t}(UCU^\mathrm{t})^{-1}(\overline{C}-UCU^\mathrm{t})(UCU^\mathrm{t})^{-1}UC,
\end{equation*}
which can be verified directly termwise. This shows the second identity.
\end{proof}

\bibliographystyle{abbrv}
\bibliography{ep}

\begin{thebibliography}{10}

\bibitem{Arridge:2006}
S.~R. Arridge, J.~P. Kaipio, V.~Kolehmainen, M.~Schweiger, E.~Somersalo,
  T.~Tarvainen, and M.~Vauhkonen.
\newblock Approximation errors and model reduction with an application in
  optical diffusion tomography.
\newblock {\em Inverse Problems}, 22(1):175--195, 2006.

\bibitem{BarzilaiBorwein:1988}
J.~Barzilai and J.~M. Borwein.
\newblock Two-point step size gradient methods.
\newblock {\em IMA J. Numer. Anal.}, 8(1):141--148, 1988.

\bibitem{Beal:2003}
M.~J. Beal.
\newblock {\em Variational {A}lgorithms for {A}pproximate {B}ayesian
  {I}nference}.
\newblock PhD thesis, Gatsby Computational Neuroscience Unit, University
  College London, 2003.

\bibitem{BilionisZabaras:2013}
I.~Bilionis and N.~Zabaras.
\newblock A stochastic optimization approach to coarse-graining using a
  relative-entropy framework.
\newblock {\em J. Chem. Phys.}, 138(4):044313, 12 pp., 2013.

\bibitem{Bogachev:1998}
V.~I. Bogachev.
\newblock {\em Gaussian {M}easures}, volume~62 of {\em Mathematical Surveys and
  Monographs}.
\newblock American Mathematical Society, Providence, RI, 1998.

\bibitem{BrascampLieb:1976}
H.~J. Brascamp and E.~H. Lieb.
\newblock On extensions of the {Brunn-Minkowski} and {Pr\'{e}kopa-Leindler}
  theorems, including inequalities for log concave functions, and with an
  application to the diffusion equation.
\newblock {\em J. Fun. Anal.}, 22(4):366--389, 1976.

\bibitem{BrooksGelman:1998}
S.~P. Brooks and A.~Gelman.
\newblock General methods for monitoring convergence of iterative simulations.
\newblock {\em J. Comput. Graph. Statist.}, 7(4):434--455, 1998.

\bibitem{Brown:1986}
L.~D. Brown.
\newblock {\em Fundamentals of {S}tatistical {E}xponential {F}amilies with
  {A}pplications in {S}tatistical {D}ecision {T}heory}.
\newblock Institute of Mathematical Statistics, Hayward, CA, 1986.

\bibitem{ChaimovichShell:2011}
A.~Chaimovich and M.~S. Shell.
\newblock Coarse-graining errors and numerical optimization using a relative
  entropy famework.
\newblock {\em J. Chem. Phys.}, 134(9):094112, 15 pp., 2011.

\bibitem{ChengIsaacsonNewellGisser:1989}
K.-S. Cheng, D.~Isaacson, J.~C. Newell, and D.~G. Gisser.
\newblock {E}lectrode models for electric current computed tomography.
\newblock {\em IEEE Trans. Biomed. Eng.}, 36(9):918--924, 1989.

\bibitem{ChristenFox:2005}
J.~A. Christen and C.~Fox.
\newblock Markov chain {M}onte {C}arlo using an approximation.
\newblock {\em J. Comput. Graph. Statist.}, 14(4):795--810, 2005.

\bibitem{ClasonJin:2012}
C.~Clason and B.~Jin.
\newblock A semi-smooth {N}ewton method for nonlinear parameter identification
  problems with impulsive noise.
\newblock {\em SIAM J. Imaging Sci.}, 5(2):505--538, 2012.

\bibitem{Dwyer:1967}
P.~S. Dwyer.
\newblock Some applications of matrix derivatives in multivariate analysis.
\newblock {\em J. Amer. Stat. Assoc.}, 62(318):607--625, 1967.

\bibitem{EfendievHouLuo:2006}
Y.~Efendiev, T.~Hou, and W.~Luo.
\newblock Preconditioning {M}arkov chain {M}onte {C}arlo simulations using
  coarse-scale models.
\newblock {\em SIAM J. Sci. Comput.}, 28(2):776--803, 2006.

\bibitem{MoselhyMarzouk}
T.~A. El~Moselhy and Y.~M. Marzouk.
\newblock Bayesian inference with optimal maps.
\newblock {\em J. Comput. Phys.}, 231(23):7815--7850, 2012.

\bibitem{FlathWilcoxAkcelikHill:2011}
H.~P. Flath, L.~C. Wilcox, V.~Ak{\c{c}}elik, J.~Hill, B.~van Bloemen~Waanders,
  and O.~Ghattas.
\newblock Fast algorithms for {B}ayesian uncertainty quantification in
  large-scale linear inverse problems based on low-rank partial {H}essian
  approximations.
\newblock {\em SIAM J. Sci. Comput.}, 33(1):407--432, 2011.

\bibitem{Franklin:1970}
J.~N. Franklin.
\newblock Well-posed stochastic extensions of ill-posed linear problems.
\newblock {\em J. Math. Anal. Appl.}, 31(3):682--716, 1970.

\bibitem{Gao:2008}
J.~Gao.
\newblock Robust {L}1 principal component analysis and its {B}ayesian
  variational inference.
\newblock {\em Neur. Comput.}, 20(2):555--572, 2008.

\bibitem{GehreKluthLipponenJinSeppanenKaipioMaass:2012}
M.~Gehre, T.~Kluth, A.~Lipponen, B.~Jin, A.~Sepp{\"a}nen, J.~P. Kaipio, and
  P.~Maass.
\newblock Sparsity reconstruction in electrical impedance tomography: an
  experimental evaluation.
\newblock {\em J. Comput. Appl. Math.}, 236(8):2126--2136, 2012.

\bibitem{GelmanRobertsGilks:1996}
A.~Gelman, G.~O. Roberts, and W.~R. Gilks.
\newblock Efficient {M}etropolis jumping rules.
\newblock In J.~M. Bernardo, J.~O. Berger, A.~P. Dawid, and A.~F.~M. Smith,
  editors, {\em Bayesian Statistics 5}, pages 599--607. Oxford University
  Press, 1996.

\bibitem{GilksRichardsonSpiegelhalter:1996}
W.~R. Gilks, S.~Richardson, and D.~J. Spielgelhalter.
\newblock {\em {M}arkov {C}hain {M}onte {C}arlo in {P}ractice}.
\newblock Chapman and Hall, London, 1996.

\bibitem{GillGolubMurraySaunders:1974}
P.~E. Gill, G.~H. Golub, W.~Murray, and M.~A. Saunders.
\newblock Methods for modifying matrix factorizations.
\newblock {\em Math. Comput.}, 28(126):505--535, 1974.

\bibitem{GirolamiCalderhead:2011}
M.~Girolami and B.~Calderhead.
\newblock Riemann manifold {L}angevin and {H}amiltonian {M}onte {C}arlo
  methods.
\newblock {\em J. R. Stat. Soc. Ser. B Stat. Methodol.}, 73(2):123--214, 2011.

\bibitem{HigdonLeeBi:2002}
D.~Higdon, H.~Lee, and Z.~Bi.
\newblock A {B}ayesian approach to characterizing uncertainty in inverse
  problems using coarse and fine-scale information.
\newblock {\em IEEE Trans. Sign. Proc.}, 50(2):389--399, 2002.

\bibitem{Jin:2008}
B.~Jin.
\newblock Fast {B}ayesian approach for parameter estimation.
\newblock {\em Int. J. Numer. Methods Engrg.}, 76(2):230--252, 2008.

\bibitem{Jin:2012}
B.~Jin.
\newblock A variational {B}ayesian method to inverse problems with impulsive
  noise.
\newblock {\em J. Comput. Phys.}, 231(2):423--435, 2012.

\bibitem{JinKhanMaass:2012}
B.~Jin, T.~Khan, and P.~Maass.
\newblock A reconstruction algorithm for electrical impedance tomography based
  on sparsity regularization.
\newblock {\em Int. J. Numer. Methods Engrg.}, 89(3):337--353, 2012.

\bibitem{JinMaass:2012}
B.~Jin and P.~Maass.
\newblock An analysis of electrical impedance tomography with applications to
  tikhonov regularization.
\newblock {\em ESAIM: Control. Optim. Cal. Var.}, 18(4):1027--1048, 2012.

\bibitem{JinMaass:2012ip}
B.~Jin and P.~Maass.
\newblock Sparsity regularization for parameter identification problems.
\newblock {\em Inverse Problems}, 28(12):123001, 70 pp., 2012.

\bibitem{JordanGhahramani:1999}
M.~I. Jordan, Z.~Ghahramani, T.~S. Jaakkola, and L.~K.~S. Saul.
\newblock An introduction to variational methods for graphical models.
\newblock {\em Mach. Learn.}, 37(2):183--233, 1999.

\bibitem{Kaczmarz:1937}
S.~Kaczmarz.
\newblock Angen\"{a}herte {A}ufl\"{o}sung von {S}ystemen linearer
  {G}leichungen.
\newblock {\em Bull. Acad. Pol. Sci. Lett. A}, 35:355--357, 1937.

\bibitem{KaipioSomersalo:2005}
J.~Kaipio and E.~Somersalo.
\newblock {\em Statistical and {C}omputational {I}nverse {P}roblems}.
\newblock Springer-Verlag, New York, 2005.

\bibitem{KaipioKolehmainenSomersaloVauhkonen:2000}
J.~P. Kaipio, V.~Kolehmainen, E.~Somersalo, and M.~Vauhkonen.
\newblock Statistical inversion and {M}onte {C}arlo sampling methods in
  electrical impedance tomography.
\newblock {\em Inverse Problems}, 16(5):1487--1522, 2000.

\bibitem{KullbackLeibler:1951}
S.~Kullback and R.~A. Leibler.
\newblock On information and sufficiency.
\newblock {\em Ann. Math. Statistics}, 22(1):79--86, 1951.

\bibitem{Liu:2008}
J.~Liu.
\newblock {\em {Monte Carlo Strategies in Scientific Computing}}.
\newblock Springer, Berlin, 2008.

\bibitem{MartinWilcoxBursteddeGhattas:2012}
J.~Martin, L.~C. Wilcox, C.~Burstedde, and O.~Ghattas.
\newblock A stochastic {N}ewton {MCMC} method for large-scale statistical
  inverse problems with application to seismic inversion.
\newblock {\em SIAM J. Sci. Comput.}, 34(3):A1460--A1487, 2012.

\bibitem{MarzoukNajm:2007}
Y.~M. Marzouk, H.~N. Najm, and L.~A. Rahn.
\newblock Stochastic spectral methods for efficient {B}ayesian solution of
  inverse problems.
\newblock {\em J. Comput. Phys.}, 224(2):560--586, 2007.

\bibitem{Minka:2001}
T.~Minka.
\newblock Expectation propagation for approximate {B}ayesian inference.
\newblock In J.~Breese and D.~Koller, editors, {\em Uncertainty in Artificial
  Intelligence 17}. Morgan Kaufmann, 2001.

\bibitem{Minka:2001a}
T.~Minka.
\newblock {\em A {F}amily of {A}lgorithms for {A}pproximate {B}ayesian
  inference}.
\newblock PhD thesis, Massachusetts Institute of Technology, 2001.

\bibitem{NichollsFox:1998}
G.~K. Nicholls and C.~Fox.
\newblock Prior modelling and posterior sampling in impedance imaging.
\newblock {\em Proc. SPIE}, 3459:116--127, 1998.

\bibitem{NissinenHeikkinenKaipio:2008}
A.~Nissinen, L.~M. Heikkinen, and J.~P. Kaipio.
\newblock The bayesian approximation error approach for electrical impedance
  tomography: experimental results.
\newblock {\em Meas. Sci. Tech.}, 19(1):015501, 9pp, 2008.

\bibitem{OpperWinther:2005}
M.~Opper and O.~Winther.
\newblock Expectation consistent approximate inference.
\newblock {\em J. Mach. Learn. Res.}, 6:2177--2204, 2005.

\bibitem{Savolainen2003}
T.~Savolainen, L.~Heikkinen, M.~Vauhkonen, and J.~Kaipio.
\newblock A modular, adaptive electrical impedance tomography system.
\newblock In {\em Proc. 3rd World Congress on Industrial Process Tomography,
  Banff, Canada}, pages 50--55, 2003.

\bibitem{Seeger:2008}
M.~W. Seeger.
\newblock Bayesian inference and optimal design for the sparse linear model.
\newblock {\em J. Mach. Learn. Res.}, 9:759--813, 2008.

\bibitem{SomersaloCheneyIsaacson:1992}
E.~Somersalo, M.~Cheney, and D.~Isaacson.
\newblock {E}xistence and uniqueness for electrode models for electric current
  computed tomography.
\newblock {\em SIAM J. Appl. Math.}, 52(4):1023--1040, 1992.

\bibitem{StramerTweedie:1999}
O.~Stramer and R.~L. Tweedie.
\newblock Langevin-type models {II}: self-targeting candidates for {MCMC}
  algorithms.
\newblock {\em Methodol. Comput. Appl. Prob.}, 1(3):307--328, 1999.

\bibitem{vanGerven:2010}
M.~A.~J. {van Gerven}, B.~Cseke, F.~P. {de Lange}, and T.~Heskes.
\newblock Efficient {B}ayesian multivariate f{MRI} analysis using a sparsifying
  spatio-temporal prior.
\newblock {\em NeuroImage}, 50(1):150--161, 2010.

\bibitem{VilhunenKaipioVauhkonen:2002}
T.~Vilhunen, J.~P. Kaipio, P.~J. Vauhkonen, T.~Savolainen, and M.~Vauhkonen.
\newblock Simultaneous reconstruction of electrode contact impedances and
  internal electrical properties: {I. Theory}.
\newblock {\em Meas. Sci. Tech.}, 13(12):1848--1854, 2002.

\bibitem{WangZabaras:2005}
J.~Wang and N.~Zabaras.
\newblock Hierarchical {B}ayesian models for inverse problems in heat
  conduction.
\newblock {\em Inverse Problems}, 21(1):183--206, 2005.

\bibitem{WatzenigFox:2009}
D.~Watzenig and C.~Fox.
\newblock A review of statistical modelling and inference for electrical
  capacitance tomography.
\newblock {\em Meas. Sci. Tech.}, 20(5):052002, 22 pp., 2009.

\bibitem{WestAykroydMengWilliams:2004}
R.~West, R.~Aykroyd, S.~Meng, and R.~Williams.
\newblock Markov chain {Monte Carlo} techniques and spatial-temporal modelling
  for medical {EIT}.
\newblock {\em Physiol. Meas.}, 25(1):181--194, 2004.

\bibitem{Woodbury:1950}
M.~A. Woodbury.
\newblock {\em Inverting modified matrices}.
\newblock Statistical Research Group, Memo. Rep. no. 42. Princeton University,
  Princeton, N. J., 1950.

\end{thebibliography}

\end{document}